\documentclass[12pt]{amsart}
\usepackage[margin=1in]{geometry}
\usepackage{cancel}
\usepackage{amssymb}
\usepackage{hyperref}
\usepackage{enumitem}
\usepackage{amsmath}
\usepackage{latexsym}
\usepackage{extarrows}
\usepackage{enumerate}
\usepackage{txfonts}
\usepackage{mathtools}
\usepackage{float}
\usepackage{bm}
\usepackage[normalem]{ulem}
\usepackage{soul}
\usepackage{tikz-cd}
\usepackage{xcolor}
\usepackage{quiver}
\usepackage{comment}
\usepackage{graphics}
\usepackage{float}
\usepackage{relsize}
\usepackage{bm}
\makeatletter
\@namedef{subjclassname@2020}{%
  \textup{2020} Mathematics Subject Classification}
\makeatother
\usepackage[T1]{fontenc}
\usepackage{cleveref}

\hypersetup{
    hidelinks,
    linktoc=page,
    }

\counterwithin{figure}{section}

\theoremstyle{plain}

\newtheorem{theorem}{Theorem}[section]
\newtheorem{proposition}[theorem]{Proposition}

\newtheorem{lemma}[theorem]{Lemma}
\newtheorem{corollary}[theorem]{Corollary}

\theoremstyle{definition}
\newtheorem{definition}[theorem]{Definition}

\newtheorem{remark}[theorem]{Remark}

\newtheorem{example}[theorem]{Example}

\newcommand\E{\mathbb{E}}
\newcommand\Z{\mathbb{Z}}
\newcommand\R{\mathbb{R}}

\newcommand\T{\mathbb{T}}
\newcommand\C{\mathbb{C}}

\newcommand\N{\mathbb{N}}
\newcommand\F{\mathbb{F}}

\newcommand\X{\mathcal{X}}
\newcommand\Y{\mathcal{Y}}

\newcommand\M{\M}

\newcommand\ZZ{\mathrm{Z}}
\newcommand\YY{\mathrm{Y}}
\newcommand\XX{\mathrm{X}}

\newcommand\Hom{\operatorname{Hom}}
\newcommand\Aut{\operatorname{Aut}}

\newcommand\id{\operatorname{id}}
\newcommand\eps{\varepsilon}

\newcommand\Poly{\mathrm{Poly}}

\newcommand{\weight}{\mathrm{wt}}

\renewcommand{\AA}{\mathrm{A}}
\newcommand{\BB}{\mathrm{B}}
\newcommand{\CC}{\mathrm{C}}

\begin{document}


\baselineskip=17pt


\title[Polynomial towers]{Polynomial towers and inverse Gowers theory for bounded-exponent groups}

\author{Asgar Jamneshan}
\address{Institute of Mathematics\\ University of Bonn\\ 53113 Bonn, Germany}
\email{ajamnesh@math.uni-bonn.de}

\author{Or Shalom}
\address{Department of Mathematics\\ Bar Ilan University\\ Ramat Gan\\ 5290002, Israel}
\email{Or.Shalom@math.biu.ac.il}

\author{Terence Tao}
\address{Department of Mathematics\\ University of California\\ Los Angeles\\ CA 90095-1555, USA}
\email{tao@math.ucla.edu}

\date{\today}

\begin{abstract} 
In this paper we develop Host--Kra and inverse Gowers theory for abelian groups of bounded exponent. We show that the Host--Kra factors $\ZZ^{\le k}(\XX)$ associated with actions of such groups admit extensions with the structure of \emph{polynomial towers}. This new notion is a system obtained as a finite iteration of abelian extensions of the trivial system by polynomial cocycles; crucially, the intermediate extensions in this system are not required to agree with the Host--Kra factors. We prove that all such extensions are Abramov (generalizing a recent result of Candela, Gonz\'alez-S\'anchez, and Szegedy), but not necessarily Weyl, and have the structure of $k$-step translational systems. 

Combining this structure theorem with a correspondence principle due to the first and third authors, we derive an inverse theorem for the Gowers norms on finite abelian groups of bounded exponent: large $U^{k+1}$-norm implies large correlation with a polynomial of degree $\le k$ (on the same group), even when the exponent is not square-free or is divisible by small primes. This resolves a conjecture of the first and third authors for such groups, and also answers a question of Candela, Gonz\'alez-S\'anchez, and Szegedy. 
\end{abstract}



\maketitle
\setcounter{tocdepth}{1} 

\tableofcontents

\section{Introduction}
In this paper, we introduce a new type of dynamical system tower, which we call a \emph{polynomial tower} (see Definition \ref{tower-def}). We show that the systems in such towers extend the Host--Kra factors associated with actions of bounded-exponent groups. This ergodic-theoretical result is then used to prove an inverse theorem for the Gowers norms over finite abelian groups of bounded exponent that does not require extending the underlying group. 

For the convenience of the reader and in order to keep the exposition self-contained, we recall the basic notation we will use in the sequel and collect in the appendices the key definitions and technical results from the existing literature.

\subsection{Measure-preserving systems}

Throughout this paper, $\Gamma = (\Gamma,+)$ is understood to be a discrete countably infinite abelian group (and later we will impose the further condition that $\Gamma$ be of bounded exponent).  We begin by recalling the notion of a measure-preserving $\Gamma$-system, or $\Gamma$-system for short.

\begin{definition}[Measure preserving systems]\ 
    \begin{itemize}
        \item A $\Gamma$-system is a quadruple $\XX=(X,\mathcal{X},\mu,T)$ where $(X,\mathcal{X},\mu)$ is a Lebesgue probability space and $T:\Gamma\times \XX\rightarrow \XX$ is a near-action\footnote{We work with near-actions rather than actions due to our identification of cocycles that agree almost everywhere, but the reader is advised to ignore the technical distinction between near-actions and actions on a first reading.} of $\Gamma$ on $\XX$ by measure-preserving transformations. Thus, $T^{\gamma_1+\gamma_2}(x) = T^{\gamma_1}\circ T^{\gamma_2}(x)$ and $T^0(x) = x$ for all $\gamma_1,\gamma_2 \in\Gamma$ and almost all $x \in X$, and $T^\gamma:X\rightarrow X$ is measure-preserving for each $\gamma\in \Gamma$. $\XX$ is called \emph{ergodic} if the only functions $f\in L^2(\XX)$ satisfying $f\circ T^\gamma  = f$ $\mu$-a.e. for all $\gamma\in \Gamma$ are the constants.
        \item We say that one $\Gamma$-system $\YY = (Y,\Y,\nu,S)$ is a \emph{factor} of another $\Gamma$-system $\XX=(X,\X,\mu,T)$ (or equivalently, that $\XX$ is an \emph{extension} of $\YY$) and write 
\begin{equation}\label{extension}
\begin{tikzcd}
    \XX \arrow[d, two heads] \\
    \YY
\end{tikzcd}
\end{equation}
if there is a measure-preserving map $\pi \colon X\rightarrow Y$ such that $\pi\circ T^\gamma(x) = S^\gamma\circ \pi(x)$ for $\mu$-a.e. $x\in X$ and every $\gamma\in \Gamma$. If $\XX$ and $\YY$ are factors of each other (with the factor maps inverting each other outside of a null set) we say that $\XX$ and $\YY$ are \emph{isomorphic} or \emph{equivalent}, and write
$$
\begin{tikzcd}
    \XX \arrow[r, leftrightarrow] & \YY
\end{tikzcd}
$$
    \end{itemize}
\end{definition}

Traditionally, ergodic structure theory has been primarily interested in the case when the group $\Gamma$ is a finitely generated abelian group (see \cite{host2005nonconventional} and \cite{ziegler2007universal} for $\Gamma=\mathbb{Z}$, \cite{gutman-lian} for $\Gamma=\mathbb{Z}^k$ and also \cite{gmv,gmvII,gmvIII} for compactly generated group actions via nilspace theory).  However, the focus of this paper will be on the case when $\Gamma$ is of bounded exponent, that is to say there exists a natural number $m$ such that the subgroup $m\Gamma \coloneqq \{m \gamma: \gamma \in \Gamma\}$ of $\Gamma$ is trivial.  A particularly well-studied example of such a group with bounded exponent in this context is the countable vector space $\F_p^\omega$ over a finite field $\F_p$ of some prime order $p$ (cf. \cite{btz,CGSS}). 

Suppose that $\Gamma$ is of bounded exponent. By Pr\"ufer's first theorem (cf.~\cite[Chapter 5, Theorem 18]{morris}), $\Gamma$ is isomorphic to a direct sum of cyclic groups of order dividing $m$. 
 It will be convenient to ``work in coordinates'' and exploit this direct sum  representation explicitly, so we shall henceforth assume that $\Gamma$ is of the concrete form
\begin{equation}\label{gamma-basis}
\Gamma = \bigoplus_{i=1}^\infty \Z/m_i\Z
\end{equation}
where each $m_i>1$ divides $m$.  In particular, with this basis, the Pontryagin dual $\hat \Gamma \coloneqq \Hom(\Gamma,\T)$ (with $\T = \R/\Z$ the standard unit circle) can be naturally expressed as an infinite product
\begin{equation}\label{gamma-basis-dual}
\hat \Gamma = \prod_{i=1}^\infty \frac{1}{m_i}\Z/\Z.
\end{equation}
This abelian group is compact\footnote{All compact or locally compact abelian groups in this paper will be understood to be metrizable, and additive rather than multiplicative.}, but uncountable.  It will be convenient\footnote{This is because we have chosen to restrict our measure-preserving systems to be Lebesgue spaces, hence countably generated, and similarly restricted our compact abelian groups to be metrizable.  One could in principle work with ``uncountable'' systems and structure groups (in the spirit of \cite{jt19}, \cite{jt20}), in which case one would not need to introduce the non-canonical countable dense subgroup $\tilde \Gamma$, but this would require adapting a significant portion of the literature to the uncountable setting (in particular adopting a more ``point-free'' approach to ergodic theory, based more upon measure algebras and von Neumann algebras than concrete representations of systems), which we have not attempted to accomplish here.} to also work with the countable dense subgroup
\begin{equation}\label{gamma-basis-dense}
\tilde \Gamma = \bigoplus_{i=1}^\infty \frac{1}{m_i}\Z/\Z \leq \hat \Gamma
\end{equation}
consisting of frequencies $\xi \in \hat \Gamma$ with only finitely many non-zero components; this group is isomorphic to $\Gamma$ itself, but is dependent on the specific coordinate system \eqref{gamma-basis} used to define $\Gamma$ and so should not be viewed as a ``canonical'' object attached to $\Gamma$ (in contrast to the Pontryagin dual $\hat \Gamma$, which has excellent functoriality properties).  

In this paper we will make particular use of a certain type of extension \eqref{extension}, namely an \emph{abelian (group) extension}.  We first give some general notation for cocycles, which we will rely heavily on in this paper.

\begin{definition}[Cocycles of abelian group actions]\label{gen-cocycle}  Let $U$ be a locally compact abelian group.  A \emph{$U$-group} is a Polish abelian group $A$ equipped with an action $u \mapsto V_u$ of $U$ on $A$ (that is to say, a continuous homomorphism $V \colon U \to \Aut(A)$ from $U$ to the automorphism group $\Aut(A)$ of $A$, which we equip with the compact-open topology).  We abbreviate $V_u a$ as $ua$ for $u \in U$ and $a \in A$.
\begin{itemize}
    \item[(i)] If $A$ is a $U$-group, we define the difference operators $\partial_{u} \colon A \to A$ as $\partial_{u} a \coloneqq u a - a$.
    \item[(ii)] If $A$ is a $U$-group, we define the invariant subgroup $A^U \leq A$ of $A$ by the formula $A^U \coloneqq \{ a \in A : u a = a\ \forall u \in U \}$.
    \item[(iii)]  If $A$ is a $U$-group, we define $C(U;A)$ to be the $U$-group of tuples $(a_u)_{u \in U}$ with $u \mapsto a_u$ taking values in $A$ depending continuously on $u$, equipped with the compact-open topology. We define the derivative operator $d_U \colon A \to C(U;A)$ by $d_U a \coloneqq (\partial_{u} a)_{u \in U}$, thus $(d_U a)_u = \partial_{u} a$ for all $u \in U$.  We abbreviate $d_U$ as $d$ if the acting group $U$ is clear from context.
    \item[(iv)] A continuous map $T \colon A \to B$ between two $U$-groups $A,B$ is said to be \emph{equivariant} if $T( u a) = u T a$ for all $a \in A$ and $u \in U$.  We write $T^{\oplus U} \colon C(U;A) \to C(U;B)$ for the map $T^{\oplus U} (a_u)_{u \in U} \coloneqq (T a_u)_{u \in U}$.
    \item[(v)]  We let $d_U A \leq C(U;A)$ denote the image of $d_U$, thus we have the short exact sequence
    \begin{equation}\label{b1-short}
    0 \to A^U \to A \to d_U A \to 0
    \end{equation}
   of abelian groups. Elements $d_U a$ of $d_U A$ will be called \emph{$U$-coboundaries}, or simply \emph{coboundaries} when the acting group $U$ is clear from context.  We will informally refer to $a$ as an ``antiderivative'' or ``integral'' of the coboundary $d_U a$.
    \item[(vi)]  An element $a \colon u \to a_u$ of $C(U;A)$ will be called a \emph{$U$-cocycle} (or simply a \emph{cocycle} when the acting group $U$ is clear from context) if one has the \emph{$U$-cocycle equation}
    \begin{equation}\label{cocycle-eq}
    a_{u+u'} = a_u + u a_{u'}
    \end{equation}
    for all $u,u' \in U$; the group of such $U$-cocycles will be denoted $Z^1(U;A)$.  One easily checks that $d_U A \leq Z^1(U;A)$, thus every $U$-coboundary is a $U$-cocycle.  We let $H^1(U;A)$ denote the quotient group $Z^1(U;A)/d_U A$, thus we have the short exact sequence
    \begin{equation}\label{cohom-short}
    0 \to d_U A \to Z^1(U;A) \to H^1(U;A) \to 0
    \end{equation}
    of abelian groups. When $H^1(U;A)$ is trivial (i.e., every $U$-cocycle is a $U$-coboundary), we say that $A$ has \emph{trivial $U$-cohomology}.  Two $U$-cocycles will be said to be \emph{$U$-cohomologous} if they differ by a $U$-coboundary.
\end{itemize}
\end{definition}

Later on we shall also need a weighted version of the above notation, in which the acting group $U$ has a ``weight filtration'' on it, and the group $A$ being acted on also has a ``polynomial filtration'' that is compatible with the weight filtration; see \Cref{weighted-cocycle}.

\begin{example}\label{trivialcohomolgy} If $U,A$ are locally compact abelian groups, we can give $C(U;A)$ the ``regular'' translation $U$-action $u_0 (a_u)_{u \in U} \coloneqq (a_{u+u_0})_{u \in U}$.  This $U$-action has trivial $U$-cohomology. Indeed, if $(b_u)_{u\in U}$ is a $U$-cocycle in $C(U;A)$, setting $u\mapsto a_u\coloneqq b_u(0)$ to be the antiderivative, it follows from the $U$-cocycle property and the translation $U$-action for every $v\in U$ that 
\begin{align*}
    d_v((a_u)_{u\in U}) &= (a_{u+v})_{u\in U} - (a_u)_{u\in U} \\
    &= (b_{u+v}(0))_{u\in U}-(b_u(0))_{u\in U} \\
    &= (b_{u}(0)+b_v(u))_{u\in U}-(b_u(0))_{u\in U}\\
    &=(b_v(u))_{u\in U}.
\end{align*}
\end{example}

\begin{remark} By default, the acting group $U$ will be the discrete group $\Gamma$, but we will frequently need to consider other locally compact acting groups (such as the compact structure groups of a tower, or various subgroups of $\Gamma$) in our arguments.  
\end{remark}

\begin{definition}[Cocycles and skew products]\label{cocycle-def}  Let $\XX=(X,\mathcal{X},\mu,T)$ be a $\Gamma$-system, and $U$ a compact abelian group.
\begin{itemize}
    \item[(i)]  We let $\mathcal{M}(\XX,U)$ be the space of measurable functions $f \colon X \to U$, up to almost everywhere equivalence; this is a $\Gamma$-group with action $T^\gamma f \coloneqq f \circ T^\gamma$ and the topology of convergence in measure.  As per \Cref{gen-cocycle}, this gives notions of derivatives, cocycles, coboundaries, etc. on $\mathcal{M}(\XX,U)$.  Note that if $\XX$ is ergodic then $\mathcal{M}(\XX,U)^\Gamma = U$, as the only $\Gamma$-invariant elements of $\mathcal{M}(\XX,U)$ are the constants.
    \item[(ii)]  We abbreviate $Z^1(\Gamma; \mathcal{M}(\XX,U))$ as $Z^1(\Gamma,\XX,U)$, which is then the space of tuples $\rho = (\rho_\gamma)_{\gamma \in \Gamma}$ with $\rho_\gamma \in \mathcal{M}(\XX,U)$ obeying the \emph{$\Gamma$-cocycle equation}
    \begin{equation}\label{cocycle-equation}
    \rho_{\gamma_1+\gamma_2} = \rho_{\gamma_1} + \gamma_1 \rho_{\gamma_2}
    \end{equation}
    for all $\gamma_1, \gamma_2 \in \Gamma$.  We similarly abbreviate $H^1(\Gamma; \mathcal{M}(\XX,U))$ as $H^1(\Gamma,X,U)$. 
    \item[(iii)]  If $\rho \in Z^1(\Gamma,\XX,U)$ is a $\Gamma$-cocycle, we define the \emph{skew product} $\XX \times_\rho U$ to be the probability space $\XX \times U$ (endowed with the product of $\mu$ and the Haar measure on $U$) with near-action
    $$ T^\gamma (x, u) \coloneqq (T^\gamma x, u + \tilde \rho_\gamma(x))$$
    where for each $\gamma$, we arbitrarily select a measurable representative $\tilde{\rho}_\gamma \colon X \to U$ of  $\rho_\gamma$.  (Thus $\XX \times_\rho U$ is only defined up to equivalence.)  We also depict the relationship between $\XX$ and $\XX \times_\rho U$ by the diagram
\begin{equation}\label{abelian-extension}
\begin{tikzcd}
    \XX \times_\rho U \arrow[d, Rightarrow, "\rho; U"] \\ 
    \XX   
\end{tikzcd}
\end{equation}
which one can view as a special case of \eqref{extension}.  
   \item[(iv)] We set the range group $U$ to be $\T$ by default unless otherwise specified, thus we abbreviate $\mathcal{M}(\XX)=\mathcal{M}(\XX,\T)$, $Z^1(\Gamma,\XX)=Z^1(\Gamma,\XX,\T)$, etc..
\end{itemize}
\end{definition}

We caution that a skew product $\XX\times_\rho U$ of an ergodic system $\XX$ need not be ergodic; by Mackey--Zimmer theory this failure of ergodicity occurs if and only if $\rho$ is $\Gamma$-cohomologous to a cocycle taking values in a proper closed subgroup of $U$ (see \Cref{zimmer}).

\begin{example}
    If $\rho \colon \Gamma \to U$ is a homomorphism into a compact abelian group $U$ (or equivalently, a $\Gamma$-cocycle using the trivial action on $U = \mathcal{M}(\mathrm{pt},U)$), then we have the extension
\begin{equation}\label{pt}
\begin{tikzcd}
    U \arrow[d, Rightarrow, "\rho; U"] \\ 
    \mathrm{pt}
\end{tikzcd}
\end{equation}
where $\mathrm{pt}$ is the trivial system consisting of a single point, and $U$ is the rotational system with $T^\gamma \colon U \to U$ given by $T^\gamma x \coloneqq x + \rho(\gamma)$.  
\end{example}

\begin{remark} As is well known, two $\Gamma$-cocycles $\rho, \rho'$ which are $\Gamma$-cohomologous will give rise to equivalent $\Gamma$-systems $\XX \times_\rho U$, thus one often has the freedom to add or subtract a $\Gamma$-coboundary to a given $\Gamma$-cocycle $\rho$.  In our arguments, though, it will frequently be important to select ``good'' representation of a cocycle that obeys some additional properties, such as being a polynomial of a given degree, which are not necessarily preserved by the addition or subtraction of a coboundary.
\end{remark}

We will be particularly concerned here with \emph{towers} 
of abelian group extensions for various cocycles $\rho_i \colon \Gamma \times X_{i-1} \to U_i$ taking values in various compact abelian groups $U_i$ (which we call the \emph{structure groups} of the tower), thus $\XX_i = \XX_{i-1} \times_{\rho_i} U_i$, or more compactly
$$ \XX_j = \mathrm{pt} \times_{\rho_1} U_1 \times_{\rho_2} U_2 \dots \times_{\rho_j} U_j$$
where the skew product should be applied from left to right; see \Cref{tower-gen}.   Abelian group extension towers are somewhat analogous to the notion of a solvable series in group theory. We refer to $j$ as the \emph{height} of the tower.

\begin{figure}
    \centering
\begin{tikzcd}
    \XX_j \arrow[d, Rightarrow, "\rho_{j}; U_j"] \\ 
    \XX_{j-1} \arrow[d, Rightarrow, "\rho_{j-1}; U_{j-1}"]  \\
    \vdots \\
    \XX_1 \arrow[d, Rightarrow, "\rho_1; U_1"] \\
    \XX_0 = \mathrm{pt}
\end{tikzcd}
    \caption{A tower of height $j$.}
    \label{tower-gen}
\end{figure}

Every ergodic $\Gamma$-system $\XX$ gives rise to a canonical tower of a given height $k \geq 1$, namely the \emph{Host--Kra tower}
of the \emph{Host--Kra factors}\footnote{The factor $\ZZ^{\leq k}(\XX)$ is sometimes also denoted $\ZZ^{<k+1}(\XX)$ in the literature.  The Kronecker factor $\ZZ^{\leq 1}(\XX)$ should not be confused with the space of cocycles $Z^1(\Gamma,\XX,U)$; we hope that this unfortunate collision of notation does not cause confusion.} $\ZZ^{\leq k}(\XX)$.  We review the formal definitions of these factors in Appendix \ref{host-kra-theory}, including the construction of the cubic measure $\mu^{[k]}$, the $\Gamma$-cocycles $\rho_k$ and structure groups $U_k$ and various related notions and facts. We remark that while the factors $\ZZ^{\leq j}(\XX)$ are canonically defined, and the structure groups $U_j$ are also unique up to isomorphism, the $\Gamma$-cocycles $\rho_j$ are only unique up to coboundaries (even after fixing the structure groups $U_j$), as per the usual Mackey--Zimmer theory of group extensions.  The smallest factor $\ZZ^{\leq 0}(\XX)$ is the invariant factor of $\XX$, which is trivial as we are assuming ergodicity. The factor $\ZZ^{\leq 1}(\XX)$ is known as the \emph{Kronecker factor}, and is generated by the eigenfunctions of $\XX$; the factor $\ZZ^{\leq 2}(\XX)$ is also known as the \emph{Conze--Lesigne factor} (cf. \cite{cl1,cl2,cl3,jst}). Roughly speaking, the $k^\mathrm{th}$ Host--Kra factor $\ZZ^{\leq k}(\XX)$ controls the distribution of $k+1$-dimensional cubes in $\XX$, and plays an important role in understanding other patterns in $\XX$, such as $k+2$-term arithmetic progressions; intuitively speaking, functions measurable in $\ZZ^{\leq k}(\XX)$ should be thought of as having ``generalized degree $\leq k$'' in some sense.  

The $\Gamma$-cocycles arising in the Host--Kra tower have the important additional property of being \emph{type\footnote{This concept was denoted ``type $<j$'' in \cite{btz}, and ``type $j$'' in most other literature.} $\leq j$}; see \Cref{type:def}.  For any ergodic $\Gamma$-system $\XX$ and any compact abelian group $U$, the collections $Z^1_{\leq j}(\Gamma,\XX,U)$ of type $\leq j$ $\Gamma$-cocycles in $Z^1(\Gamma,\XX,U)$ form a nested chain of subgroups
\begin{equation}\label{type-chain}
d_\Gamma \mathcal{M}(\XX,U) = Z^1_{\leq 0}(\Gamma,\XX,U) \leq Z^1_{\leq 1}(\Gamma,\XX,U) \leq Z^1_{\leq 2}(\Gamma,\XX,U) \leq \dots \leq Z^1(\Gamma,\XX,U)
\end{equation}
which will play an important role in our analysis.

We say that a system $\XX$ is of order\footnote{The symbol $\leq$ is often omitted in the literature.} $\leq k$ if $\XX$ is isomorphic to $\ZZ^{\leq k}(\XX)$, such that $\XX$ itself appears at the top of the height $k$ Host--Kra tower.  The basic properties of systems of order $\leq k$ are recalled in \Cref{prop-functoriality}.  For instance, if a system is of order $\leq k$, it is of order $\leq k'$ for any $k' \geq k$. However, as we shall see in this paper, it will be important to consider alternate representations of $\XX$ by towers, even though this will complicate the type structure of each level of the tower.

The Host--Kra tower interacts (in a somewhat complicated fashion) with the concept of an \emph{Abramov system}.

\begin{definition}[Polynomials]
    Let $\Gamma$ be a countable abelian group, let $\XX=(X,\X,\mu,T)$ be an ergodic $\Gamma$-system, $U$ be a compact abelian group, and let $k$ be an integer. 
    \begin{itemize}
    \item[(i)]  We let $\Poly_{\leq k}(\XX,U)$ be the subgroup of $\mathcal{M}(\XX,U)$ consisting of $f$ which are \emph{polynomials of degree at most $k$} in the sense that $\partial_{\gamma_1}\dots \partial_{\gamma_{k+1}} f = 0$ for all $\gamma_1,\dots,\gamma_{k+1}\in \Gamma$; thus for instance $\Poly_{\leq 0}(\XX,U) = U$ (by ergodicity) and $\Poly_{\leq -1}(\XX,U) = \{0\}$. More generally we adopt the convention $\Poly_{\leq k}(\XX,U) = \{0\}$ if $k \leq -1$, including if $k = -\infty$.   We similarly let $\Poly_{\leq k}^1(\Gamma, \XX, U)$ be the subgroup of $Z^1(\Gamma,\XX,U)$ consisting of cocycles $(\rho_\gamma)_{\gamma \in \Gamma}$ such that each $\rho_\gamma$ is a polynomial of degree $\leq k$.   We also define $\Poly_{\leq k}(\Gamma,U)$ and $\Poly^1_{\leq k}(\Gamma,\Gamma,U)$ in the same fashion, using the translation action of $\Gamma$ on itself. As before, we abbreviate $\Poly_{\leq k}(\XX) = \Poly_{\leq k}(\XX,\T)$, $\Poly_{\leq k}^1(\Gamma,\XX) = \Poly_{\leq k}^1(\Gamma,\XX,\T)$, and $\Poly_{\leq k}(\Gamma) = \Poly_{\leq k}(\Gamma,\T)$. 
    \item[(ii)] We say that $\XX$ is \emph{Abramov of order $\leq k$} if it is generated (as a measure algebra\footnote{That is to say, the $\sigma$-algebra of $\XX$, quotiened out by null sets, is generated by the $\sigma$-algebra associated to (arbitrary measurable representatives) of the polynomials in $\Poly_{\leq k}(\XX)$.)}) by $\Poly_{\leq k}(\XX)$.  In particular if a system is Abramov of order $\leq k$, it is also Abramov of order $\leq k'$ for any $k' \geq k$.
\end{itemize}
\end{definition}

We have nested sequences of abelian groups
\begin{equation}\label{poly-chain}
\begin{split}
0 = \Poly_{\leq -1}(\XX,U) &\leq \Poly_{\leq 0}(\XX,U) = U  \\
&\leq \Poly_{\leq 1}(\XX,U) \\
&\leq \dots \\
&\leq \mathcal{M}(\XX,U)
\end{split}
\end{equation}
and
\begin{align*}
0 = \Poly^1_{\leq -1}(\Gamma,\XX,U) &\leq \Poly^1_{\leq 0}(\Gamma,\XX,U) = \Hom(\Gamma,U) \\
&\leq \Poly^1_{\leq 1}(\Gamma,\XX,U) \\
&\leq \dots \\
&\leq Z^1(\Gamma,\XX,U)   
\end{align*}
which are related to the chain \eqref{type-chain} by the subgroup relation
$$ \Poly^1_{\leq k-1}(\Gamma,\XX,U) \leq Z^1_{\leq k}(\Gamma,\XX,U)$$
for any $k \geq 1$ (see \cite[Lemma 4.3(iii)]{btz}); we also clearly have the short exact sequence
$$ 0 \to U \to \Poly_{\leq k}(\XX,U) \stackrel{d_\Gamma}{\to} \Poly^1_{\leq k-1}(\Gamma,\XX,U) \cap d_\Gamma \mathcal{M}(\XX,U) \to 0$$
of abelian groups for any $k \geq 0$.

We recall some previously known relationships between order and Abramov systems:

\begin{proposition}[Order and Abramov systems]\label{order-prop}  Let $\XX=(X,\mathcal{X},\mu,T)$ be an ergodic $\Gamma$-system, and let $k \geq 1$.
\begin{itemize}
    \item[(i)] If $\XX$ is Abramov of order $\leq k$, then it is of order $\leq k$. (Equivalently: polynomials of degree $\leq k$ are $\ZZ^{\leq k}(\XX)$-measurable.)
    \item[(ii)] If $\XX$ is of order $\leq 1$ (i.e., a Kronecker system), then it is Abramov of order $\leq 1$.
    \item[(iii)] For $k \geq 2$, there exist $\Z$-systems of order $\leq k$ that are not Abramov of order $\leq k$.
    \item[(iv)] If $\Gamma = \F_p^\omega$ and $p \geq k-1$, and $\XX$ is of order $\leq k$, then it is Abramov of order $\leq k$.
    \item[(v)] If $\Gamma = \F_p^\omega$ and $\XX$ is of order $\leq k$, then it is Abramov of order $\leq C(p,k)$ for some quantity $C(p,k)$ depending only on $p,k$.
    \item[(vi)]  If $\XX$ is of order $\leq k$, then so is any factor of $\XX$.
    \item[(vii)] If $k=5$, then there exists a $\F_2^\omega$-system of order $\leq 5$ that is not Abramov of order $\leq 5$, but is a factor of a system that is Abramov of order $\leq 5$. 
\end{itemize}
\end{proposition}

\begin{proof}  For (i), see \cite[Lemma A.35]{btz}.  The claim (ii) is immediate from the fact that Kronecker systems are generated by eigenfunctions.  For (iii), see \cite{furstenbergexample}.  For (iv), see  \cite{btz}, \cite{CGSS}.  For (v), see \cite{btz}. For (vi), see \Cref{prop-functoriality}(i). For (vii), see \cite{jstcounterexample} and \cite{CGSSextensions}.
\end{proof}

\Cref{order-prop}(vii) disproves a conjecture from \cite{btz} that an $\F_p^\omega$-system of order $\leq k$ was necessarily Abramov of order $\leq k$ even in low characteristic cases.  One of the main results of this paper is to salvage a weaker form of this conjecture (which we proposed previously in \cite{jst-tdsystems}, and which is also suggested by the second part of \Cref{order-prop}(vii)):

\begin{theorem}[Abramov extension in the bounded-exponent case]\label{main-thm}  Every ergodic $\Gamma$-system of order $\leq k$ is a \emph{factor} of an Abramov system of order $\leq k$. 
\end{theorem}

In light of Proposition \ref{order-prop} and the example constructed in Section \ref{sec-example} below (which shows that finite order systems for the action of groups of bounded exponent do not admit Weyl tower extensions in general (see the next section for the definition)), Theorem \ref{main-thm} provides an essentially sharp structural description of the Host--Kra factors for bounded-exponent actions.  

\Cref{main-thm} had previously been established in the model case $\Gamma = \F_p^\omega$ by Candela--Gonz\'alez-S\'anchez--Szegedy \cite{CGSSextensions} using methods from nilspace theory \cite{candela-szegedy-cubic}; by combining this result  with the ergodic Sylow decomposition from \cite{jst-tdsystems}, one can then also establish \Cref{main-thm} in the case when the exponent $m$ is square-free (as in this case the $p$-Sylow components of $\Gamma$ are either finite or isomorphic to $\F_p^\omega$). 

The idea to use extensions to simplify the structural description of factors of measure-preserving systems characteristic for certain seminorms first appeared in the work of Austin \cite{austin-norm,austin2015pleasant}. 

By combining \Cref{main-thm} with a now-standard correspondence principle argument, combined with some additional technical ingredients, we can obtain an inverse theorem for Gowers norms in finite groups of bounded exponent, confirming a conjecture of two of the authors \cite[Conjecture 1.11]{jt21-1} for that class of groups.

\begin{theorem}[Inverse Gowers theorem for groups of bounded exponent]\label{inversegowers}
    Let $m\geq 1$ be an integer, let $k\geq 1$ and let $\delta >0$. Then there exists some $\varepsilon = \varepsilon(k,m,\delta)$ such that for every finite, $m$-exponent, abelian group $G$ and any $1$-bounded\footnote{A function $f \colon G \to \C$ is \emph{$1$-bounded} if $|f(x)| \leq 1$ for all $x \in G$.} function $f \colon G \rightarrow \mathbb{C}$ with $\|f\|_{U^{k+1}(G)}>\delta$, there exists a polynomial $P\in\Poly_{\leq k}(G)$ (viewing $G$ as a translational $G$-system in the obvious fashion) such that $$\left|\E_{x\in G} f(x) e(-P(x))\right|>\varepsilon,$$
where $e(\theta) \coloneqq e^{2\pi i \theta}$.
\end{theorem}

We define the Gowers norms $U^{k+1}(G)$ and derive this theorem from \Cref{main-thm} in \Cref{proofgowers}.

This result was previously known for $k \leq 2$ \cite{jt21-1}. For vector spaces over finite fields $\F_p$, the high-characteristic case $p>k$ was established by the third author and Ziegler in \cite{taozieglerhigh}. Their proof used a correspondence principle together with a structure theorem from ergodic theory for $\F_p^\omega$-actions, established jointly with Bergelson in \cite{btz}. One of the main goals of the present paper is to develop a general framework in which their strategy applies in both high and low characteristic, and more generally to arbitrary groups of bounded exponent.

The low-characteristic case $p\leq k$ was established in \cite{tz-lowchar} by studying various notions of rank and the equidistribution of polynomials. A crucial ingredient in that proof was the exact roots property of non-classical polynomials on vector spaces over finite fields. This property fails for general groups of bounded torsion; see the beginning of \Cref{roots-sec} for a discussion. This failure is a principal motivation for the machinery introduced in this paper.

See also \cite{milicevic}, \cite{milicevic-2}, \cite{milicevic-u56}, \cite{BSST}, \cite{milicevicquasi}, and \cite{CGSS} for other proofs of, and results related to, the inverse theorem for vector spaces over finite fields, including several proofs that give quantitative control of $\eps$. 

More recently, \Cref{inversegowers} was established for groups of square-free exponent $m$ in \cite{CGSS-squarefree}. In parallel with and independently of the present work, Mili\'cevi\'c \cite[Theorem~1.3]{lukau4general} established a quantitative version of \Cref{inversegowers} for $k=3$ and bounded-exponent groups of the form $G=(\mathbb{Z}/2^d\mathbb{Z})^n$. We note, in particular, that \Cref{inversegowers} appears to be new for groups of exponent $4$ and arbitrary $k$.

In \cite[Theorem~1.12]{jst-tdsystems}, we proved a weaker version of \Cref{inversegowers}, in which the correlating polynomial is allowed to have degree greater than $k$, with the degree depending on $m$ and $k$. Candela, Gonz\'alez-S\'anchez, and Szegedy subsequently established another weaker form \cite[Theorem~1.12]{CGSS-squarefree}, in which the correlation is with a so-called projected phase polynomial of degree at most $k$; see \cite{CGSS-squarefree} for the definition.\footnote{Very recently, these authors extended this result to the unbounded-exponent setting \cite{CGSS-proj}, where projected phase polynomials are replaced by projected nilsequences.} They further showed that projected phase polynomials are averages of higher-degree phase polynomials, thereby recovering our earlier weaker inverse theorem.

They asked in \cite[Question~6.1]{CGSS-squarefree} whether projected phase polynomials of degree at most $k$ can be approximated by a bounded number of genuine phase polynomials of degree at most $k$, and proved that this question is equivalent to establishing \Cref{inversegowers}; see \cite[Proposition~6.2]{CGSS-squarefree}. Thus, the proof of \Cref{inversegowers} also answers their question affirmatively.

In order to avoid working with projected phase polynomials (or equivalently, with passing from $G$ to a larger group extension in order to define the polynomial $P$) we will need to spend a non-trivial amount of effort in the proof of \Cref{inversegowers} to ensure that the tower of abelian extensions arising in \Cref{main-thm} are still factors of (ultraproducts of) the group $G$.

\subsection{Weyl towers and polynomial towers}

Henceforth $\Gamma$ is assumed to be a countable abelian group of bounded exponent. We establish \Cref{main-thm} by obtaining a more precise (but technical) description of $\Gamma$-systems of order $\leq k$ in terms of a certain new type of tower, which we call a \emph{polynomial tower}.  To motivate this construction, we first recall the existing notion of a Weyl system.

\begin{definition}[Weyl system]  An ergodic $\Gamma$-system $\XX$ is said to be a \emph{Weyl system of order $\leq k$} if it is of order $\leq k$, and the cocycles $\rho_1,\dots,\rho_k$ appearing in the Host--Kra tower  can be chosen to lie in $\Poly^1_{\leq j-1}(\Gamma,\ZZ^{\leq j-1}(\XX), U_j)$ for all $j=1,\dots,k$.
\end{definition}

\begin{remark}
   Every Weyl system of order $\leq k$ is also an Abramov system of order~$\leq k$ (cf.~\cite[Theorem 3.8]{btz}). However the converse is false, as demonstrated in \cite[Appendix D]{jst-tdsystems}. An expanded version of the example constructed in \cite[Appendix D]{jst-tdsystems} is presented in \Cref{sec-example} and analyzed in detail in  \Cref{analysis-sec} below. 
\end{remark}

One natural way to prove \Cref{main-thm} would be to show that every ergodic $\Gamma$-system of order $\leq k$ can be extended to a Weyl system of order $\leq k$. In \cite{btz} this was established for $\Gamma=\mathbb{F}_p^\omega$ in the "high characteristic" regime when $p>k$ (without extensions), and when $p\leq k$ it was shown that an ergodic $\Gamma$-system is a Weyl system of order $\leq C_p(k)$ for some constant depending only on $p,k$. In our previous work \cite{jst-tdsystems}, we were able to accomplish this result for all groups of bounded exponent if one generalized the notion of extension by allowing the acting group $\Gamma$ to also be extended to a larger group (and in particular to permit $\Gamma$ to become exponent-free). However, we have not been able to make this strategy work while keeping the group $\Gamma$ unchanged. In fact, we show in \Cref{analysis-sec} that the example constructed in \Cref{sec-example} does not admit any extension to a Weyl $\Gamma$-system of the same order. Our solution to this has been to replace the Host--Kra tower by a more general type of tower, where the intermediate factors are no longer required to match the Host--Kra factors $\ZZ^{\le j}(\XX)$, but for which the cocycles $\rho_j$ still can be expressed as polynomials. To this end, we define
  
\begin{definition}[Polynomial tower]\label{tower-def}  Let $k,j \geq 1$. An ergodic $\Gamma$-system $\XX$ is said to admit a \emph{polynomial tower of order $\leq k$ and height $j$} if $\XX$ is isomorphic to the top $\XX_j$ of a tower (\Cref{tower-gen}), in which each cocycle $\rho_i$, $1 \leq i \leq j$ lies in $\Poly^1_{\leq k-1}(\Gamma,\XX_{i-1}, U_i)$.
\end{definition}

Clearly, every Weyl system admits a polynomial tower of order $\leq k$ and height $k$ (in fact the lower cocycles of the Host--Kra tower are even lower degree than what is needed for this claim).  On the other hand, we have the following simple observation:

\begin{lemma}\label{tower-abramov}  Let $k \geq 1$.  If an ergodic $\Gamma$-system $\XX$ admits a polynomial tower of order $\leq k$ and some height $j$, then it is Abramov of order $\leq k$ (and hence of order $\leq k$, by \Cref{order-prop}(i)).
\end{lemma}

\begin{proof} By induction, it suffices to show that if $\XX$ is Abramov of order $\leq k$ and $\rho \colon \Gamma \times \XX \to U$ is a polynomial cocycle of degree $\leq k-1$, then $\XX \times_\rho U$ is also Abramov of order $\leq k$.  If we let $u \colon (x,u) \mapsto u$ be the vertical coordinate function of $\XX \times_\rho U$, then one has $\partial_\gamma u = \rho_\gamma$ a.e. for all $\gamma \in \Gamma$, hence $u$ is a polynomial of degree $\leq k$.  Since $\XX \times_\rho U$ is generated as a $\sigma$-algebra by $\XX$ and $u$, the claim follows.
\end{proof}

In view of \Cref{tower-abramov}, we see that to prove \Cref{main-thm} (and thus \Cref{inversegowers}), it suffices to show (assuming bounded-exponent) that every ergodic $\Gamma$-system of order $\leq k$ extends to an ergodic system that admits a polynomial tower of order $\leq k$ and some height $j$.  In fact, for technical inductive reasons it is convenient to establish a stronger result in which the polynomial tower enjoys additional useful properties (see \Cref{technical}).
\begin{definition}[Exact cocycles and tower, large spectrum and purity]\label{key-three} Let $\XX$ be an ergodic $\Gamma$-system, and $U$ a compact abelian group.  
\begin{itemize}
    \item[(i)] A $\Gamma$-cocycle $\rho \in Z^1(\Gamma,\XX,U)$ is said to be \emph{exact} if for all $d \geq 0$, and any frequency $\xi \in \hat U$ in the Pontryagin dual $\hat U \coloneqq \Hom(U,\T)$ of $U$, the cocycle $\xi \circ \rho$ is of type $\leq d$ if and only if it is polynomial of degree $\leq d-1$, that is to say the obvious subgroup relation
    $$ \{ \xi \in \hat U : \xi \circ \rho \in \Poly^1_{\leq d-1}(\Gamma,\XX)\} \leq  \{ \xi \in \hat U : \xi \circ \rho \in Z^1_{\leq d}(\Gamma,\XX)\}  $$
    is in fact an equality for all $d \geq 0$. A polynomial tower (\Cref{tower-gen}) of order $\leq k$ and height $j$ is said to be an \emph{exact polynomial tower} if all the $\Gamma$-cocycles $\rho_1,\dots,\rho_j$ in the tower are exact.
    \item[(ii)]  An ergodic $\Gamma$-system $\XX$ is said to have \emph{large spectrum} if for any $\xi$ in the countable group $\tilde \Gamma$ defined in \eqref{gamma-basis-dense}, there is a solution $\phi_\xi \in \mathcal{M}(\XX,\T)$ to the eigenfunction equation $\partial_\gamma \phi_\xi = \xi \cdot \gamma$; in other words, every element of $\tilde \Gamma$ is an eigenvalue of the $\Gamma$-action on $\XX$.
    \item[(iii)]  An ergodic $\Gamma$-system $\XX$ is said to be \emph{$k$-pure} if for almost every $x_0\in \XX$ the filtered group $\imath_{x_0}(\Poly_{\leq k}(\XX))$ is pure (see Definition \ref{pure:def}) in $\Poly_{\leq k}(\Gamma)$ where both of these groups are equipped with the polynomial filtration and $\iota_{x_0}$ is the sampling map
    \begin{equation}\label{sampling} \iota_{x_0} P(\gamma) \coloneqq P(T^\gamma x_0),
    \end{equation}
    which (as we shall show in \Cref{roots-sec}) is a well-defined injective $\Gamma$-equivariant morphism from $\Poly_{\leq k}(\XX)$ to $\Poly_{\leq k}(\Gamma)$ for almost all $x_0$. 
    \item[(iv)] An ergodic extension $\XX = \YY \times_\rho U$ of $\Gamma$-systems $\YY$ (as in \eqref{abelian-extension}) is called relatively $k$-pure if the group $(d_\Gamma \Poly_{\leq k}(\XX))^U$ of $U$-invariant functions in $d_\Gamma\Poly_{\leq k}(\XX)$ is a pure subgroup of $\Poly_{\leq k-1}^1(\Gamma,\XX)$.  
\end{itemize}
\end{definition}
In more informal terms,
\begin{itemize}
    \item [(i)] The $\Gamma$-cocycles $\rho_i$ in an exact polynomial tower are ``as polynomial as possible'', in the sense that any component $\xi \circ \rho_i$ of the $\Gamma$-cocycle will be a polynomial of the smallest degree compatible with its type (if $\xi \circ \rho_i$ is of type $\leq k$, then it will be polynomial of degree $\leq k-1$); 
    \item[(ii)] Systems with large spectrum have a plentiful (though still countable) supply of eigenfunctions; and
    \item[(iii)] In systems that are $k$-pure one can manipulate polynomials by working ``locally'' using shifts by $\Gamma$, rather than ``globally'' in the system $\XX$. 
    \item[(iv)]  The notion of relative purity will be a technical one, required to ensure that extensions of a pure system are also pure. 
\end{itemize}

\begin{theorem}[Technical form of main theorem]\label{technical} 
Let $1 \leq j\leq k$ and let $\XX$ be an ergodic $\Gamma$-system of order $\leq j$. Then there is an ergodic $\Gamma$-extension $\XX'$ of $\XX$ that has large spectrum and admits an exact polynomial tower $(\XX'_{i})_{i=0,\ldots,j}$ of order $\leq k$ and height $j$ such that the factor $\XX'_i$ is $k$-pure for each $1\leq i\leq j$ and the abelian extension 
\[\begin{tikzcd}
    \XX'_i \arrow[d, Rightarrow, "\rho_{i}; U_i"] \\
    \XX'_{i-1} 
\end{tikzcd}
\]
is relatively $k$-pure for each $1<i\leq j$. 
\end{theorem}

Thanks to \Cref{tower-abramov}, the $j=k$ case of \Cref{technical} will imply \Cref{main-thm}. We will in fact strengthen the large spectrum hypothesis in \Cref{technical} slightly, in that we will require that the factor $\XX'_1$ in the polynomial tower already has large spectrum (which of course implies the same for the full system $\XX'$), but we ignore this minor detail for this introduction. 
We can combine all the above implications together into an equivalence; see also \Cref{fig2:new}.

\begin{theorem}[Equivalence form of main theorem]\label{equiv}  Let $\XX$ be an ergodic $\Gamma$-system, and let $k \geq 1$ be an integer.  Then the following are equivalent:
\begin{itemize}
    \item[(i)]  $\XX$ is of order $\leq k$.
    \item[(ii)]  $\XX$ is a factor of an Abramov system of order $\leq k$.
    \item[(iii)]  $\XX$ is a factor of a polynomial tower of order $\leq k$ and some height $j$.
    \item[(iv)]  $\XX$ is a factor of an exact polynomial tower of order $\leq k$ and height $k$ of large spectrum that is also $k$-pure. 
    \item[(v)] $\XX$ is a factor of an exact polynomial tower of order $\leq k$ and height $k$ of large spectrum that is also $k$-pure and such that every extension in the tower is relative $k$-pure. 
\end{itemize}
\end{theorem}

\begin{figure}[ht]
\centering
 \begin{tikzcd}
  {} & {\substack{\text{Exact polynomial tower}\\ \text{order } \leq k\\ \text{height }k \\ \text{large spectrum} \\ k\text{-pure} \\ \text{relative } k\text{-pure}}} \arrow[d] \\
 {} & {\substack{\text{Exact polynomial tower}\\ \text{order } \leq k\\ \text{height }k \\ \text{large spectrum} \\ k\text{-pure}}} \arrow[d] \\
{\substack{\text{Weyl}\\{\text{order }k}}} \arrow[r] & {\substack{\text{Polynomial tower}\\ \text{order } \leq k\\ \text{height }k}} \arrow[d] \\
& {\substack{\text{Polynomial tower}\\ \text{order } \leq k\\ \text{any height }}} \arrow[d]  \\
& {\substack{\text{Abramov}\\ \text{order } \leq k}} \arrow[d]  \\
& {\substack{\text{order } \leq k}} \arrow[uuuuu, bend right=40, "\text{extension}"', dashed]  
\end{tikzcd}
    \caption{Key implications between various properties of ergodic $\Gamma$-systems when $\Gamma$ is an arbitrary group of bounded exponent, for a given $k \geq 1$. For the dashed line, one needs to pass to an extension of the original system.}
  \label{fig2:new}
\end{figure} 

In \Cref{sec:tran}, we demonstrate that exact polynomial towers admit a representation as translational systems. While this structural characterization is not strictly required to derive our main ergodic-theoretic and combinatorial conclusions (Theorems \ref{technical} and \ref{inversegowers}), it establishes an analogue for groups of bounded exponent of existing results involving different group actions \cite{shalom2,shalom1, jst-tdsystems}. To state the result precisely, we first recall the definition of a translational system.

\begin{definition}\label{tran:def}
A \emph{translational system} is a
$\Gamma$-system of the form $\XX=(G/\Lambda,\mathcal{B},\mu_{G/\Lambda},T)$ where
$G$ is a Polish group, $\Lambda\leq G$ is a closed co-compact subgroup,
$\mathcal{B}$ is the Borel $\sigma$-algebra on $G/\Lambda$,
$\mu_{G/\Lambda}$ is a normalized $G$-invariant regular Borel probability measure on $G/\Lambda$,
and $T^\gamma(g\Lambda)=\phi(\gamma)\,g\Lambda$ for some homomorphism $\phi \colon  \Gamma\to G$.  If $G$ is $k$-step nilpotent, we call $\XX$
a \emph{$k$-step translational system}. 
\end{definition}

\begin{theorem}\label{thm:tran}
Let $1\le j\le k$, and let $\XX$ be an exact polynomial tower of order $\le k$ and height $j$. 
Then $\XX$ is isomorphic to a $k$-step translational system. 
\end{theorem}

Combining \Cref{technical} and \Cref{thm:tran}, we immediately obtain the following structural description of finite order systems.  

\begin{corollary}\label{cor:tran}
  Every ergodic $\Gamma$-system of order $\leq k$ is a \emph{factor} of a $k$-step translational system. 
\end{corollary}

We discuss the proof of our main technical result, \Cref{technical}, and the key innovations behind it in a separate overview section (\Cref{sec:overview}). In the next section we present an example that illustrates our new notion of polynomial tower and, at the same time, shows that one cannot in general strengthen \Cref{main-thm} by replacing ``Abramov'' with ``Weyl'' if one insists on not also extending the acting group.\footnote{In our earlier paper \cite{jst-tdsystems} we showed that if one is allowed to extend the acting group as well - a notion we termed \emph{generalized extensions} - then every bounded-exponent system of finite order admits a Weyl extension.}

\subsection{An illustrative example}\label{sec-example}

The following example, building upon a construction from \cite{tz-lowchar}, will help motivate the concept of a polynomial tower.  Here we take $\Gamma = \F_2^\omega$.  We begin with a rotational system $\XX_{(0)}$ which we define to be the compact abelian group $(\Z/2\Z)^\N$ with the standard translation action
$$ T^{\sum_i \gamma_i e_i} (x_i)_{i \in \N} \coloneqq (x_i + \gamma_i)_{i \in \N}$$
where $e_1,e_2,\dots$ is the standard basis of $\F_2^\omega$, we use the obvious additive action of $\F_2$ on $\Z/2\Z$.  We then let $\XX = \XX_{(0)} \times_{\rho} \Z_2$ be the abelian extension of $\XX_{(0)}$ by the $2$-adic group $\Z_2 = \varprojlim \Z/2^n \Z$ and the cocycle $\rho \colon \Gamma \times X_{(0)} \to \Z_2$ given by the formula 
$$ \rho\left(\sum_i \gamma_i e_i, (x_i)_{i \in \N} \right)\coloneqq \sum_i |\gamma_i| (1 - 2 |x_i|),$$
where we define $|x| \in \Z_2$ for $x \in \Z/2\Z$ by setting $|x|=0$ if $x = 0 \mod 2$ and $|x|=1$ if $x = 1 \mod 2$, where we embed the integers into $\Z_2$ in the usual manner; note that as only finitely many of the $\gamma_i$ are non-zero, the summations here are well-defined.  One can check that this is indeed a $\Gamma$-cocycle.  

\begin{remark}  The system $\XX$ can be informally viewed, in the limit as $N \to \infty$ of the random process $((x_i)_{i=1}^N, s)$, where $(x_1,\dots,x_N)$ is drawn uniformly at random from $(\Z/2\Z)^N$, and $s \coloneqq \sum_i |x_i|$ is the Hamming norm of the random vector $(x_1,\dots,x_N)$, viewed as a $\Z_2$-valued random variable.  The shift $T^{e_i}$ then corresponds to the operation of flipping the value of the ``bit'' at $x_i$, which adjusts $s$ by $1-2|x_i|$.
\end{remark}

For any natural number $n$, we can quotient $\Z_2$ down to $\Z/2^n \Z$, which then produces intermediate factors $\XX_{(n)} \coloneqq \XX_{(0)} \times_{\rho \mod 2^n} \Z/2^n\Z$ by reducing the cocycle $\rho$ modulo $2^n$.  The factor $\XX_{(1)}$ is again a rotational system, as it can be identified with $(\Z/2\Z)^\N \times (\Z/2\Z)$ with rotational action
$$ T^{\gamma}((x,s)) = ((x,s) + \rho_1(\gamma))$$
where $\rho_{1} \colon \Gamma \to (\Z/2\Z)^\N \times (\Z/2\Z)$ is the homomorphism
$$ \rho_{1}\left(\sum_i \gamma_i e_i\right) \coloneqq \left((\gamma_i)_{i \in \N}, \sum_i \gamma_i\right).$$
One can then create a height $k$ tower
for any $k \geq 1$, where each $\XX_{(j)} = \XX_{(j-1)} \times_{\rho_j} 2^{j-1} \Z/2^j\Z$ and the cocycles $\rho_j \colon \Gamma \times \XX_{(j-1)} \to 2^{j-1} \Z/2^j \Z$ for $j \geq 2$ can be taken for instance to be
$$ \rho_j(\gamma, (x, s)) \coloneqq \rho(\gamma, x) \mod 2^j - \iota_j( \rho(\gamma, x) \mod 2^{j-1}) \in 2^{j-1} \Z/2^j \Z $$
where $\iota_j \colon \Z/2^{j-1}\Z \to \Z/2^j \Z$ is an arbitrary section of the short exact sequence
$$ 0 \to 2^{j-1} \Z / 2^j \Z \to \Z/2^j\Z \to \Z/2^{j-1}\Z \to 0$$
(for instance, one could arbitrarily define $\iota_j(k \mod 2^{j-1}) = k \mod 2^j$ for $k=0,\dots,2^j-1$, although many other choices are possible); see \Cref{height}.  

\begin{figure}[ht]
\centering
\begin{tikzcd}
    \XX_{(k)} \arrow[d, Rightarrow, "\rho_{k}; 2^{k-1}\Z/2^k\Z"] \\ 
    \XX_{(k-1)} \arrow[d, Rightarrow, "\rho_{k-1}; 2^{k-2}\Z/2^{k-1}\Z"]\\
    \vdots  \arrow[d, Rightarrow, "\rho_2; 2\Z/4\Z"] \\
    \XX_{(1)} \arrow[d, Rightarrow, "\rho_1; (\Z/2\Z)^\N \times \Z/2\Z"] \\
    \operatorname{pt}  
\end{tikzcd}
    \caption{The Host--Kra tower for $\XX_{(k)}$. For $k \geq 3$, this is not a Weyl tower (or even a polynomial tower of order $\leq k$).}
  \label{height}
\end{figure} 

In \Cref{analysis-sec} we establish the following facts about this tower:

\begin{proposition}\label{facts}\ 
\begin{itemize}
    \item[(i)]  $\XX$ is ergodic (and hence $\XX_{(k)}$ is ergodic for all $k \geq 0$).
    \item[(ii)]  For each $k \geq 1$, $\XX_{(k)}$ is of order $\leq k$ and Abramov of order $\leq k$.  Also, $\XX_{(k)} = \ZZ^{\leq k}(\XX) = \ZZ^{\leq k}(\XX_{(l)})$ for any $l \geq k$.  (In particular, this tower is the Host--Kra tower for $\XX$, and the $\rho_j$ are of type $\leq j$.)
    \item[(iii)]  For any $k,d \geq 1$, the $\Gamma$-cocycle $\rho_k$ can be selected to be polynomial of degree $\leq d$ if and only if $d \geq 2^{k-1}-1$.
    \item[(iv)] For each $k \geq 1$, $\rho \mod 2^k$ is a polynomial of degree $\leq k-1$. 
    \item[(v)] $\XX_{(3)}$ does not admit an $\F_2^\omega$ Weyl extension.
\end{itemize}
\end{proposition}

The powers of two arising in part (iii) are related to Lucas's theorem, which among other things asserts that the degree $\leq d$ polynomial $n \mapsto \binom{n}{d} \mod 2$ is $2^k$-periodic if and only if $d \geq 2^{k-1}$.  From this proposition we see that the cocycle $\rho_3$ that extends the Conze--Lesigne factor $\XX_{(2)} = \ZZ^2(\XX)$ to the third Host--Kra factor $\XX_{(3)} = \ZZ^3(\XX)$ cannot be a quadratic polynomial, but (if it is to be polynomial at all) must be at least cubic in degree.  In particular, once $k \geq 3$, $\XX_{(k)}$ is of order $\leq k$ but \emph{not} a Weyl system of order $\leq k$, and \eqref{height} is \emph{not} a polynomial tower of order $\leq k$ and height $k$.

On the other hand, each $\XX_{(k)}$, $k \geq 2$ does admit a polynomial tower of order $\leq k$ and height $2$ in \Cref{alt-tower}
where $\rho_0 \colon \Gamma \to (\Z/2\Z)^\N$ is the homomorphism
$$ \rho_0\left(\sum_i \gamma_i e_i\right) \coloneqq (\gamma_i)_{i \in \N}.$$

\begin{figure}[ht]
\centering
\begin{tikzcd}
    \XX_{(k)} \arrow[d, Rightarrow, "\rho \mod 2^k; \Z/2^k\Z"]  \\ 
    \XX_{(0)} \arrow[d, Rightarrow, "\rho_0; (\Z/2\Z)^\N"]\\
    \operatorname{pt}  
\end{tikzcd}
    \caption{A polynomial tower of order $\leq k$ and height $2$ for $\XX_{(k)}$.}
  \label{alt-tower}
\end{figure}

\begin{remark}  The first few systems $\XX_{(k)}$, $k=0,1,2$ of the above tower were studied in \cite[Appendix E]{tz-lowchar}, where it was noted that $\XX_{(2)}$ failed a variant of the exact roots property (a property weaker than $k$-purity), in that there existed quadratic polynomials $P \in \Poly_{\leq 2}(\XX_{(2)})$ that were not of the form $P=2Q$ for any cubic polynomial $Q \in \Poly_{\leq 3}(\XX_{(2)})$, although it turns out that such $P$ can be expressed as $2R$ for a cubic polynomial $R \in \Poly_{\leq 3}(\XX_{(3)})$ in the extension $\XX_{(3)}$.  This failure of the exact roots property is related to the inability to ``straighten'' the cocycles in \Cref{facts}(iii) to be of lower degree than $2^{k-1}-1$.
\end{remark}



\subsection*{Acknowledgements}
A.J. was funded by the Deutsche Forschungsgemeinschaft (DFG, German Research Foundation) Heisenberg Grant - 547294463. 
OS was supported by NSF grant DMS-1926686 and Alon Fellowship.  Over the course of this research, TT was supported by a Simons Investigator grant, the James and Carol Collins Chair, the Mathematical Analysis \& Application Research Fund, and by NSF grants DMS-1764034 and DMS-2347850, and is particularly grateful to recent donors to the Research Fund.

\section{Overview of the main steps in the argument}\label{sec:overview}

The primary technical contribution of this paper is \Cref{technical}, from which our main ergodic-theoretic result, \Cref{main-thm}, follows. In this section, we outline the methodology employed to derive \Cref{technical}. We also emphasize our main new contributions. These include:
\begin{itemize}
    \item An integration lemma for cocycles in exact polynomial towers with large spectrum (\Cref{cocycle-integ}). See also \Cref{int1-sec,int2-sec} for related results.
    \item The classification of injective objects (purity) in the category of filtered abelian groups (see Appendix \ref{filteredcategory}). 
    \item The characterization of purity via the existence of specific retractions (\Cref{roots-sec,invariant}).
    \item The existence of pure extensions (\Cref{proof}).
\end{itemize}

Let $1\leq j \leq k$, and let $\XX$ be an ergodic $\Gamma$-system of order $\leq j$. Our goal is to find an extension $\XX'$ that is an (exact) polynomial tower of order $\leq k$. Furthermore, we require that $\XX'$ satisfies the following conditions:
\begin{itemize}
    \item[(i)] large spectrum;
    \item[(ii)] $k$-pure;
    \item[(iii)] $\XX_i'$ is a relatively $k$-pure extension of $\XX'_{i-1}$.
\end{itemize}

\Cref{abelext} allows us to write $\XX$ as an abelian extension
\[
\XX \cong \ZZ^{\leq j-1}(\XX)\times_\rho U,
\]
where $\rho$ is a cocycle of type $\leq j$ and $U$ is a compact abelian group. Conventionally, this proposition is applied in conjunction with an inductive argument on $j$, to replace $\ZZ^{\leq j-1}(\XX)$ with a more tractable system. However, a system of order $\leq j-1$ generally does not admit extensions that are $j$-pure (or, more generally, extensions that are $k$-pure with the degree control required later) if one insists on working strictly within the Host--Kra tower factors. To overcome this, we must extend $\ZZ^{\leq j-1}(\XX)$ to a system of order $\leq j$ (or more generally, order $\leq k$). To implement this move rigorously within an inductive framework, we must formally decouple the notions of \emph{order} and \emph{height}. This distinction is the primary motivation for the separate indices $j$ and $k$ appearing in the statement of \Cref{technical}.

By the induction hypothesis, there exists an extension $\YY^{j-1}$ which is an exact polynomial tower of order $\leq k$ and satisfies the technical conditions (i)--(iii). The extension $\pi^{j-1}\colon \YY^{j-1}\rightarrow \ZZ^{\leq j-1}(\XX)$ allows us to lift the cocycle $\rho$, thereby inducing an extension of $\XX$ given by
\[
\YY^{j-1}\times_{\rho\circ\pi^{j-1}} U.
\]
This extension may not be ergodic, but a result of Zimmer \cite{zimmer} shows that one can always replace $\rho\circ\pi^{j-1}$ with a minimal cocycle $\rho'$ cohomologous to $\rho\circ\pi^{j-1}$, and replace $U$ by the corresponding Mackey range (a closed subgroup $U'\leq U$), to obtain an ergodic extension $\YY^{j-1}\times_{\rho'} U'$ of $\XX$.

At this stage, our goal is to demonstrate that $\rho'$ is cohomologous to an exact cocycle. Establishing this would allow us to replace $\rho'$ with such an exact representative, thereby realizing the system as an exact polynomial tower extending $\XX$. Following standard practice in the field (cf.~\cite{btz,shalom2,jst-tdsystems}), we employ Pontryagin duality and a form of divisibility (in our case the $k$-purity of the extension), to reduce the problem to the case where the cocycle takes values in the torus $\T$.

We proceed by adapting the framework from \cite{btz}. We define a system to satisfy the \emph{straightening property up to level $k$} if every $\T$-valued cocycle of type $\leq d$ is cohomologous to a polynomial cocycle of degree $\leq d-1$ for all $d\leq k$. By induction on $d$, we demonstrate that $\YY^{j-1}$ possesses this property. Assuming we succeed in this regard, it follows that for every character $\xi\in\widehat{U'}$, the torus-valued cocycle $\xi\circ\rho'$ is cohomologous to a polynomial of the smallest degree compatible with its type; unpacking the definition of exactness then shows that $\rho'$ is cohomologous to an exact cocycle, as required.

Let $\sigma$ be a (generic) $\T$-valued cocycle of type $\leq d$ on $\YY^{j-1}$, and assume the straightening claim holds for all smaller values of $d$. We represent $\YY^{j-1}$ as a tower
\[
\YY^{j-1}
= U_1\times_{\rho_1} U_2\times \dots \times_{\rho_{j-2}} U_{j-1},
\]
where $U_1,\dots,U_{j-1}$ are compact abelian groups and $\rho_1,\dots,\rho_{j-2}$ are exact. We then employ a downward induction to linearize the \emph{higher order Conze--Lesigne equations} associated with vertical translations by the structure groups $U_{j-1},\dots,U_1$.

As the inductive step mirrors the basis, we focus here on the structure group $U_{j-1}$. The inductive hypothesis on $d$ yields a higher order Conze--Lesigne equation
\begin{equation}\label{hoCL}
    \partial_u \sigma \;=\; p_u \;+\; d_\Gamma F_u,
\end{equation}
for some polynomial $p_u$ and a measurable map $F_u$. The degree bookkeeping here is governed by the type (weight) filtration on $U_{j-1}$: if $u$ has weight $\weight(u)$ (equivalently $u\in (U_{j-1})_{>\ell-1}$ with $\weight(u)\ge \ell$), then $p_u$ can be chosen to have degree
\[
p_u \in \Poly^1_{\le d-1-\weight(u)}(\Gamma,\YY^{j-1})
\]
\[
\text{(in particular, }p_u\in \Poly^1_{\le d-\ell-1}\text{ whenever }u\in (U_{j-1})_{>\ell-1}\text{)}.
\]
Noting that the pair $(p_u,F_u)$ is not unique, our objective is to find a choice where $u\mapsto p_u$ and $u\mapsto F_u$ are linear (i.e. cocycles in $u$):
\[
p_{u+v}=p_u+p_v\circ V_u,
\qquad
F_{u+v}=F_u+F_v\circ V_u,
\qquad
(u,v\in U_{j-1}),
\]
where $V_u$ denotes the vertical translation action of $U_{j-1}$.

Once this linearity is established, we take advantage of the large spectrum and our integration theorem for cocycles (see \Cref{cocycle-integ}) in order to write $(p_u)_{u\in U_{j-1}} = d_U P$ for some polynomial $P$ of the appropriate degree. Moreover, by \Cref{HK-C8} we can write $(F_u)_{u\in U_{j-1}} = d_UF$ for some $F\in \mathcal{M}(\YY^{j-1})$. Consequently, $\sigma-P-d_\Gamma F$ is $U_{j-1}$-invariant. This invariance implies that the remainder is measurable with respect to a polynomial tower of strictly smaller height, allowing us to close the argument via the inductive hypothesis on $j$.

The linearization process typically involves two steps: linearizing $u\mapsto p_u$ over an open subgroup of $U_{j-1}$, followed by resolving the case where the last structure group is finite. The former is achieved by modifying the argument in \cite[Proposition 6.1]{btz}; the primary challenge here is preserving the degrees of the polynomials $p_u$, which may vary with $u$ via the weight filtration. Given the similarity of the first stage to existing literature, we shall focus our exposition on the case where $U_{j-1}$ is finite.

In the finite group case, we rely on the bounded-exponent assumption to apply a result from \cite[Theorem 1.4]{jst-tdsystems}, allowing us to represent
\[
U_{j-1} \;=\; \prod_{i=1}^N \Z/m_i\Z,
\]
where each $m_i$ divides the torsion (exponent) of $\Gamma$. Following the strategy in \cite[Proposition 7.1]{btz}, we linearize \Cref{hoCL} coordinate-by-coordinate, repeating the argument $N$ times via induction.

To illustrate this, we focus on a single cyclic component by selecting a generator $e\in U_{j-1}$ for one such component. Let $U=\langle e\rangle$, equipped with the filtration induced by $U_{j-1}$. For each $1\leq i \leq k$, we define $n_i$ to be the minimal integer so that $n_ie\in U_{>i}$. These integers generate a \emph{system of relations}. For the purpose of this overview, consider the special case where $U=\Z/4\Z$ is endowed with the $2$-adic filtration (i.e. $n_1 = 2$, $n_2 = 4$). From \Cref{hoCL}, we obtain the following relations for the generator $e$ and its multiple $2e$,
\[
\partial_e \sigma - d_\Gamma F_e = p_e,
\qquad
\partial_{2e} \sigma - d_\Gamma F_{2e} = p_{2e},
\]
where (by the preceding degree bookkeeping) $p_e$ has degree $\le d-2$ and $p_{2e}$ has degree $\le d-3$ in this $2$-adic example (reflecting $\weight(e)=2$ and $\weight(2e)=3$). The cocycle identity yields the telescoping series $\sum_{i=0}^3 \partial_e \sigma \circ V_{ie} =0$, and the identity $\partial_{2e}\sigma = \partial_e \sigma + \partial_e \sigma\circ V_{e}$. By substituting these into our expressions for $p_u$, we derive the following system of relations,
\begin{align*}
   &d_\Gamma\!\left(\sum_{i=0}^3 \partial_e F_e\circ V_{ie}\right) \;=\; \sum_{i=0}^3 \partial_e p_e\circ V_{ie}, \\
   &d_\Gamma\!\left(F_{2e} - \sum_{i=0}^1 F_e\circ V_{ie}\right) \;=\; p_{2e}-\sum_{i=0}^1 p_e\circ V_{ie}.
\end{align*}
For the sake of simplicity, we further assume that $p_e$ is invariant to translations by $e$ (a non-trivial assumption that we address in the sequel). Under this condition, the system reduces to
\begin{align*}
   &\sum_{i=0}^3 \partial_e F_e\circ V_{ie} \in 4\Poly_{\leq d}(\YY^{j-1}), \\
   &F_{2e} - \sum_{i=0}^1 F_e\circ V_{ie} \in 2\Poly_{\leq d}(\YY^{j-1})+\Poly_{\leq d-1}(\YY^{j-1}).
\end{align*}
We interpret these as linear equations: $b_1=4x_1$, and $b_2 = 2x_1 +x_2$, where $b_1,b_2$ are fixed constants, $x_1\in \Poly_{\leq d}(\YY^{j-1})$ is a variable to be determined later, and $x_2\in \Poly_{\leq d-1}(\YY^{j-1})$ is a lower order error term.

The preceding analysis ensures that this system is solvable over $\Gamma$. Formally, we may embed $\Poly_{\leq d}(\YY^{j-1})$ in $\Poly_{\leq d}(\Gamma)$ by fixing a generic point $x_0\in Y^{j-1}$ and defining the map $\iota_{x_0}(P)(\gamma)=P(T^\gamma x_0)$. Applying this map to the equations above yields a system of equations where the variables are polynomials on $\Gamma$. The assumption of $k$-purity of $\YY^{j-1}$ means that whenever there is a solution to a finite system of relations in the filtered group $\Poly_{\leq d}(\Gamma)$, there is also a solution in $\Poly_{\leq d}(\YY^{j-1})$. By subtracting this solution from $F_e$, and adjusting $p_e$ by adding the corresponding derivative, we can force the first equation to zero while preserving the degrees of the polynomials $p_u$ for all $u\in U$. Fortunately, in that case \Cref{HK-C8} can be used to linearize $u\mapsto F_u$, but since $u\mapsto \partial_u \sigma$ is a cocycle, equation \Cref{hoCL} guarantees that $u\mapsto p_u$ is now also linear.

In the general case where $p_e$ is not invariant to translations by $e$, we develop a mechanism that allows us to solve equations involving vertical translations by the last structure group. We observe that, in our setting, $k$-purity can be characterized via the existence of retractions for certain short exact sequences in the category of filtered locally compact abelian groups. Specifically, by fixing a generic point $x_0\in Y^{j-1}$, the embedding $\iota_{x_0}$ induces a short exact sequence
\[
0\rightarrow \Poly_{\leq k}(\YY^{j-1})\overset{\iota_{x_0}}{\rightarrow}\Poly_{\leq k}(\Gamma)\rightarrow \Poly_{\leq k}(\Gamma)/\iota_{x_0}(\Poly_{\leq k}(\YY^{j-1}))\rightarrow 0.
\]
While this sequence does not split in general, $k$-purity implies that its restriction to any subgroup $B\subseteq \Poly_{\leq k}(\Gamma)$ which contains $\iota_{x_0}(\Poly_{\leq k}(\YY^{j-1}))$ as a subgroup of finite index, does split in the category of filtered locally compact abelian groups. Let $B\leq \Poly_{\leq k}(\Gamma)$ denote the subgroup generated by $\iota_{x_0}(\Poly_{\leq k}(\YY^{j-1}))$, $p_e$ and all of its translations under the action of $e$. Under these conditions, we obtain a retraction $r\colon B\rightarrow \Poly_{\leq k}(\YY^{j-1})$ which preserves the degrees of polynomials. In \Cref{EDPretraction}, we demonstrate that the assumption of relative $k$-purity implies that there exists a retraction that is additionally equivariant with respect to the action of $U$ on $B$ and $\Poly_{\leq k}(\YY^{j-1})$. The existence of such an equivariant retraction allows us to resolve the previously derived equations without assuming invariance of $p_e$, thereby achieving the required linearization.

To complete the induction, we must demonstrate that the resulting polynomial tower is exact and satisfies the technical conditions (i)--(iii). While exactness follows from the straightening property, the system constructed thus far may not inherit the other properties. Consequently, we must pass to a further extension that preserves the exact polynomial-tower structure while satisfying the required properties.

Condition (i), the large spectrum property, is achieved by taking an ergodic component of the joining of $\XX$ with a rotational system possessing a large spectrum. Condition (ii), $k$-purity, follows from condition (iii); indeed, a relatively $k$-pure extension of a $k$-pure system is itself $k$-pure.

To achieve relative $k$-purity, we utilize a construction that, given a system $\XX$ and a family of $\Gamma$-polynomials, realizes these polynomials on an extension $\XX'$. In fact, we demonstrate that this construction can be performed such that $\XX'$ remains an exact polynomial tower. Specifically, we consider the countable collection of all finite systems of relations. Whenever a solution exists in $\Gamma$, we extend the system to incorporate it. Since this extension may generate more systems of linear equations, we repeat this process iteratively and take an inverse limit (see \Cref{relativepure} for a precise formulation). The resulting system is a relatively $k$-pure extension that remains an exact polynomial tower with a large spectrum. We demonstrate that such a system is necessarily $k$-pure, thereby completing the proof of \Cref{technical}.

\section{Analysis of the example}\label{analysis-sec}

In this section we prove \Cref{facts}, establishing parts (i)--(v) in turn.  Strictly speaking, the material in this section is not needed elsewhere in the paper, but may serve to provide useful intiution for the arguments in those sections.

We begin with part (i).  The translational system $\XX_{(0)} = \mathrm{pt} \times_{\rho_0} (\Z/2\Z)^\N$ is ergodic because the homomorphism $\rho_0 \colon \Gamma \to (\Z/2\Z)^\N$ has dense image.  To show that the extension $\XX = \XX_{(0)} \times_{\rho} \Z_2$ (and hence all the intermediate extensions $\XX_{(k)}$) are ergodic, it suffices by \Cref{zimmer} to show that the $\Gamma$-cocycle $\rho$ is not cohomologous to a cocycle taking values in a closed proper subgroup of $\Z_2$.  The only such subgroups are of the form $2^m \Z_2$ for some $m \geq 1$, and in particular are contained in $2\Z_2$; so it suffices to show that the quotient of $\rho$ by $2\Z_2$, that is to say $\rho_1$, is not a $\Gamma$-coboundary in $\Z/2\Z$.  Suppose for contradiction that this were the case, thus there existed $F \in \mathcal{M}(\XX_{(0)}, \Z/2\Z)$ such that $\rho_1 = d_\Gamma F$. 
  In particular this would imply that $\partial_{e_n} F = 1 \mod 2$, but by approximating $F$ in measure by functions depending on finitely many coordinates, we have $\partial_{e_n}F \to 0$ in measure as $n\to\infty$, giving the desired contradiction.  This establishes ergodicity. 

We introduce the coordinate functions $x_n \in \mathcal{M}(\XX, \Z/2\Z)$ for $n \in \N$ and $s \in \mathcal{M}(\XX,\Z_2)$ by the formulae
$$ x_n \colon ((x_i)_{i=1}^\infty, s) \mapsto x_n; \quad s \colon ((x_i)_{i=1}^\infty, s) \mapsto s.$$
Clearly we have $\partial_{e_j} x_n = 1_{n=j} \mod 2$, so the $x_n$ are linear.  From construction we also have
$$ \partial_{e_j} s = \rho(e_j) = 1 - 2 |x_j|$$
so in particular $\partial_{e_j} \partial_{e_k} s = 0$ for $j \neq k$.
From the identity $|a+b| = |a| + |b| - 2 |a| |b|$, we have $\partial_h |f| = |\partial_h f| - 2 |f| |\partial_h f|$ for any function $f$ and shift $h$, so by a routine induction we have
$$ \partial_{e_j}^k s = (-2)^{k-1} (1 - 2 |x_j|)$$
for any $k \geq 1$.  From this we conclude that $s \mod 2^k$ is a polynomial of degree $\leq k$ for all $k \geq 1$.  Since $\partial_h s  = \rho(h) \mod 2^k$, we conclude that $\rho\mod 2^k$ is a polynomial of degree $\leq k-1$, giving part (iv).  As $\XX_{(k)}$ is generated by the $x_n$ and $s \mod 2^k$, we also see that $\XX_{(k)}$ is an Abramov system of order $\leq k$, and hence also of order $\leq l$ for any $l \geq k$, giving some of the components of part (ii).  To complete the proof of (ii), it suffices by \Cref{prop-functoriality} to show that $\ZZ^{\leq k}(\XX) \leq \XX_{(k)}$ for all $k \geq 1$.

The Pontryagin dual of the compact abelian group $(\Z/2\Z)^\N \times \Z_2$ is the abelian discrete group $\Gamma \times \Z(2^\infty)$, where $\Z(2^\infty) \coloneqq \varinjlim \frac{1}{2^n} \Z/\Z$ is the Pr\"ufer $2$-group.  By Plancherel's theorem, this implies that the characters
$$ \chi_{\sum_i \gamma_i e_i,a/2^m \mod 1}((x_i)_{i=1}^\infty, s) \coloneqq e\left( \frac{\sum_i \gamma_i x_i}{2} + \frac{a s}{2^m} \right)$$
for $\sum_i \gamma_i e_i \in \Gamma$, $a$ an odd integer, and $m \geq 1$, form an orthonormal basis of $L^2(\XX)$.  We claim that the Gowers--Host--Kra seminorms $$\| \chi_{\sum_i \gamma_i e_i,a/2^m \mod 1} \|_{U^{k+1}(\XX)}$$ (as defined in, e.g., \cite[Appendix A]{btz}) vanish for $m \geq k+1$; by \cite[Lemma A.32]{btz}, this implies that all such characters are orthogonal to $\ZZ^{\le k}(\XX)$.  The remaining characters are all $\XX_{(k)}$ measurable, so this implies that $\ZZ^{\leq k}(\XX) \leq \XX_{(k)}$ as required.

It remains to calculate the Gowers--Host--Kra seminorms.  Let $h_1,\dots,h_{k+1} \in \Gamma$ be shifts, and write $h_j = \sum_i h_{j,i} e_i$ for $h_{j,i} \in \{0,1\}$.  The expression $\frac{\sum_i \gamma_i x_i}{2}$ is linear and thus annihilated by $\partial_{h_1} \dots \partial_{h_{k+1}}$.  As for the
$\frac{a s}{2^m}$ term, a routine induction using the identity $|a+b|-|a| = |b| (1 - 2|a|)$ for $a,b \in \Z/2\Z$ shows that
$$ \partial_{h_1} \dots \partial_{h_{k+1}} \frac{as}{2^m} = \sum_i \frac{a h_{1,i} \dots h_{k+1,i}}{2^m} (-2)^{k} (1 - 2|x_i|)$$
and hence, with $\Delta_h \chi(x) \coloneqq \chi(x+h) \overline{\chi(x)}$ denoting the multiplicative derivative,
$$ \Delta_{h_1} \dots \Delta_{h_{k+1}} \chi_{\sum_i \gamma_i e_i,a/2^m \mod 1}((x_i)_{i=1}^\infty, s) = \prod_i e\left( \frac{a h_{1,i} \dots h_{k+1,i}}{2^m} (-2)^{k} (1 - 2|x_i|) \right).$$
(All but finitely many of the terms in the product are equal to one.)  Integrating in the $x_i, s$ variables, we conclude that
$$ \int_{\XX} \Delta_{h_1} \dots \Delta_{h_{k+1}} \chi_{\sum_i \gamma_i e_i,a/2^m \mod 1} = \prod_i \cos\left( 2 \pi \frac{a h_{1,i} \dots h_{k+1,i}}{2^m} (-2)^{k} \right).$$
If $m\ge k+2$, then for any $i$ with $h_{1,i}\cdots h_{k+1,i}=1$ we have
\[
\left|\cos\!\left(2\pi\frac{a}{2^{m-k}}\right)\right|<1,
\]
and the right-hand side is bounded uniformly away from $1$ (since $a$ is odd and $m-k\ge 2$).
By the law of large numbers, for large $N$ and for asymptotically almost all
$h_1,\dots,h_{k+1}$ in the span of $e_1,\dots,e_N$, the set of indices
$i\in\{1,\dots,N\}$ with $h_{1,i}\cdots h_{k+1,i}=1$ has cardinality
$(2^{-(k+1)}+o(1))N$, and hence the above product is exponentially small in $N$.

In the remaining boundary case $m=k+1$, the integrand simplifies to
\[
\Delta_{h_1}\dots\Delta_{h_{k+1}}\chi_{\sum_i\gamma_i e_i,a/2^{k+1}\!\!\!\mod 1}
= (-1)^{\#\{i:\ h_{1,i}\cdots h_{k+1,i}=1\}},
\]
so averaging over $h_1,\dots,h_{k+1}$ in the span of $e_1,\dots,e_N$ gives
\[
\mathbb{E}_{h_1,\dots,h_{k+1}}
(-1)^{\#\{i:\ h_{1,i}\cdots h_{k+1,i}=1\}}
=\left(1-\frac{1}{2^k}\right)^N \to 0 \qquad (N\to\infty).
\]
Averaging in $h_1,\dots,h_{k+1}$ and using the convergence properties of the
Gowers--Host--Kra seminorms (see \cite[Lemma A.18]{btz}),
we obtain the desired vanishing
$$ \left\| \chi_{\sum_i \gamma_i e_i,a/2^m \mod 1} \right\|_{U^{k+1}(\XX)}^{2^{k+1}} = 0$$
for all $m\ge k+1$, completing the proof of (ii).

We establish (iii).  We begin with the ``only if'' direction.  Suppose for contradiction that one could choose the $\Gamma$-cocycle 
$$\rho_k \in Z^1_{\leq k}(\Gamma, \XX_{(k-1)}, 2^{k-1} \Z / 2^k \Z)$$
to be polynomial of degree at most $2^{k-1}-2$.  
Since $\XX_{(k)}$ is isomorphic to $\XX_{(k-1)} \times_{\rho_k} 2^{k-1} \Z / 2^k \Z$, we can use the vertical coordinate of the latter space to construct a measurable function $u \in \mathcal{M}(\XX_{(k)}, 2^{k-1} \Z / 2^k \Z)$ such that $\partial_\gamma u = \rho_k(\gamma)$ for all $\gamma \in \Gamma$, and such that $\XX_{(k)}$ is generated as a measure algebra by $\XX_{(k-1)}$ and $u$.  Since we are assuming the $\rho_k(\gamma)$ to be polynomials of degree $\leq 2^{k-1}-2$, we conclude that $u$ is polynomial of degree $\leq 2^{k-1}-1$.  In particular, any $2^{k-1}-1$-fold derivative of $u$ is constant, hence any $2^{k-1}-1$-fold derivative of $\partial_S u$ vanishes; thus $\partial_S u$ is a polynomial of degree at most $2^{k-1}-2$.  Iterating this, we conclude that $\partial_S^{2^{k-1}-1} u$ is of degree $\leq 0$, thus constant by ergodicity; in particular, $\partial_S^{2^{k-1}} u = 0$.

On the other hand, as the group $2^{k-1} \Z / 2^k \Z$ has exponent $2$, we have $2u = 0$.  Using the identity $\partial_{S^2} = 2 \partial_S + \partial_S^2$, we conclude that $\partial_{S^2} u = \partial_S^2 u$.  A routine induction then shows that $\partial_{S^{2^{k-1}}} u = \partial_S^{2^{k-1}} u = 0$; thus $u$ is $S^{2^{k-1}}$-invariant and thus lies in $\mathcal{M}(\XX_{(k-1)},2^{k-1}\Z/2^k\Z)$.  But then $\XX_{(k)}$ would be equal to $\XX_{(k-1)}$, which is absurd.  Thus, $\rho_k$ cannot have degree less than or equal to $2^{k-1}-2$.

To conclude the proof of (iii), it suffices to exhibit a polynomial $u \in \Poly_{\leq 2^{k-1}}(\XX_{(k)}, 2^{k-1} \Z / 2^k \Z)$ such that $\XX_{(k)}$ is generated as a measure algebra by $\XX_{(k-1)}$ and $u$. We can take the explicit choice
$$ u( (x_i)_{i=1}^\infty, s ) \coloneqq 2^{k-1} \binom{s}{2^{k-1}} \mod 2^k,$$
noticing from Lucas's theorem that the binomial coefficient $\binom{s}{2^{k-1}} \mod 2$ is periodic modulo $2^k$, but not modulo $2^{k-1}$, and so $u$ is well-defined and measurable with respect to $\XX_{(k)}$ but not $\XX_{(k-1)}$; it is then clear that $\XX_{(k)}$ is generated by $\XX_{(k-1)}$ and $u$; it remains to verify that $u$ is a polynomial of degree at most $2^{k-1}$.  It will suffice to prove that $\binom{s}{d} \mod 2$ is of degree $\leq d$ for any $d \geq 0$.  This is clear for $d=0$; for higher $d$ one can proceed by induction, noting from the binomial theorem that 
$$ \partial_{e_i} \binom{s}{d} = \binom{s}{d-1} (1-2|x_i|) + \binom{s}{d-2} \binom{1-2|x_i|}{1}+ \dots + \binom{1-2|x_i|}{d-1} \mod 2,$$
which is a linear combination of $\binom{s}{d'}$ for $d' \leq d-1$ and $\binom{s}{d'} x_i$ for $d' \leq d-2$ and thus of degree at most $d-1$ by induction hypothesis.  This gives the claim. 

It remains to establish (v). To this end, we represent 
\[
\XX_{(2)} = (\Z/2\Z)^\N\times_\rho \Z/4\Z
\]
where 
\[
\rho\colon\Gamma\times (\Z/2\Z)^\N\to \Z/4\Z
\]
is defined by 
\[
\rho\left(\sum_{i\in A} e_i,x\right)\coloneqq \sum_{i\in A} (1-2|x_i|) \bmod 4
\]
for any finite $A$. 
Let $\sigma\colon \Gamma\times \XX_{(2)}\to \frac{1}{2}\Z/\Z$ be the cocycle defined by 
\[
\sigma(\gamma,(x,t))\coloneqq \frac{\rho(\gamma,x) - \partial_\gamma F(x,t)}{4}
\]
where $F(x,i\bmod 4)= i \bmod 8$ for $i=0,\ldots,3$. 
Thus, we have 
\[
\XX_{(3)}=\XX_{(2)}\times_{\sigma} \frac{1}{2}\Z/\Z.
\]

\begin{lemma}\label{noquadraticlift}
    Let $\pi\colon \YY\to \XX_{(2)}$ be a Weyl extension. Then $\sigma\circ \pi$ is not cohomologous to a quadratic cocycle on $\YY$ (with values in $\frac{1}{2}\Z/\Z$).  
\end{lemma}

Assuming this lemma, let us prove that there does not exist a Weyl extension of $\XX_{(3)}$, establishing (v). 
    Towards a contradiction, assume there exists a Weyl extension $\YY$ of $\XX_{(3)}$. 
    Writing $\YY=\ZZ^{\le 2}(\YY)\times_p U$, denoting by $\pi\colon \ZZ^{\le 2}(\YY)\to \XX_{(2)}$ the factor map, and using \cite[Proposition A.9]{jst-tdsystems}, there is a surjective group homomorphism $\varphi\colon U\to \frac{1}{2}\Z/\Z$ such that $\varphi \circ p$ is cohomologous to $\sigma\circ\pi$ on $\ZZ^{\le 2}(\YY)$. Since $\varphi \circ p$ is a quadratic cocycle by assumption, this contradicts \Cref{noquadraticlift}.

It remains to establish \Cref{noquadraticlift}. Towards a contradiction, assume that there are $G\in\mathcal{M}(\YY,\frac{1}{2}\Z/\Z)$ and a quadratic cocycle $q\colon \Gamma\times \YY\to \frac{1}{2}\Z/\Z$ such that 
\[
\sigma\circ \pi = q + dG.
\]
Then 
\[
d\left(\frac{F\circ \pi}{8}+\frac{G}{2}\right) = \frac{\rho\circ \pi_1\circ\pi(y)}{8} - \frac{q}{2}
\]
where $\pi_1\colon \XX_{(2)}=(\Z/2\Z)^\N\times \Z/4\Z \to (\Z/2\Z)^\N$ describes the coordinate projection. It was observed in \cite[Appendix D]{jst-tdsystems} that the right hand side of the previous equation is a polynomial of degree $\leq 2$. Thus $Q=\frac{F\circ \pi}{8}+\frac{G}{2}$ is a cubic polynomial. Moreover, we have 
\[
2 Q = \frac{F\circ \pi}{4} = \iota\circ \pi
\]
where $\iota\colon (\Z/2\Z)^\N\times \Z/4\Z\to \T$ is the map $\iota(x,t)\coloneqq \frac{t}{4}$ which is polynomial of degree $\leq 2$ (cf. \cite[Appendix E]{tz-lowchar}). 

We could conclude by showing that on $\YY$ there is no cubic polynomial $Q$ such that $2Q=\iota\circ \pi$. To this end, we will use a different representation of $\XX_{(2)}$. From the previous analysis, we can represent 
\[
\XX_{(2)} = ((\Z/2\Z)^\N\times \Z/2\Z )\times_{\rho_2} 2\Z/4\Z
\]
with the action $T^\gamma(x,r,t)=((x,r)+\rho_1(\gamma), \rho_2(\gamma,(x,r))+t)$. 

We write $\YY=\ZZ^{\le 1}(\YY)\times_p U$. Let $\varphi\colon U\to 2\Z/4\Z$ be the surjective homomorphism given by \cite[Proposition A.9]{jst-tdsystems}. Now choose $t\in U$ such that $\varphi(t)=2\in 2\Z/4\Z$. 
 Since $(0,t)$ fixes $\ZZ^1(\YY)$, the derivative by $(0,t)$ reduces the degree of polynomials by $2$ (cf. \Cref{ppfacts}(iii)), if $P$ is a cubic polynomial on $\YY$, then $\Delta_{(0,t)}\Delta_{(0,t)}P=0$. Moreover it follows from the cocycle identity and since $U$ is a $2$-torsion group (cf. \cite[Theorem 1.4]{jst-tdsystems}) and therefore $2t=0$ that
\[
\Delta_{(0,t)}P + \Delta_{(0,t)}P\circ V_{(0,t)} = 0.
\]
It follows that $\Delta_{(0,t)}(2P)=0$, but $\Delta_{(0,t)}(\iota\circ \pi)=\frac{1}{2}\neq 0$, yielding the desired contradiction. 

\section{Integrating cocycles on compact groups: the role of exact cocycles}\label{int1-sec}

Let $\XX$ be an ergodic $\Gamma$-system.  A compact abelian group $U$ is said to \emph{act freely} on this $\Gamma$-system if the $\Gamma$ action can be enlarged to a $\Gamma \times U$ action (in particular, the $U$-action and $\Gamma$-action commute), and the vertical shift maps $V \colon u \mapsto V_u$ of this action are continuous and faithful (i.e., injective) from $U$ to the unitary group $\mathcal{U}(L^2(\XX))$ (equipped with the strong operator topology).  For instance, any ergodic skew product $\YY \times_\rho U$ has a free action of $U$ coming from the vertical shifts $V_v (x,u) \coloneqq (x,u+v)$; conversely, it is well known (see \cite[Theorem 3.29]{glasner2015ergodic}) that any free action arises in this fashion, up to isomorphism.

If $\XX$ has a free $U$ action $u \mapsto V_u$, then $U$ also acts on $\mathcal{M}(\XX,A)$ by setting $u F \coloneqq F \circ V_u$ for any compact group $A$.  Thus we can define $U$-cocycles $(f_u)_{u \in U}$ in $Z^1(U; \mathcal{M}(\XX,A))$, $U$-coboundaries $d_U F$ in $d_U \mathcal{M}(\XX,A)$, and so forth as per \Cref{gen-cocycle}.  We first observe that the action here has trivial cohomology:

\begin{lemma}[Integrating a cocycle]\label{HK-C8}  Let $U, A$ be compact abelian groups, and let $\XX$ be an ergodic $\Gamma$-system with a free $U$ action $u \mapsto V_u$.  Let $f \colon u \mapsto f_u$ be an element of $C(U;\mathcal{M}(\XX,A))$.  Then the following are equivalent:
\begin{itemize}
    \item[(i)] (Coboundary) $f$ lies in $d_U \mathcal{M}(\XX,A)$, that is to say there exists $F \in \mathcal{M}(\XX,A)$ such that $f = d_U F$.
    \item[(ii)]  (Cocycle) $f$ lies in $Z^1(U; \mathcal{M}(\XX,A))$, that is to say one has the $U$-cocycle equation
\begin{equation}\label{cocycle-u}
f_{u+v} = f_u + V_u f_v
\end{equation}
for all $u,v \in U$.
\end{itemize}
In fact, $\mathcal{M}(\XX,A)$ has trivial $U$-cohomology, in the sense that we have the short exact sequence
\begin{equation}\label{c8-seq}
 0 \to \mathcal{M}(\XX,A)^U \to \mathcal{M}(\XX,A) \stackrel{d_U}{\to} Z^1(U; \mathcal{M}(\XX,A)) \to 0
 \end{equation}
of locally compact abelian groups, where $Z^1(U; \mathcal{M}(\XX,A))$ is equipped with the topology inherited from $\mathcal{M}(\XX\times U,A).$
\end{lemma}

Most of this lemma is already contained in \cite[Lemma C.8]{host2005nonconventional}.  One technical point here is that we require the sequence \eqref{c8-seq} to be short exact in the category of locally compact abelian groups, not just abelian groups, so in particular the map $d_U$ here is required to be open; see Appendix \ref{filteredcategory}.

\begin{proof} A short calculation shows that (i) implies (ii).  In the converse direction, express $\XX = \YY \times_\rho U$ as above, and then set \begin{equation}\label{antiderivativef}F(y,u) \coloneqq f_u(y,u_0)\end{equation} for some point $u_0 \in U$.  A direct computation then shows that for almost all $u \in U$, one has $f_u = (d_U F)_u = \partial_{u} F$.

Clearly, $\mathcal{M}(\XX,A)^U$ is a closed subgroup of $\mathcal{M}(\XX,A)$. To complete the proof it is left to show that $d_U$ is an open map. Let $f\in \mathcal{M}(\XX,A)$ and $\varepsilon>0$ be arbitrary, the topology on $\mathcal{M}(\XX,A)$ is generated by sets of the form 
$$B_\varepsilon(f) \coloneqq \{g\in \mathcal{M}(\XX,A) : \mu(\{x\in X : |f(x)-g(x)|>\varepsilon\})<\varepsilon\}$$ where $|f(x)-g(x)|$ is the distance between $f(x)$ and $g(x)$ in $A$. We now prove that $d_U(B_\varepsilon(f))$ is open in $Z^1(U;\mathcal{M}(\XX,A))$. Let $g\in B_{\varepsilon}(f)$ be arbitrary, then there exists $\varepsilon_1<\varepsilon$ such that 
\begin{equation}\label{assumption}
    \mu(\{x\in X : |f(x)-g(x)|>\varepsilon_1\})<\varepsilon_1.
\end{equation}
Choose $\varepsilon_2 = \varepsilon-\varepsilon_1$, and let $(h_u)_{u\in U}\in Z^1(U;\mathcal{M}(\XX,A))$ be such that 
\begin{equation}\label{closetog}
    \mu(\{(u,x)\in U\times X : |\partial_u g(x)-h_u(x)|>\varepsilon_2\})<\varepsilon_2.
\end{equation}
We need to show that there exists some $h\in B_\varepsilon(f)$ such that $(h_u)_{u\in U} = d_U h$. Since $u\mapsto \partial_u g-h_u$ is a cocycle, we can define $F$ as in \eqref{antiderivativef} such that $d_U F = d_U g - (h_u)_{u\in U}$ and furthermore from the construction and \eqref{closetog} we see that $F$ is $\varepsilon_2$-close to $0$ in measure. Now set $h\coloneqq F-g$, then clearly $d_U h = (h_u)_{u\in U}$ and from the construction, $h$ is $\varepsilon_2$-close to $g$ in measure. Combining this with \eqref{assumption} using the triangle inequality we see that $h$ is $\varepsilon$-close to $f$ in measure, or equivalently $h\in B_\varepsilon(f)$ as required. 
\end{proof}
Informally, this lemma asserts that one can ``integrate'' any $U$-cocycle $(f_u)_{u \in U}$ on a compact group acting freely to obtain an ``antiderivative'' $F$.  However, in our applications, we will frequently want to also require $F$ to be a polynomial of degree at most $d+1$ for some $d \geq -1$.  Thus, we are interested in classifying the $U$-cocycles $(f_u)_{u \in U}$ that are of the form $d_U F$ for some $F \in \Poly_{\leq d+1}(\XX,A)$.

Now that we are no longer necessarily working with the Host--Kra tower, it turns out that different elements of $U$ can act with a different ``weight'', even if one restricts attention to non-zero elements.  We now introduce some algebraic notation to handle weights on the acting group $U$, and their impact on ``polynomiality'' on a $U$-group $A$.

\begin{definition}[Weighted actions and cocycles]\label{weighted-cocycle}  Let $U$ be a compact abelian group.  A \emph{weight filtration} on $U$ is a nested sequence of compact subgroups
$$ U = U_{>0} \geq U_{>1} \geq \dots$$
of $U$.  The \emph{weight} $\weight(u) \in [1,+\infty]$ of an element $u$ of $U$ is then defined as
$$ \weight(u) \coloneqq \sup \{ i+1 : u \in U_{>i} \},$$
thus elements of $U_{>i}$ have weight at least $i+1$, and $0$ has weight $+\infty$.  

If $U$ has a weight filtration, a \emph{polynomial filtration} on a $U$-group $A$ compatible with that weight filtration is a nested sequence of $U$-subgroups
$$ \Poly_{\leq -\infty}[A] \leq \dots \leq \Poly_{\leq -1}[A] \leq \Poly_{\leq 0}[A] \leq \Poly_{\leq 1}[A] \leq \dots \leq A$$
with the property that
\begin{equation}\label{weight-lower}
 \partial_{u} \Poly_{\leq d}[A] \leq \Poly_{\leq d - \weight(u)}[A] 
\end{equation}
for all $d \in \Z$ and $u \in U$.  In particular, if $P \in \Poly_{\leq d}[A]$ and $u \in U_{> i}$, then $\partial_{u} P \in \Poly_{\leq d - i - 1}[A]$.  Informally: if $u$ is of weight $k$, then $\partial_{u}$ behaves like a differential operator of order $\leq k$.

For any $d \in \Z$, we define $\Poly^1_{\leq d-\weight}[U;A] \leq Z^1(U;A)$ to be the collection of cocycles $(a_u)_{u \in U} \in Z^1(U;A)$ obeying the additional condition
\begin{equation}\label{d-weight}
a_u \in \Poly_{\leq d - \weight(u)}[A]
\end{equation}
for all $u \in U$.  In particular, $a_u \in \Poly_{\leq d-1}[A]$ for all $u \in U$, and one can strengthen this to $a_u \in \Poly_{\leq d-i-1}[A]$ whenever $u \in U_{>i}$ for some $i \geq 0$.  By construction, one has the sequence
$$ 0 \to \Poly_{\leq d}[A]^U \to \Poly_{\leq d}[A] \stackrel{d_U}{\to} \Poly_{\leq d-\weight}^1[U; A]$$
exact at the first two terms. 
Informally, elements of $\Poly_{\leq d-\weight}^1[U; A]$ are like ``virtual derivatives'' of ``virtual polynomials of degree $\leq d$''. 
\end{definition}

The structure group $U$ of an ergodic abelian extension $\XX = \YY \times_\rho U$ comes with a natural weight filtration, which we call the \emph{type filtration}: if $u \in U$ and $i \geq 0$, we say that $u$ is \emph{of weight $>i$} if the action of $u$ fixes the factor $\ZZ^{\leq i}(\XX)$, or equivalently the derivative operator $\partial_{u}$ annihilates $\mathcal{M}(\ZZ^{\leq i}(\XX),A)$ for any compact abelian group $A$.  The set of all $u$ of weight $>i$ will be denoted $U_{>i}$; this clearly forms a weight filtration (recall from ergodicity that $\ZZ^{\leq 0}(\XX)$ is trivial). One can think of $U_{>i}$ as a sort of ``orthogonal complement'' to $\ZZ^{\leq i}(\XX)$.

\begin{remark}  At each stage $\ZZ^{\leq k}(\XX) = \ZZ^{\leq k-1}(\XX) \times_{\rho_k} U_k$ of the Host--Kra tower, the type filtration $(U_{k,>i})_{i=0}^\infty$ is simply given by $U_{k,>i}=U_k$ for $i < k$ and $U_{k,>i}=0$ for $i \geq k$, such that all non-zero elements of $U_k$ have the same weight $k$.
But, as the example in \Cref{analysis-sec} shows, we cannot always assume that we are working with the Host--Kra tower, and so are forced to consider more general filtrations.  For instance, if we consider the extension
\begin{center}
\begin{tikzcd}
    \XX_{(k)} \arrow[d, Rightarrow, "\rho \mod 2^k; \Z/2^k\Z"]  \\ 
    \XX_{(0)} 
\end{tikzcd}
\end{center}
from that example, then one sees from \Cref{facts} that the weight of an element of $\Z/2^k \Z$ is one plus the number of times $2$ can divide into the element; thus $(\Z/2^k \Z)_{>i} = 2^i\Z/2^k\Z$ for $i \leq k$ and $(\Z/2^k \Z)_{>i}=0$ for $i \geq k$.
\end{remark}

If $u_{>i} \in U_{>i}$, then by \Cref{order-prop}(i) $\partial_{{u_{>i}}}$ annihilates $\Poly_{\leq i}(\XX,A)$ for any compact abelian group $A$.  
From this we see that $\partial_{{u_{>i}}}$ acts like a differential operator of order $>i$, in the sense that we have homomorphisms
\begin{equation}\label{diff-incl}
\partial_{{u_{>i}}} \colon \Poly_{\leq d}(\XX,A) \to \Poly_{\leq d-i-1}(\XX,A)
\end{equation}
for any integer $d$, as can be seen by a routine induction using $d \leq i$ as the base case.  Thus, we see that setting
$$ \Poly_{\leq d}[\mathcal{M}(\XX,A)] \coloneqq \Poly_{\leq d}(\XX,A),$$
with the convention that $\Poly_{\leq -\infty}[\mathcal{M}(\XX,A)] = 0$, will give a polynomial filtration on $\mathcal{M}(\XX,A)$ compatible with the type filtration on $U$.  Since $\YY$ is equivalent to the $U$-invariant factor of $\XX$, we now have the exact sequence
$$ 0 \to \Poly_{\leq d}(\YY,A) \to \Poly_{\leq d}(\XX,A) \stackrel{d_U}{\to} \Poly_{\leq d-\weight}[U; \mathcal{M}(\XX,A)]$$
for any integer $d$.

In particular we now see that a necessary condition to have $f = d_U F$ for some $F \in \Poly_{\leq d}(\XX,A)$ is that
\begin{equation}\label{i-cons}
    f \in \Poly_{\leq d-\weight}[U; \mathcal{M}(\XX,A)]
\end{equation}
or equivalently that
$$    
    f_{u_{>i}} \in \Poly_{\leq d-i-1}(\XX,A)
$$
for all $i \geq 0$ and $u_{>i} \in U_{>i}$.

The main objective of this section is to show that these necessary conditions are in fact sufficient in the case where the cocycle $\rho$ one can associate to the $U$ action is exact in the sense of \Cref{key-three}. We begin with the following technical algebraic calculation:

\begin{proposition}[Polynomial degree calculation]\label{poly-calc}  Suppose one has an abelian extension $\XX=\YY \times_\rho U$ of ergodic $\Gamma$-systems, and let $A$ be a closed subgroup of $\T$. We give $U$ the type filtration and $\mathcal{M}(\XX)$ the polynomial filtration.  Let $d,\ell_1,\dots,\ell_s \geq 0$ for some $s \geq 1$. Suppose one has the following objects:
\begin{itemize}
    \item[(i)] A function $F \in \mathcal{M}(\XX,A)$ that is ``virtually of degree $\leq d$`` in the sense that $d_U F \in \Poly^1_{\leq d-\weight}[U; \mathcal{M}(\XX)]$, or equivalently that
    $$ \partial_{{u}} F \in \Poly_{\leq d-i-1}(\XX)$$
for all $i \geq 0$ and $u \in U_{>i}$;
    \item[(ii)] A map $q \in \mathcal{M}(\XX,U)$ such that $q \mod U_{>i} \in \Poly_{\leq i}(\XX,U/U_{>i})$ for all $i \geq 0$; and
    \item[(iii)]  For each $1 \leq j \leq s$, a map $r_j \in \mathcal{M}(\XX,U)$ such that $r_j \mod U_{>i} \in \Poly_{\leq i-\ell_j}(\XX,U/U_{>i})$ for all $i \geq 0$.
\end{itemize}
Then the function 
$$ g(x) \coloneqq \partial_{{r_1(x)}} \dots \partial_{{r_s(x)}} F( V_{q(x)} x ) = \sum_{\omega \in \{0,1\}^s} (-1)^\omega F( V_{q(x)+ \sum_{j=1}^s \omega_j r_j(x)} x )$$ 
lies in $\Poly_{\leq d-\sum_{j=1}^s \ell_j}(\XX)$.
\end{proposition}

\begin{proof}
In order to avoid measure-theoretical technicalities we first reduce matters to the case where $U$ is a finite group. From \Cref{ppfacts}(i) and the continuity of $u\mapsto \partial_u F$ we see that we can find an open neighborhood $U'\subseteq U$ such that $\partial_u F$ is a constant on $u\in U'$. In \cite[Theorem 1.4]{jst-tdsystems} we proved that $U$ is totally disconnected and so we may assume without loss of generality that $U'$ is an open subgroup. Write $\partial_u F = \xi(u)$ the cocycle identity shows that $\xi \colon  U'\rightarrow \mathbb{T}$ is a character. Again, since $U$ is totally disconnected so is $U'$ and so $\ker \xi$ is an open subgroup. Thus, $\partial_u F = 0$ for all $u\in \ker \xi$ and so we may quotient out by $\ker \xi$ and assume without loss of generality that $U$ is finite.

Set $m \coloneqq d-\sum_{j=1}^s \ell_j$.  First suppose that $m$ is negative.  By assumption (iii), for any $j=1,\dots,s$, $r_j$ vanishes modulo $U_{>\ell_j-1}$, and thus takes values in $U_{>\ell_j-1}$ (with the convention that $U_{>-1}=U_{>0}$).  For every $x_0 \in \XX$, we conclude from one application of assumption (i) and $s-1$ applications of \eqref{diff-incl} that
$$ \partial_{{r_1(x_0)}} \dots \partial_{{r_s(x_0)}} F \in \Poly_{\leq d-\sum_{j=1}^s \ell_j}(\XX) = 0$$
since $m$ is negative, and so $g$ vanishes identically, giving the claim in this case.

Now suppose that inductively $m$ is non-negative, and that the claim has already been proven for $m-1$.  It will then suffice to show that
$$ \partial_\gamma g \in \Poly_{\leq m-1}(\XX)$$
for every $\gamma \in \Gamma$.   If $d=0$ then all derivatives $\partial_{u} F$ of $F$ vanish, and the claim is trivial; so we will assume that $d \geq 1$.

Morally speaking, one expects to be able to expand the ``discrete derivative'' $\partial_\gamma g$ by some sort of ``discrete chain rule'', and this is precisely what we shall now attempt. Fix $\gamma$, and pick a point $x$ in $\XX$.  Consider the quantity
$$ T^\gamma g(x) = \partial_{V_{T^\gamma r_1(x)}} \dots \partial_{ V_{ T^\gamma r_s(x)}} F( V_{T^\gamma q(x)} T^\gamma x ).$$
Since $T^\gamma$ commutes with $V_{T^\gamma q(x)}$, one can write this as the sum of
\begin{equation}\label{term-1}
\partial_{V_{T^\gamma r_1(x)}} \dots \partial_{ V_{ T^\gamma r_s(x)}} (\partial_\gamma F)( V_{T^\gamma q(x)} x )
\end{equation}
and
$$ \partial_{V_{T^\gamma r_1(x)}} \dots \partial_{ V_{ T^\gamma r_s(x)}} F( V_{T^\gamma q(x)} x ).$$
Writing $T^\gamma q(x) = q(x) + \partial_\gamma q(x)$, one can write the latter term as the sum of
\begin{equation}\label{term-2}
\partial_{V_{T^\gamma r_1(x)}} \dots \partial_{ V_{ T^\gamma r_s(x)}} \partial_{{\partial_\gamma q(x)}} F( V_{q(x)} x )
\end{equation}
and
$$\partial_{V_{T^\gamma r_1(x)}} \dots \partial_{ V_{ T^\gamma r_s(x)}} F( V_{q(x)} x ).$$
If one successively applies the cocycle identity $\partial_{{T^\gamma r_j(x)}} = \partial_{{\partial_\gamma r_j(x)}} V_{r_j(x)} + \partial_{{r_j(x)}}$ for $j=1,\dots,s$, one can express the latter term as the sum of
\begin{equation}\label{term-3}
\partial_{{r_1(x)}} \dots \partial_{{r_{j-1}(x)}} \partial_{{\partial_\gamma r_j(x)}} \partial_{{r_{j+1}(x)}} \dots \partial_{ V_{ T^\gamma r_s(x)}} F( V_{q(x)+r_j(x)} x )
\end{equation}
for $j=1,\dots,s$, and
$$ \partial_{{r_1(x)}} \dots \partial_{{r_s(x)}} F( V_{q(x)} x ).$$
The latter expression is of course just $g(x)$.  We conclude that $\partial_\gamma g$ can be expressed as the sum of \eqref{term-1}, \eqref{term-2}, and the $s$ terms \eqref{term-3} (this is the aforementioned ``discrete chain rule'').  It therefore suffices to show that each of these terms is a polynomial of degree at most $m-1$.

The expression \eqref{term-1} is of the same form as $g$, but with $F$ replaced by $\partial_\gamma F$.  Observe that $\partial_\gamma F$ obeys the same hypotheses as $F$, but with $d$ replaced by $d-1$, thus effectively lowering the quantity $m = d - \sum_{j=1}^s \ell_j$ by one; and so the claim for this term follows from the induction hypothesis.

The expression \eqref{term-2} is also of the same form as $g$, but with a new function $r_{s+1} \coloneqq \partial_\gamma q$ added to the collection $r_1,\dots,r_s$.  From assumption (ii) we see that $\partial_\gamma q$ obeys assumption (iii) with $\ell_{s+1}=1$, thus again effectively lowering
$m = d - \sum_{j=1}^s \ell_j$ by one; and so the claim again follows from the induction hypothesis.

Finally, any term of the form \eqref{term-3} is also of the same form as $g$, but with some of the $r_{j'}$ translated (which does not impact $\ell_{j'}$), one of the $r_j$ replaced by a derivative $\partial_\gamma r_j$ (which increases $\ell_j$ by one), and $q$ shifted by $r_j$ (which does not affect assumption (ii), thanks to assumption (iii)).  Thus again $m = d - \sum_{j=1}^s \ell_j$ has effectively been lowered by one, and the claim again follows from the induction hypothesis.
\end{proof}

\begin{remark}  One can use this proposition to obtain a new proof of \cite[Lemma 8.14]{btz}, after setting $s$ equal to one and $F$ equal to the potential function provided by \Cref{HK-C8}; we leave the details to the interested reader.  Actually, this new proof repairs a gap in the original proof given in \cite{btz}; the difficulty there being that the derivation of part (ii) of that proposition from part (i) was unjustified, since the map $(y,u) \mapsto p_t(y,uq(y,u))$ is not necessarily a cocycle.  Roughly speaking, this corresponded to omitting a treatment of the term \eqref{term-2} in the proof above, which is also the main reason why the proposition had to study higher derivatives of $F$ and not just first derivatives.
\end{remark}

Next, we recall an alternate description of the type filtration, after representing $\XX$ as a skew product $\YY \times_\rho U$ as above.

\begin{lemma}[Duality between $U_{>i}$ and $Z^1_{\leq i}$]\label{u-dual}  Suppose one has an abelian extension $\XX=\YY \times_\rho U$ of ergodic $\Gamma$-systems. Let $U_\bullet = (U_{>i})_{i=0}^\infty$ be the type filtration.  Then for each $i$, $U_{>i}$ is the smallest (necessarily compact) subgroup of $U$ such that $\rho \mod U_{>i} \in Z^1_{\leq i}(\Gamma,\YY,U/U_{>i})$; in particular $\rho$ is of type $\leq i$ if and only if $U_{>i}=0$.   Equivalently (by Pontryagin duality and \Cref{type}(ii)), one has
    $$ U_{>i} = \mathrm{Ann} \left\{ \xi \in \hat U : \xi \circ \rho \in Z^1_{\leq i}(\Gamma,\YY)\right\}.$$
\end{lemma}

\begin{proof}  See \cite[Proposition 7.6]{host2005nonconventional}; the proof there is stated in the case $\Gamma=\Z$, but extends to arbitrary countable abelian groups without difficulty.
\end{proof}

Now we can give our main polynomial integration result, which highlights the useful role of exactness for a cocycle.

\begin{theorem}[Polynomial integration lemma]\label{poly-integ}  Suppose one has an abelian extension $\XX=\YY \times_\rho U$ of ergodic $\Gamma$-systems with $\rho$ exact. We give $U$ the type filtration and $\mathcal{M}(\XX)$ the polynomial filtration.  Let $f \colon u \mapsto f_u$ be an element of $C(U; \mathcal{M}(\XX))$. Then for any $k \geq 0$, the following are equivalent:
\begin{itemize}
    \item[(i)] (Coboundary of polynomial) $f = d_U F$ for some $F \in \Poly_{\leq k}(\XX)$.
    \item[(ii)]  (Polynomial cocycle) $f$ lies in $\Poly_{\leq k-\weight}^1[U; \mathcal{M}(\XX)]$; that is to say, it obeys the $U$-cocycle condition \eqref{cocycle-u} for all $u,v \in U$, as well as the degree condition \eqref{i-cons} for all $i \geq 0$ and $u \in U_{>i}$.
\end{itemize}
More compactly, one has the short exact sequence
\begin{equation}\label{dUses} 0 \to \Poly_{\leq k}(\YY) \to \Poly_{\leq k}(\XX) \stackrel{d_U}{\to} \Poly_{\leq k-\weight}^1[U; \mathcal{M}(\XX)] \to 0.
\end{equation}
Furthermore, this short exact sequence splits in the category of $k$-filtered groups, where the groups are equipped with the polynomial filtration introduced in Example \ref{polynomialfiltration}.
\end{theorem}

\begin{proof}  The implication of (ii) from (i) has already been established.  Now suppose that (ii) holds. Then by \Cref{HK-C8} we have $f = d_U F$ for some $F \in \mathcal{M}(\XX)$.  By \eqref{i-cons} we see that $F$ obeys assumption (i) of \Cref{poly-calc}.  Now set $r(y,u) \coloneqq u - u_0$ and $q(y,u) \coloneqq u_0-u$ for some fixed $u_0 \in U$.  By the exactness hypothesis, we see for every $i$ that $\rho \mod U_{>i} \in \Poly_{\leq i-1}^1(\Gamma,\XX)$.  Since $\partial_\gamma r = \rho_\gamma$ and $\partial_\gamma q = -\rho_\gamma$, we conclude that $q \mod U_{>i}, r \mod U_{>i}$ lie in $\Poly_{\leq i}^1(\Gamma,\XX)$.  By \Cref{poly-calc}, we conclude that the function
$$ g(y,u) \coloneqq \partial_{{u-u_0}} F(V_{u_0-u}(y,u)) = F(y,u) - F(y,u_0)$$
is a polynomial of degree $\leq k$.  But for any $v \in U$, we have $\partial_{v} g = \partial_{v} F = f_v$, giving (i).

Since $\Poly_{\leq k}(\YY)$ is closed in $\Poly_{\leq k}(\XX)$ and $d_U$ is an open map (\Cref{HK-C8}), we see that \eqref{dUses} is a short exact sequence of $k$-filtered locally compact abelian groups. We now establish the splitting of \eqref{dUses}. By \Cref{puresplit} and \Cref{ppfacts}(i), it suffices to prove that $\Poly_{\leq k}(\YY) \leq \Poly_{\leq k}(\XX)$ is $\omega$-pure.  Let $\mathcal{R}\subset \Z^\omega\times\{1,\ldots,k+1\}$ be a countable set of relations and let $P=(P_i)_{i\in\omega}$ be a countable sequence of polynomials in $\Poly_{\leq k}(\XX)$ such that for every $(\vec{m};j)\in \mathcal{R}$ there exists $Q_{(\vec{m};j)}\in \Poly_{\leq k}(\YY)$ satisfying
$$Q_{(\vec{m};j)}\circ \pi - \vec{m}\cdot P\in \Poly_{\leq k-j}(\XX)$$ where $\pi \colon \XX\rightarrow \YY$ is the factor map. Since $Q_{(\vec{m};j)}\circ \pi $ is $U$-invariant, \Cref{ppfacts}(iii) implies that 
\begin{equation}\label{mdup}
    \vec{m}\cdot d_U P \in \Poly^1_{\leq k-j-\weight}(\XX).
\end{equation}
By another application \Cref{ppfacts}(iii), we have for every $i\in\omega$ that 
$$d_U P_i \in \Poly^1_{\leq k-\weight}(\XX).$$ For a fixed $u_0\in U$, define again $r(y,u)\coloneqq u-u_0$, $q(y,u)\coloneqq u_0-u$, and set for each $i\in\omega$, 
$$P'_i(y,u) \coloneqq \partial_{u-u_0} P(V_{u_0-u}(y,u)).$$
Again  by \Cref{poly-calc}, $P'_i$ is a polynomial of degree $\leq k$. Since $\partial_u P_i = \partial_u P'_i$, it follows from \eqref{mdup} that $\vec{m}\cdot P'$, where $P'=(P_i)_{i\in \omega}$, is an element of $\Poly^1_{\leq k-j}(\XX)$. From $U$-invariance it follows that $Q\coloneqq P-P'$ is a sequence of polynomials in $\Poly_{\leq k}(\YY)$ satisfying $$Q_{(\vec{m};j)} - \vec{m}\cdot Q\in \Poly_{\leq k-j}(\YY),$$ and this completes the proof.
\end{proof}


\section{Integration of finite coordinate subgroups: the role of large spectrum}\label{int2-sec}

We have seen how the exactness of a cocycle can enable efficient integration of polynomial cocycles.  In this section, we similarly demonstrate how a large spectrum hypothesis can facilitate the integration of polynomial cocycles on certain finite subgroups of $\Gamma$, specifically the groups
\begin{equation}\label{gamma-l} 
\Gamma_{[\ell]} \coloneqq \bigoplus_{i=1}^\ell \Z/m_\ell\Z \leq \Gamma
\end{equation}
for $\ell \geq 0$; this is a nested sequence
$$ 0 = \Gamma_{[0]} \leq \Gamma_{[1]} \leq \dots \leq \Gamma$$
of finite subgroups of $\Gamma$ that exhaust $\Gamma$ in the sense that
\begin{equation}\label{exhaust}
\varinjlim \Gamma_{[\ell]} = \bigcup_{l=0}^\infty \Gamma_{[\ell]} = \Gamma.
\end{equation}
In particular, the $\Gamma_{[\ell]}$ form a F{\o}lner sequence for $\Gamma$.  We also have the splitting $\Gamma = \Gamma_{[\ell]} \times \Gamma_{[\ell]}^\perp$ where $\Gamma_{[\ell]}^\perp$ is the finite index subgroup
\begin{equation}\label{gamma-ge} 
\Gamma_{[\ell]}^\perp\coloneqq \bigoplus_{i>\ell} \Z/m_\ell\Z \leq \Gamma
\end{equation}

Before we begin integrating on these subgroups, we first pause to ensure that systems with large spectrum actually exist. It is clear that any (ergodic) extension of a system with large spectrum, also has large spectrum.  We also have a basic example:

\begin{example}[Standard rotational system]\label{standard-ex}  Let $\iota \colon \Gamma \to \tilde \Gamma$ be the standard (but non-canonical) isomorphism
$$ \iota \colon (\gamma_i)_{i \in \N} \mapsto \left(\frac{\gamma_i}{m_i} \right)_{i \in \N}$$
between $\Gamma$ and the dense countable subgroup $\tilde \Gamma$.  This gives a translational system
\begin{center}
\begin{tikzcd}
    \hat \Gamma\arrow[d, Rightarrow, "\iota; \hat \Gamma"]  \\ 
    \operatorname{pt}  
\end{tikzcd}
\end{center}
given explicitly by $T^\gamma(\xi) \coloneqq \xi + \iota(\gamma)$.  As the homomorphism $\iota \colon \Gamma \to \hat \Gamma$ has dense image, this is an ergodic system; it is clearly of order $\leq 1$.  If $\gamma_0 \in \Gamma$, one can easily check that if we define the  function $\phi_{\iota(\gamma_0)} \in \mathcal{M}(\hat \Gamma,\T)$ by
$$ \phi_{\iota(\gamma_0)}(\eta) \coloneqq \eta \cdot \gamma_0$$
then for any $\gamma \in \Gamma$ one has
$$ \partial_\gamma \phi_{\iota(\gamma_0)}(\eta) = \iota(\gamma) \cdot \gamma_0 = \iota(\gamma_0) \cdot \gamma$$
and so $\phi_{\iota(\gamma_0)}$ is an eigenfunction of the system with eigenvalue $\iota(\gamma_0)$. 
\end{example}

Using this example, we can easily extend systems to have large spectrum:

\begin{lemma}\label{large-spectrum-exist}  Let $k \geq 1$.  Then any ergodic $\Gamma$-system $\XX$ of order $\leq k$ has an abelian extension $\XX'$ which is also ergodic of order $\leq k$, and has large spectrum.
\end{lemma}

\begin{proof}  The direct product $\XX \times \hat \Gamma$ of $\XX$ with the standard system in \Cref{standard-ex} is a $\Gamma$-system which contains both $\XX$ and $\hat \Gamma$ as factors.  It need not be ergodic; but (almost) any ergodic component $\XX'$ will also extend both $\XX$ and $\hat \Gamma$, as both factors are ergodic (i.e., these components are joinings of $\XX$ and $\hat \Gamma$).  Also, each such component is an abelian extension of $\XX$ by \Cref{zimmer}. As $\XX$ is of order $\leq k$ and $\hat \Gamma$ of order $\leq 1$, hence order $\leq k$, it is easy to check that $\XX \times \hat \Gamma$, as well as almost all of its ergodic components, are also of order $\leq k$.  Thus a generic ergodic component will yield the desired extension.
\end{proof}

The presence of large spectrum allows us to ``trivialize'' the behavior of any finite coordinate subgroup:

\begin{proposition}[Trivializing finite coordinate subgroups]\label{fin-coord}  Let $\XX$ be an ergodic $\Gamma$-system with large spectrum, and let $\ell \geq 0$.Then, up to isomorphism, we can express $\XX$ as the direct sum $\YY \oplus \Gamma_{[\ell]}$ of an ergodic $\Gamma_{{[\ell]}}^\perp$-system $\YY$ and the finite $\Gamma_{[\ell]}$-system $\Gamma_{[\ell]}$ with the regular translation action, in the sense that the action of $\Gamma$ on $\YY \oplus \Gamma_{[\ell]}$ takes the form
$$ T^{\gamma_{[\ell]} + \gamma_{[\ell]}^\perp} (y, \sigma_{[\ell]}) = (T^{\gamma_{[\ell]}^\perp} y, \sigma_{[\ell]} + \gamma_{[\ell]})$$
for almost all $y \in \YY$, $\gamma_{[\ell]}, \sigma_{[\ell]} \in \Gamma_{[\ell]}$, and $\gamma_{[\ell]}^\perp \in \Gamma_{[\ell]}^\perp$.
\end{proposition}

\begin{proof}  The action of $\Gamma_{[\ell]}$ on $\XX$ is free, since for any $\gamma_{[\ell]} \in \Gamma_{[\ell]}$ one can use the large spectrum hypothesis to locate an eigenfunction $\phi \in \mathcal{M}(\XX)$ whose eigenvalue $\partial_{\gamma_{[\ell]}} \phi$ at $\gamma_{[\ell]}$ is non-trivial.  Thus, by 
\cite[Theorem 3.29]{glasner2015ergodic} one can write $\XX$ up to isomorphism as $\XX = \YY \times_\rho \Gamma_{[\ell]}$ for some cocycle $\rho \in Z^1(\Gamma,\YY,\Gamma_{[\ell]})$.  By comparing the cocycle action with the vertical action of $\Gamma_{[\ell]}$, we see that $\Gamma_{[\ell]}$ must act trivially on $\YY$ and the cocycle $\rho_{\gamma_{[\ell]}}$ must vanish for all $\gamma_{[\ell]} \in \Gamma_{[\ell]}$; this means that the action takes the form
$$ T^{\gamma_{[\ell]} + \gamma_{[\ell]}^\perp} (y, \sigma_{[\ell]}) = (T^{\gamma_{[\ell]}^\perp} y, \sigma_{[\ell]} + \rho_{\Gamma_{[\ell]}^\perp}(y) + \gamma_{[\ell]})$$
where $\rho$ is now viewed as a cocycle in $Z^1(\Gamma_{[\ell]}^\perp, \YY, \Gamma_{[\ell]})$.  To finish the proof, it will suffice to show that this cocycle is a coboundary in $d_{\Gamma_{[\ell]}^\perp} \mathcal{M}(\YY, \Gamma_{[\ell]})$.  If we write $\rho$ in components as $(\rho_i)_{i=1}^\ell$ with $\rho_i \in Z^1(\Gamma_{[\ell]}^\perp, \YY, \Z/m_i\Z)$, it suffices to show that each $\rho_i$ is a coboundary in $d_{\Gamma_{[\ell]}^\perp} \mathcal{M}(\YY, \Z/m_i\Z)$ for each $i=1,\dots,\ell$.

Fix $i$. By the large spectrum hypothesis, one can find an eigenfunction $\phi \in \mathcal{M}(\XX,\T)$ such that $\partial_\gamma \phi = \frac{\gamma_i }{m_i} \mod 1$ for all $\gamma = (\gamma_{i'})_{i' \in \N}$ in $\Gamma$.  Then $m_i \phi$ is invariant and thus constant by ergodicity; by subtracting a constant we may assume without loss of generality that $m_i \phi = 0$, thus $\phi$ now takes values in $\frac{1}{m_i}\Z/\Z$. Specializing the eigenfunction equation to shifts $\gamma$ in $\Gamma_{[\ell]}$, we conclude that $\phi$ must take the form
$$ \phi(y, \sigma_{[\ell]}) = \underline{\phi}(y) + \frac{\gamma_i \sigma_i }{m_i} \mod 1$$
for some $\underline{\phi} \in \mathcal{M}(\YY,\frac{1}{m_i}\Z/\Z)$. If we now test the eigenfunction equation on $\Gamma_{[\ell]}^\perp$, we conclude that
$$ 0 = \partial_{\Gamma_{[\ell]}^\perp} \phi(y,\sigma) = \partial_{\Gamma_{[\ell]}^\perp} \underline{\phi}(y) + \frac{(\rho_{\gamma_{[\ell]}^\perp})_i}{m_i} \mod 1$$
and thus
$$ (\rho_{\Gamma_{[\ell]}^\perp})_i =  \partial_{\Gamma_{[\ell]}^\perp} (-m_i\underline{\phi}(y)).$$
Thus $\rho_i$ is a coboundary in $d_{\Gamma_{[\ell]}^\perp} \mathcal{M}(\YY, \Z/m_i\Z)$  as required.
\end{proof}

As a corollary of this decomposition, we obtain an integration result on finite coordinate subgroups.

\begin{theorem}[Integration on finite coordinate subgroups]\label{integ-finite}  Let $\XX$ be an ergodic $\Gamma$-system with large spectrum, and $\ell \geq 0$. Let $f \in C(\Gamma_{[\ell]};\mathcal{M}(\XX))$ be a function, and let $k \geq 0$.  Then the following are equivalent:
\begin{itemize}
    \item[(i)] (Coboundary of polynomial) There exists $F \in \Poly_{\leq k}(\XX)$ such that $f_{\gamma_{[\ell]}} = \partial_{\gamma_{[\ell]}} F$ for all $\gamma_{[\ell]} \in \Gamma_{[\ell]}$.
    \item[(ii)]  (Polynomial cocycle) One has the $\Gamma_{[\ell]}$-cocycle equation
\begin{equation}\label{cocycle-fin}
f_{\gamma_{[\ell]} + \gamma'_I} = f_{\gamma_{[\ell]}} + T^{\gamma'_I} f_{\gamma_{I'}}
\end{equation}
for all $\gamma_{[\ell]}, \gamma'_I \in \Gamma_{[\ell]}$, and $f_{\gamma_{[\ell]}} \in \Poly_{\leq k-1}(\XX)$ for all $\gamma_{[\ell]} \in \Gamma_{[\ell]}$.
\end{itemize}
More compactly, $\Poly_{\leq k-1}(\XX)$ has trivial $\Gamma_{[\ell]}$-cohomology, and one has the short exact sequence
\begin{equation}\label{finite-short}
0 \to \Poly_{\leq k}(\XX)^{\Gamma_{[\ell]}} \to \Poly_{\leq k}(\XX) \stackrel{d_{\Gamma_{[\ell]}}}{\to} Z^1[\Gamma_{[\ell]}; \Poly_{\leq k-1}(\XX)] \to 0.
\end{equation}
Furthermore, \eqref{finite-short} splits in the category of $k$-filtered locally compact abelian groups (see Appendix \ref{filteredcategory}).
\end{theorem}

\begin{proof} The implication of (ii) from (i) is clear.  Now suppose that (ii) holds.  By \Cref{fin-coord}, we may assume without loss of generality that $\XX$ is of the form $\YY \oplus \Gamma_{[\ell]}$.  We set $F \in \mathcal{M}(\XX,A)$ to be the function \begin{equation}\label{antiderivative}
    F(y,\gamma_{[\ell]}) \coloneqq f_{\gamma_{[\ell]}}(y,0),
\end{equation} then we have $f_{\gamma_{[\ell]}} = \partial_{\gamma_{[\ell]}} F$ from the cocycle equation.  In particular any $k+1$-fold derivative of $F$ will vanish if one of the derivatives is along a direction in $\Gamma_{[\ell]}$.  On the other hand, since $f_{\gamma_{[\ell]}}$ is of degree $\leq k-1$ and the $\Gamma_{[\ell]}^\perp$ action is trivial in the vertical direction, any $k$-fold derivative of $F$ will vanish if all of the derivatives lie along $\Gamma_{[\ell]}^\perp$.  By repeated use of the identity $\partial_{a+b} = \partial_a + \partial_b + \partial_a \partial_b$ we conclude that any $k+1$-fold derivative of $F$ in any direction in $\Gamma = \Gamma_{[\ell]} \oplus \Gamma_{[\ell]}^\perp$ will vanish, and the claim follows\footnote{One could also have proceeded via \Cref{poly-calc}, although this is somewhat of an overkill here since the system is so simple.}.

Since $\Poly_{\leq k}(\XX)^{\Gamma_{[\ell]}}$ is closed in $\Poly_{\leq k}(\XX)$ and the kernel of $d_{\Gamma_{[\ell]}}$ contains $\Poly_{\leq 0}(\XX)$, \Cref{ppfacts} implies that $d_{\Gamma_{[\ell]}}$ is an open map and \eqref{finite-short} is therefore a short exact sequence of filtered locally compact abelian groups. Now we show that the sequence \eqref{finite-short} splits.
By \Cref{puresplit}, it suffices to show that $\Poly_{\leq k}(\XX)^{\Gamma_{[\ell]}}$ is a $\omega$-pure subgroup of $\Poly_{\leq k}(\XX)$.

Let $\mathcal{R}\subset \Z^\omega\times\{1,\ldots,k+1\}$ be a countable set of relations and let $P=(P_i)_{i\in\mathbb{N}}$ be a countable sequence of polynomials in $\Poly_{\leq k}(\XX)$ be such that for every $(\vec{m};j)\in \mathcal{R}$ there exists $Q_{(\vec{m};j)}\in \Poly_{\leq k}(\XX)^{\Gamma_{[\ell]}}$ satisfying
$$Q_{(\vec{m};j)} - \vec{m}P\in \Poly_{\leq k-j}(\XX).$$ Since $Q_{(\vec{m};j)}$ is $\Gamma_{[\ell]}$-invariant, we have  
$$\vec{m}\cdot d_{\Gamma_{[\ell]}} P \in \Poly_{\leq k-j-\weight}(\XX).$$ Write $\XX=\YY\oplus \Gamma_{[\ell]}$ as in \Cref{fin-coord} and write:
$$P'(y,\gamma) \coloneqq \partial_\gamma P(y,0)$$ for all $\gamma\in \Gamma_{[\ell]}.$
Arguing as we did with \eqref{antiderivative}, we see that that $P'$ is a polynomial of degree $\leq k$, $\vec{m}\cdot P'$ is of degree $\leq k-j$ and $\partial_\gamma P = \partial_\gamma P'$. Thus, $Q=P-P'$ is a sequence of polynomials in $\Poly_{\leq k}(\XX)^{\Gamma_{[\ell]}}$ satisfying that $$Q_{(\vec{m};j)} - \vec{m}\cdot Q\in \Poly_{\leq k-j}(\XX)^{\Gamma_{[\ell]}}$$ as required.
\end{proof}

\subsection{Cocycle integration} 

Our next goal is to combine \Cref{poly-integ} and \Cref{integ-finite} in order to provide a cocycle analogue of \Cref{poly-integ}.  We start with a technical lemma.
\begin{lemma}\label{rXY}
    Let $k\geq 0$, let $\YY$ be an ergodic $\Gamma$-system with large spectrum, and let $\XX=\YY\times_\rho U$ be an ergodic abelian extension of $\YY$ with $\rho$ exact. 
    Then for all $\ell\geq 0$, the short exact sequence
    $$0\rightarrow \Poly_{\leq k}(\YY)\to \Poly_{\leq k}(\XX)\overset{d_U}{\to} \Poly_{\leq k-\weight}^1[U;\mathcal{M}(\XX)]\rightarrow 0$$ 
    splits in the category of $k$-filtered locally compact $\Gamma_{[\ell]}$-groups (all morphisms are also required to be $\Gamma_{[\ell]}$-equivariant). Equivalently, there exists a filtration preserving $\Gamma_{[\ell]}$-equivariant retraction $$r \colon \Poly_{\leq k}(\XX)\rightarrow \Poly_{\leq k}(\YY),$$ or a filtration preserving $\Gamma_{[\ell]}$-equivariant cross-section $$s \colon \Poly_{\leq k-\weight}^1[U;\mathcal{M}(\XX)]\rightarrow \Poly_{\leq k}(\XX).$$
\end{lemma}
\begin{proof}
The equivalences follow from \Cref{equivalenceses}. It therefore suffices to construct the retraction $r$. Composing the retractions from \Cref{poly-integ} and \Cref{integ-finite} we can find a filtration preserving retraction
    $$r_0 \colon \Poly_{\leq k}(\XX)\rightarrow \Poly_{\leq k}(\YY)^{\Gamma_{[\ell]}}.$$
  For each $0\leq i\leq k$, let $\Poly_{\leq k,\leq i}(\XX)$ denote the subgroup of $\Poly_{\leq k}(\XX)$ of polynomials of degree $\leq i$ with respect to the $\Gamma_{[\ell]}$-action, and similarly define $\Poly_{\leq k,\leq i}(\YY)$. 
   By induction on $i$, we prove that there exists a filtration preserving, $\Gamma_{[\ell]}$-equivariant retraction $r_{\leq i} \colon  \Poly_{\leq k,\leq i}(\XX)\rightarrow \Poly_{\leq k,\leq i}(\YY)$. The case $i=k$, then gives the desired retraction $r$. When $i=0$, we let $r_{\leq 0}$ denote the restriction of $r_0$ to $\Poly_{\leq k,\leq 0}(\XX) = \Poly_{\leq k}(\XX)^{\Gamma_{[\ell]}}$, the $\Gamma_{[\ell]}$-equivariance in this case is trivial since the action of this group on the domain and range of $r_{\leq 0}$ is trivial. Let $i\geq 1$ and suppose we have already constructed $r_{\leq i-1}$. Let $Q\in\Poly_{\leq k,\leq i}(\XX)$, then from the induction hypothesis and \Cref{integ-finite} we have $r_{\leq i-1}^{\oplus \Gamma_{[\ell]}}(d_{\Gamma_{[\ell]}}Q) \in Z^1[\Gamma_{[\ell]};\Poly_{\leq k-1}(\YY)]\cong d_{\Gamma_{[\ell]}}\Poly_{\leq k}(\YY)$. \Cref{integ-finite} guarantees the existence of a cross-section to the short exact sequence
   $$0\rightarrow \Poly_{\leq k}(\YY)^{\Gamma_{[\ell]}}\rightarrow \Poly_{\leq k}(\YY)\overset{d_{\Gamma_{[\ell]}}}{\to}  Z^1[\Gamma_{[\ell]};\Poly_{\leq k-1}(\YY)]\rightarrow 0.$$
   Since the restriction of $r_0$ to $\Poly_{\leq k}(\YY)$ is a retraction, we can choose a cross-section $s\colon Z^1[\Gamma_{[\ell]};\Poly_{\leq k-1}(\YY)]\to \Poly_{\leq k}(\YY)$ such that $r_0+s\circ d_{\Gamma_{[\ell]}}$ is the identity map on $\Poly_{\leq k}(\YY)$. Now define
    $$r_{\leq i} = s\circ r_{\leq i-1}^{\oplus \Gamma_{[\ell]}}\circ d_{\Gamma_{[\ell]}} + r_0.$$
    Since $s,r_{\leq i-1}$ and $r_0$ are filtration preserving, so is $r_{\leq i}$. Furthermore, $d_{\Gamma_{[\ell]}} r_{\leq i} = r_{\leq i-1}^{\oplus\Gamma_{[\ell]}}\circ d_{\Gamma_{[\ell]}}$ gives the $\Gamma_{[\ell]}$-equivariance. Thus the image of $r_{\leq i}$ is in $\Poly_{\leq k,\leq i}(\YY)$. Finally, suppose that $Q\in \Poly_{\leq k,\leq j}(\YY)$, we prove by induction on $j\leq i$ that $r_{\leq j}(Q\circ \pi)=Q$, where $\pi \colon \XX\rightarrow \YY$ is the factor map. Indeed, if $j=0$, this follows from the fact that $r_0$ is a retraction. Let $j\geq 1$ and assume inductively the claim holds for all smaller values of $j$, we see that
    $$r_{\leq j}(Q\circ \pi) =  s\circ r_{\leq j-1}^{\oplus \Gamma_{[\ell]}}\circ d_{\Gamma_{[\ell]}}(Q\circ \pi) + r_0(Q\circ \pi) = s\circ d_{\Gamma_{[\ell]}}(Q) + r_0(Q\circ \pi)=Q.$$
\end{proof}
Our main result is the following analogue of \Cref{poly-integ}.
\begin{theorem}[Cocycle integration]\label{cocycle-integ}
 Let $k \geq 0$, let $\YY$ be an ergodic $\Gamma$-system with large spectrum, and let $\XX = \YY \times_\rho U$ be an ergodic abelian extension of $\YY$ with $\rho$ exact. Let $p \in C(U; Z^1(\Gamma,\XX))$.  Observe that $U$ acts on $Z^1(\Gamma,\XX)$ by the diagonal action $u (p_\gamma)_{\gamma \in \Gamma} \coloneqq (u p_\gamma)_{\gamma \in \Gamma}$.  Then the following are equivalent:
\begin{itemize}
    \item[(i)] There exists $q \in \Poly^1_{\leq k}(\Gamma,\XX)$ such that $p_u = \partial_{u} q$ for all $u \in U$.
    \item[(ii)]  $p$ lies in $\Poly^1_{\leq k-\weight}[U; Z^1(\Gamma,\XX)]$, where we give $Z^1(\Gamma,\XX)$ the polynomial filtration $(\Poly^1_{\leq k}(\Gamma,\XX))_{k \in \Z}$.  In other words, $p_u$ obeys the $U$-cocycle equation 
\begin{equation}\label{cocycle-cocycle}
p_{u+v} = p_u + u p_v
\end{equation}
for all $u,v \in U$, and
\begin{equation}\label{deriv-cocycle}
p_{u_{>i}} \in \Poly^1_{\leq k-i-1}(\Gamma,\XX)
\end{equation}
for all $i \geq 0$ and $u_{>i} \in U_{>i}$, where $(U_{>i})_{i=0}^\infty$ is the type filtration of $U$.
\end{itemize}
More compactly, one has the short exact sequence
\begin{equation}\label{GammadU} 0 \to \Poly^1_{\leq k}(\Gamma,\YY) \to \Poly^1_{\leq k}(\Gamma,\XX) \stackrel{d_U}{\to} \Poly^1_{\leq k-\weight}[U; Z^1(\Gamma,\XX)] \to 0.
\end{equation}
Furthermore, this short exact sequence splits in the category of $k$-filtered locally compact abelian groups.
\end{theorem}
\begin{proof}
     The implication of (ii) from (i) follows from \Cref{poly-integ} (noting that any derivative along $U$ preserves cocycle equations).  Now suppose that (ii) holds. Recall the F\o lner sequence \eqref{exhaust} of finite subgroups $\Gamma_{[\ell]}$.  We will recursively construct, for each $\ell$,   
        $Q_{[\ell]} \in \Poly_{\leq k+1}(\XX)$ such that
    \begin{equation}\label{qgam}
    \partial_{u} \partial_\gamma Q_{[\ell]} = (p_u)_\gamma
    \end{equation}
    holds for all $u \in U$ and $\gamma \in \Gamma_{[\ell]}$, and
    \begin{equation}\label{compat}
    \partial_\gamma Q_{[\ell]} = \partial_\gamma Q_{[\ell-1]}
    \end{equation}
    holds for all $\gamma \in \Gamma_{[\ell-1]}$. Assuming this, we see that for all $\gamma\in \Gamma$, we can choose $\ell$ sufficiently large and set $q_\gamma=\partial_\gamma Q_{[\ell]}$. By \eqref{compat} this definition is independent of the choice of $\ell$ for sufficiently large $\ell$ and so $q$ is well defined and a cocycle (because each two $\gamma_1,\gamma_2\in \Gamma$ belong to some $\Gamma_{[\ell]}$ for sufficiently large $\ell$.). From \eqref{qgam} we will then have $\partial_{u} q_\gamma = (p_u)_\gamma$, giving (i). 
    
It remains to construct the $Q_{[\ell]}$.  For $\ell=0$ we can take $Q_{(0)}=0$.  Now suppose that $\ell \geq 1$ and $Q_{[\ell-1]}$ has already been constructed.  Let $e_\ell$ be the $\ell^{\mathrm{th}}$ generator of $\Gamma$; thus $e_\ell$ has order $m_\ell$, and $\Gamma_{[\ell]}$ is the direct sum of $\Gamma_{[\ell-1]}$ and the cyclic group $\langle e_\ell \rangle$ of order $m_\ell$.  This will be a complex diagram chase, starting with a traversal of the diagram in \Cref{initial-diag}.

\begin{figure}[ht]
\centering
\begin{tikzcd}
   &\Poly_{\leq k+1}(\XX) \arrow[d, "d_{\Gamma_{[\ell]}}"] \arrow[dd, "\partial_{e_\ell}"', shift right = 3ex, bend right=50] \\
    & \Poly^1_{\leq d}(\Gamma_{[\ell]},\XX) \arrow[r, "d_U"] \arrow[d, "|_{e_\ell}"] & \Poly^1_{\leq k-\weight}[U; Z^1(\Gamma_{[\ell]},\XX)] \arrow[d, "|_{e_\ell}"]  \\ 
    \Poly_{\leq k}(\YY) \arrow[r, hook] \arrow[d, "d_{\Gamma_{[\ell-1]}}"] &  \Poly_{\leq k}(\XX) \arrow[r, "d_U"] \arrow[d, "d_{\Gamma_{[\ell-1]}}"]  & \Poly^1_{\leq k-\weight}[U; \mathcal{M}(\XX)] \arrow[d, "d_{\Gamma_{[\ell-1]}}"]  \\
    \Poly^1_{\leq k-1}(\Gamma_{[\ell-1]},\YY) \arrow[r, hook] & \Poly^1_{\leq k-1}(\Gamma_{[\ell-1]},\XX) \arrow[r, "d_U"]  & \Poly^1_{\leq k-1-\weight}[U;Z^1(\Gamma_{[\ell-1]},\XX)] 
\end{tikzcd}
    \caption{The initial diagram that we will chase as part of the construction of $Q_{[\ell]}$, where $|_{e_\ell}$ denotes the operation of evaluation at $e_\ell$. The diagram is commutative and horizontally exact, but not vertically exact.  Initially, one can place (a portion of) $p$ in the right group on the second row, and $Q_{[\ell-1]}$ in the top group.}
  \label{initial-diag}
\end{figure} 

By (ii), $p$ lies in $\Poly^1_{\leq k-\weight}[U; Z^1(\Gamma,\XX)]$, and hence so does\footnote{Actually, only the portions of the double cocycle $p = ((p_u)_\gamma)_{u \in U, \gamma \in \Gamma}$ in which the $\gamma$ parameter is restricted to $\Gamma_{[\ell]}$ will be relevant for our analysis.}
\begin{equation}\label{p'-def}
 p' \coloneqq p - d_U d_\Gamma Q_{[\ell-1]}.
 \end{equation}
We can evaluate $p'$ at $e_\ell$ to obtain an element $((p'_u)_{e_\ell})_{u \in U}$ of $\Poly^1_{\leq k-\weight}[U; \mathcal{M}(\XX)]$. By \eqref{qgam}, $p'$ is $\Gamma_{[\ell-1]}$-invariant. Let $s_U^{\ell} \colon  \Poly^1_{\leq k-\weight}[U; \mathcal{M}(\XX)]\rightarrow \Poly_{\leq k}(\XX) $ be a $\Gamma_{[\ell]}$-equivariant cross-section from \Cref{rXY}. Then $Q\coloneqq s_{U}^{\ell}(((p'_u)_{e_\ell})_{u \in U})$ is a $\Gamma_{[\ell-1]}$-invariant polynomial satisfying
$$ d_U Q = ((p'_u)_{e_\ell})_{u \in U}.$$

To get back into the group $\Poly_{\leq k+1}(\XX)$ at the top of \Cref{initial-diag} one would like to ``integrate'' or invert the $\partial_{e_\ell}$ operator at $Q$.  
This will be another diagram chase (see \Cref{final-diag}), involving the cyclic group $\langle e_\ell \rangle$ of order $m_\ell$, and a summation homomorphism $\Sigma \colon B \to B^{\langle e_\ell \rangle}$ defined on any $\Gamma$-group $B$ by the formula
$$ \Sigma \coloneqq \sum_{j=0}^{m_\ell-1} V_{e_\ell}^j.$$

\begin{figure}[ht]
\centering
\begin{tikzcd}
& \Poly_{\leq k+1}(\XX) \arrow[d, "d_{\Gamma_{[\ell]}}"] \arrow[dd, "\partial_{e_\ell}"', shift right = 2.5ex, bend right=50]  \\
  {}  & \Poly^1_{\leq k}(\Gamma_{[\ell]},\XX) \arrow[r, "d_U"] \arrow[d, "|_{e_\ell}"]  & \Poly^1_{\leq k-\weight}[U; Z^1(\Gamma_{[\ell]},\XX)] \arrow[d, "|_{e_{\ell}}"]\\
    \Poly_{\leq k}(\YY) \arrow[r,hook] \arrow[d,"\Sigma"] & \Poly_{\leq d}(X) \arrow[r, "d_U"] \arrow[d,"\Sigma"] & \Poly^1_{\leq k-\weight}[U; \mathcal{M}(\XX)] \arrow[d, "\Sigma"]  \\ 
    \Poly_{\leq k}(\YY)^{\langle e_\ell \rangle} \arrow[r,hook] &  \Poly_{\leq k}(\XX)^{\langle e_\ell \rangle} \arrow[r, "d_U"]  & \Poly^1_{\leq k-\weight}[U; \mathcal{M}(\XX)]^{\langle e_\ell \rangle}  \\
\end{tikzcd}
    \caption{A portion of the final diagram that we chase after \Cref{initial-diag} to complete the construction of $Q_{[\ell]}$; the diagram is partially commutative and partially exact, and overlaps to some extent with the preceding diagram. Initially, the reader should place (a portion of) $p'$ in the right group of the second row, $Q'$ in the center group of the third row, and $Q_{[\ell-1]}$ in the top row; the objective then involves constructing a new element $Q_{[\ell]}$ in the top row.}
  \label{final-diag}
\end{figure} 

Note that as $e_\ell$ has order $m_\ell$, any element in the range of $\Sigma$ is invariant under $e_\ell$ and hence $\langle e_\ell \rangle$, and by telescoping series $\Sigma$ also annihilates any element in the range of $\partial_{e_\ell}$; in other words, $\Sigma$ serves as an obstruction to inverting $\partial_{e_\ell}$ (this is the ``line cocycle'' condition in \cite{btz}).  The strategy is then to exploit some partial exactness properties of the sequence
$$ B \stackrel{\partial_{e_\ell}}{\to} B \stackrel{\Sigma}{\to} B^{\langle e_\ell \rangle}$$
for a suitable $B$.
Since $s_U^{\ell}$ is also $\left<e_{\ell}\right>$-equivariant, we observe that $\Sigma Q = s(\Sigma^{\oplus U} ((p'_u)_{e_{\ell}})_{u\in U}) = s(0) = 0$. Now that $Q$ is known to be in the kernel of $\Sigma$, we can ``integrate'' it by defining $q \in \Poly^1_{\leq k}(\Gamma_{[\ell]},\XX)$ by the formula
\begin{equation}\label{qgam-def}
 q_\gamma \coloneqq \sum_{i=0}^{\Gamma_{[\ell]} - 1} T^{ie_\ell} Q
 \end{equation}
for $\gamma = (\gamma_1,\dots,\Gamma_{[\ell]})$ in $\Gamma_{[\ell]}$.  The condition $\Sigma Q=0$, ensures that $q$ is a $\Gamma_{[\ell]}$-cocycle, and it is of degree $\leq k$ since $Q$ is.  Applying \Cref{integ-finite}, we can write $q = d_{\Gamma_{[\ell]}} P$ for some $P \in \Poly_{\leq k+1}(\XX)$; we denote by $s_{\ell}$ the map $Q\mapsto P$ defined above and then set 

\begin{equation}\label{Qlwithsections}
    Q_{[\ell]} \coloneqq Q_{[\ell-1]} + s_{\ell}(Q).
\end{equation}
We first verify the compatibility condition \eqref{compat}, which is equivalent to $P$ being $\Gamma_{[\ell-1]}$-invariant, or equivalently that $q_\gamma$ vanishes for $\gamma \in \Gamma_{[\ell-1]}$.  But this is clear from \eqref{qgam-def}.

The same calculation also establishes \eqref{qgam} at $\ell$ for $\gamma \in \Gamma_{[\ell-1]}$ thanks to the induction hypothesis for \eqref{qgam}.  By the cocycle equation, it remains to verify \eqref{qgam} at $\gamma = e_\ell$.  But from \eqref{qgam-def} and \eqref{p'-def} we have
\begin{align*}
\partial_{u} \partial_{e_\ell} Q_{[\ell]}  &= \partial_{u} \partial_{e_\ell} Q_{[\ell-1]} + \partial_{u} \partial_{e_\ell} P \\
&= \partial_{u} \partial_{e_\ell} Q_{[\ell-1]} + \partial_{u} q_{e_\ell} \\
&= \partial_{u} \partial_{e_\ell} Q_{[\ell-1]} + \partial_{u} Q \\
&= \partial_{u} \partial_{e_\ell} Q_{[\ell-1]}  + (p'_u)_{e_\ell} \\
&= (p_u)_{e_\ell} 
\end{align*}
giving \eqref{qgam} as required.  This closes the recursive construction of the $Q_{[\ell]}$, and the claim follows.

Note that $\Poly_{\leq d}^1(\Gamma,\YY)$ is a closed subgroup of $\Poly_{\leq d}^1(\Gamma,\XX)$, and we can view $\Poly_{\leq d}^1(\Gamma,\XX)$ as a closed subgroup of $\mathcal{M}(\XX)^\Gamma$ and $\Poly^1_{\leq d-\weight}[U;Z^1(\Gamma,\XX)]$ as a closed subgroup of $\mathcal{M}(U\times \XX)^\Gamma$, where  $\mathcal{M}(\XX)^\Gamma$ and $\mathcal{M}(U\times \XX)^\Gamma$ are equipped with the product topology respectively. It now follows from the definition of the product topology and the proof of \Cref{HK-C8} that $d_U$ is an open map. 

It is left to prove that the short exact sequence \eqref{GammadU} splits. Observe that for each $p\in \Poly^1_{\leq k-\weight}[U; \mathcal{M}(\XX)]$ we assigned $s(p)\in \Poly^1_{\leq k}(\Gamma,\XX)$ defined by 
$$s(p) = (\lim_{\ell\rightarrow\infty} \partial_\gamma Q_{[\ell]})_{\gamma\in \Gamma}$$
 where the limit is well defined because $\partial_\gamma Q_{[\ell]}$ is eventually a constant. We claim that this $s$ is already a filtration preserving cross-section for \eqref{GammadU}. By construction, $s$ is a homomorphism of groups satisfying $d_U\circ s = \mathrm{Id}$. It remains to show that for each $1\leq j \leq k$ and $p\in \Poly^1_{\leq k-j-\weight}[U; \mathcal{M}(\XX)]$, we have that $s(p) \in   \Poly^1_{\leq k-j}(\Gamma,\XX)$. Now since $s_{\ell}, s_U^{\ell}$ are filtration preserving, by an induction on $\ell$ and \eqref{Qlwithsections}, it follows that $\partial_\gamma Q_{[\ell]}\in \Poly_{\leq k-j}(\XX)$, and thus the limit $s(p)\in \Poly_{\leq k-j}^1(\Gamma,\XX)$ which proves the claim.  
\end{proof}

\section{Taking roots of polynomials: the role of purity}\label{roots-sec}

The arguments in \cite{btz} relied crucially on classifying the gap between $\Poly_{\leq k}(\XX)$ and $n \Poly_{\leq k}(\XX)$ for various $n$ and $k$, focusing in particular on the ability to take $n^{\mathrm{th}}$ roots of polynomials while only increasing the degree of the polynomial by the least amount possible.  These arguments were algebraic in nature and specific to the case of vector spaces $\Gamma = \F_p^\omega$ over finite fields.  Due to Sylow type theorems, one can perform similar analysis when $m$ is square-free, but the algebra becomes considerably more complicated for general $m$.  For instance, we do not have a tractable description\footnote{To illustrate, let us consider the one-dimensional case of polynomials over $\mathbb{Z}/4\mathbb{Z}$. By \Cref{mtimes}, the range of such polynomials consists of certain $2^{m}$-th roots of unity, so it suffices to analyze their behavior in terms of multiplying by powers of $2$. In contrast with the $\mathbb{Z}/2\mathbb{Z}$ case, where multiplication by $2$ always reduces the degree by $1$ (cf.~\cite[Lemma 1.7]{tz-lowchar}), multiplication by $2$ in the setting of $\mathbb{Z}/4\mathbb{Z}$ can reduce the degree by $0$, $1$, $2$, or $3$. 

For example, multiplying $\frac{|x|_4^2}{32}\bmod 1$ by $2$ reduces the degree by $2$; multiplying $\frac{|x|_4^2}{16} \bmod 1$ by $2$ reduces the degree by $3$; multiplying $\frac{|x|_4^2}{8} \bmod 1$ by $2$ reduces the degree by $0$; and multiplying $\frac{|x|_4^2}{4} \bmod 1$ by $2$ reduces the degree by $1$. Here, $|\cdot|_4$ denotes the standard integer lift (or section) from $\mathbb{Z}/4\mathbb{Z}$ to $\mathbb{Z}$.
} of $n \Poly_{\leq k}((\Z/4\Z)^N)$ for general $n,k,N$, though of course for any specific choice of these parameters one can in principle compute this abelian group explicitly. Fortunately for our analysis, we will not need to understand either $n \Poly_{\leq k}(\XX)$ or $n \Poly_{\leq k}(\Gamma)$ precisely, but only know that these abelian groups are somehow ``compatible''.

The link is through the sampling homomorphisms $\iota_{x_0}$ defined in \eqref{sampling}.  For each polynomial $P \in \Poly_{\leq k}(\XX)$ (modulo constants), we assign a measurable representative, which obeys the defining equation $\partial_{\gamma_1} \dots \partial_{\gamma_{k+1}} P(x) = 0$ of a polynomial for all $x$ outside of a null set.  By \Cref{ppfacts}(i), there are only countably many such representatives one needs to select.  This makes $\iota_{x_0}$ a homomorphism from $\Poly_{\leq k}(\XX)$ to $\Poly_{\leq k}(\Gamma)$ for all $k$ and all $x_0$ outside of a null set, which preserves the constants $\T$ and is $\Gamma$-equivariant.  For almost all $x_0$, the pointwise ergodic theorem asserts that the ergodic averages of $e(\iota_{x_0} P)$ using the F{\o}lner sequence $\Gamma_{[\ell]}$ converge to the integral of $e(P)$.  In particular, if $P$ is non-constant, then the mean of $e(P)$ has magnitude strictly less than one, and so $\iota_{x_0} P$ cannot be constant; this implies that $\iota_{x_0}$ is injective for almost all $x_0$. Thus we have the short exact sequence \eqref{sampling-exact} for each $k$ outside of a null set.  Among other things, this makes $\iota_{x_0}$ degree preserving: for any polynomial $P$, $\iota_{x_0} P$ has the same degree as $P$.

We are interested in whether this short exact sequence is split, or equivalently whether there is a retract homomorphism from $\Poly_{\leq k}(\Gamma)$ to $\Poly_{\leq k}(\XX)$ which is a left inverse of $\iota_{x_0}$.  This turns out to be closely related to the $n^{\mathrm{th}}$ root problem and solving more general linear equations:

\begin{lemma}\label{exact-crit}  Let $\XX$ be an ergodic $\Gamma$-system, and $k \geq 0$. The following statements are equivalent:
\begin{itemize}
    \item[(i)]  (Splitting modulo constants) The short exact sequence
    \begin{equation}\label{sampling-exact}
        0\rightarrow \Poly_{\leq k}(\XX)\overset{\imath_{x_0}}{\to} \Poly_{\leq k}(\Gamma)\rightarrow \Poly_{\leq k}(\Gamma)/\imath_{x_0}(\Poly_{\leq k}(\XX))\rightarrow 0
    \end{equation}
    of $k$-filtered abelian groups splits finitely (see Definition \ref{finite-split}) for almost every $x_0 \in X$.
    \item[(ii)]  (Purity) $\XX$ is $k$-pure, i.e., $\iota_{x_0}(\Poly_{\leq k}(\XX))$ is pure up to length $\alpha$ in $\Poly_{\leq k}(\Gamma)$ for every finite $\alpha$ (in the sense of \Cref{pure:def}) for almost every $x_0 \in X$.   
\end{itemize}
\end{lemma}

Informally, (ii) asserts that if one can solve a linear equation in polynomials of degree $\leq k$ ``locally'' (i.e., after applying the sampling operator $\iota_{x_0}$) without any loss of degree, then one can also do so ``globally'' (without sampling). The proof is an immediate application of \Cref{puresplit}.
\begin{proof}
    We start with $(ii)\Rightarrow(i)$. By assumption, $\imath_{x_0}(\Poly_{\leq k}(\XX))\leq \Poly_{\leq k}(\Gamma)$ is pure for almost all $x_0\in X$. By \Cref{puresplit}, \eqref{sampling-exact}  splits finitely for almost every $x_0$. Now we prove $(i)\Rightarrow(ii)$. Let $x_0\in X$ be such that \eqref{sampling-exact} splits finitely for $x_0$. We need to show that $\imath_{x_0}(\Poly_{\leq k}(\XX))\leq \Poly_{\leq k}(\Gamma)$ is pure up to length $\alpha$ for every finite $\alpha$.   
    Let $\mathcal{R}\subset \Z^n\times \{1,\ldots,k+1\}$ be a system of relations in $n$-variables for some $n\in \mathbb{N}$. Let $P=(P_1,\ldots,P_n)$ with $P_i\in \Poly_{\leq k}(\Gamma)$ and suppose that for all $(\vec{m};j)\in \mathcal{R}$ there exists $Q_{(\vec{m};j)}\in \Poly_{\leq k}(\XX)$ such that 
    \begin{equation}\label{virtualsolution}
        \vec{m}\cdot P = \imath_{x_0}(Q_{(\vec{m};j)})+\Poly_{\leq k-j}(\Gamma). 
    \end{equation}
    Consider the subgroup $B\leq \Poly_{\leq k}(\Gamma)$, generated by $\imath_{x_0}(\Poly_{\leq k}(\XX))$ and $P_1,\ldots,P_n$. Since $B/\imath_{x_0}(\Poly_{\leq k}(\XX))$ is finitely generated, by assumption the short exact sequence
    $$0\rightarrow \Poly_{\leq k}(\XX)\overset{\imath_{x_0}}{\rightarrow} B\rightarrow B/\imath_{x_0}(\Poly_{\leq k}(\XX))\rightarrow 0$$ splits. 
    Let $r:B\rightarrow \Poly_{\leq k}(\XX)$ be a retraction and let $Q=(r(P_1),\ldots,r(P_n))$. From \eqref{virtualsolution} we deduce that 
    $$
        \vec{m}\cdot Q = Q_{(\vec{m};j)}+\Poly_{\leq k-j}(\XX),
    $$
    as required.
\end{proof}

\begin{remark}\label{rem:0-1-law}
For any ergodic $\Gamma$-system $\XX$, $k\geq 0$, and finite $\alpha$, the set 
     $$\{x_0 \in X:  \iota_{x_0}(\Poly_{\leq k}(\XX)) \text{pure up to length $\alpha$ in } \Poly_{\leq k}(\Gamma)\}$$ is both measurable and $\Gamma$-invariant, and thus has measure $0$ or $1$ due ergodicity. 
Indeed, $\Gamma$-invariance follows from the above observation that the embedding $\imath_{x_0}$ is $\Gamma$-equivariant for almost all $x_0\in X$. As for measurability, due to the countability of polynomials modulo constants (cf. \Cref{ppfacts}(i)), checking that $\iota_{x_0}(\Poly_{\leq k}(\XX)) \text{ is pure up to length $\alpha$  in } \Poly_{\leq k}(\Gamma)$ only depends on countably many conditions.  
\end{remark}
  
Conveniently, $k$-purity is preserved under reduction of the structure group, at least when the cocycle is exact.

\begin{proposition}[Descent of $k$-purity]\label{exactroots-descent}  Let $k \geq 1$.  Suppose one has an extension $\XX = \YY \times_\rho (V \times W)$
of ergodic $\Gamma$-systems $\XX, \YY$, where $V,W$ are compact abelian groups, and $\rho \in \Poly^1_{\leq k-1}(\Gamma,\YY,V \times W)$ is exact.  If $\XX$ is $k$-pure, then so is $\XX' \coloneqq \YY \times_{\rho \mod V} W$.
\end{proposition}

\begin{proof} We have the commuting diagram of abelian extensions
\begin{center}
\begin{tikzcd}
    \XX \arrow[dd, Rightarrow, "\rho; V \times W"'] \arrow[dr, Rightarrow, "\rho \circ \pi; V"] \\ 
    & \XX' \arrow[dl, Rightarrow, "\rho \mod V; W"]\\
    \YY 
\end{tikzcd}
\end{center}
where $\pi \colon \XX' \to \YY$ is the factor map.  Let $x_0$ be a point in $\XX$, let $x'_0$ be its image in $\XX'$. Let $\mathcal{R}\subset \Z^n\times \{1,\ldots,k+1\}$  be a system of relations in $n$-variables. Let $\tilde{Q} = (\tilde{Q}_1,\ldots,\tilde{Q}_n)$ be in $\Poly_{\leq k}(\Gamma)$ and suppose that for all $(\vec{m};j)\in \mathcal{R}$ there exists $Q_{(\vec{m};j)} \in \Poly_{\leq k}(\XX')$ such that 
$$\imath_{x'_0}(Q_{(\vec{m};j)}) = \vec{m}\cdot \tilde{Q} \mod \Poly_{\leq k-j}(\Gamma).$$

By $k$-purity of $\XX$, we find $P = (P_1,\ldots,P_n)$ in $\Poly_{\leq k}(\XX)$ such that $Q_{(\vec{m};j)} = \vec{m}\cdot P \mod \Poly_{\leq k-j}(\XX)$. This almost accomplishes our goal as $P$ is measurable with respect to $\XX$ rather than $\XX'$. The type filtration $((V \times W)_i)_{i=1}^\infty$ on $V \times W$ restricts to the type filtration $(V_i)_{i=1}^\infty$ on $V$, thanks to \Cref{u-dual}.  In particular, by \eqref{diff-incl}, we see that for any $i \geq 0$ and $v \in V_{>i}$, that $\partial_v P \in \Poly_{\leq k-i-1}(\XX)$. Since $\vec{m} \cdot \partial_v P = \partial_v Q_{(\vec{m};j)} \mod \Poly_{\leq k-j-i-1}(\XX)$, since $\partial_v Q_{(\vec{m};j)}=0$, we conclude that
$$\vec{m}\cdot \partial_v P  \in \Poly_{\leq k-j-i-1}(\XX).$$ As the cocycle $\rho$ is exact, one sees from \Cref{u-dual} and \Cref{8.11} that the quotient cocycle $\rho \circ \pi$ is also exact.  
As $\partial_v P_i$ is a cocycle in $V$ for all $i$, we may now invoke \Cref{poly-integ} for each polynomial $P_i$, and conclude that $\partial_v P = \partial_v R$ for some $R=(R_1,\ldots,R_n)$ in $\Poly_{\leq k}(\XX)$ such that $\vec{m}\cdot R\in \Poly_{\leq k-j}(\XX)$.  Subtracting, we see that $P-R$ is $V$-invariant and thus equal to some $P' \in \Poly_{\leq k}(\XX')$; as modulo $\Poly_{\leq k-j}(\XX)$ we have $\vec{m}\cdot P' = \vec{m}\cdot P - \vec{m}\cdot R = \vec{m}\cdot P = Q_{(\vec{m};j)}$, we obtain the claim.
\end{proof}
A first application of $k$-purity will be to split the short exact sequence relating polynomial cocycles and polynomial coboundaries.

\begin{lemma}\label{poly-root}
    Let $k\geq 1$, and let $\XX$ be an ergodic $k$-pure $\Gamma$-system. Then $d_\Gamma \Poly_{\leq k}(\XX)$ is a pure (filtered) subgroup of $\Poly^1_{\leq k-1}(\Gamma,\XX)$. In particular, the short exact sequence
    \begin{equation}\label{poly-root-seq}
        0\rightarrow d_\Gamma \Poly_{\leq k}(\XX)\to \Poly_{\leq k-1}^1(\Gamma,\XX)\rightarrow \Poly_{\leq k-1}^1(\Gamma,\XX)/_\Gamma \Poly_{\leq k}(\XX)\rightarrow 0
    \end{equation}
    of $k$-filtered abelian groups splits finitely. 
\end{lemma}
\begin{proof}
    Let $\mathcal{R}\subset \Z^n\times \{1,\ldots,k\}$  be a system of relations in $n$-variables, and let $q=(q_1,\ldots q_n)$ be tuple of polynomial cocycles of degree $\leq k-1$ on $\XX$. Suppose that for every $(\vec{m},j)\in \mathcal{R}$ we find some $Q_{(\vec{m};j)}\in \Poly_{\leq k}(\XX)$ such that
    $$d_\Gamma Q_{(\vec{m};j)} = \vec{m}\cdot q \mod \Poly^1_{\leq k-j-1}(\Gamma,\XX).$$
    Then for almost every $x_0\in \XX$ and all $\gamma\in \Gamma$ we have
    $\imath_{x_0}(Q_{(\vec{m};j)})(\gamma) = \vec{m}\cdot q(\gamma,x_0) \mod \Poly_{\leq k-j}(\Gamma)$. Hence, by purity of $\XX$, we find  $Q=(Q_1,\ldots,Q_n)$ with $Q_i\in \Poly_{\leq k}(\XX)$ such that $Q_{(\vec{m},j)} = \vec{m}\cdot Q \mod \Poly_{\leq k-j}(\XX)$. Taking $d_\Gamma$ on both sides of this equation we deduce that
    $$\vec{m}\cdot (q-d_\Gamma Q) \in \Poly^1_{\leq k-j-1}(\Gamma,\XX),$$ as required.
\end{proof}

Now we begin combining $k$-purity with the large spectrum property.  We first use the large spectrum property to provide some equivariance to the retract homomorphisms arising from $k$-purity:

\begin{theorem}[Retract with invariances]\label{retract-invariances}
Let $\XX$ be an ergodic $k$-pure $\Gamma$-system with large spectrum for some $k\geq 1$. Then the short exact sequence \eqref{sampling-exact}
splits finitely in the category of $k$-filtered $\Gamma_{[\ell]}$-groups (where the morphisms are required to be also $\Gamma_{[\ell]}$-equivariant) for all $\ell\geq 0$. 
\end{theorem}

\begin{proof}We introduce the increasing sequence of groups
$$ \Poly_{\leq k,\leq 0}(\Gamma) \leq \Poly_{\leq k,\leq 1}(\Gamma) \leq \dots \leq \Poly_{\leq k,\leq k}(\Gamma) = \Poly_{\leq  k}(\Gamma)$$
where $\Poly_{\leq k, \leq i}(\Gamma)$ denotes the polynomials in $\Poly_{\leq  k}(\Gamma)$ that are of degree $\leq i$ with respect to the $\Gamma_{[\ell]}$ action. Let $B=B_{x_0}=\iota_{x_0}(\Poly_{\leq k}(\XX))$ and denote by $B_{\leq i} = B\cap \Poly_{\leq k,\leq i}(\Gamma)$.  For almost every $x_0 \in X$, we have the short exact sequence
\begin{equation}\label{sampling-exact-0}
  0 \to \Poly_{\leq k,\leq i}(\XX)\stackrel{\iota_{x_0}}{\to} B_{\leq i}\to B_{\leq i}/\imath_{x_0}(\Poly_{\leq k, \leq i}(\XX))\rightarrow 0.
\end{equation}
of locally compact $k$-filtered abelian groups.  We show that for each $0 \leq i \leq k$ this sequence splits with a retract homomorphism 
 $$r_{\leq i} \colon B_{\leq i} \to \Poly_{\leq k,\leq i}(\XX)$$ which is additionally $\Gamma_{[\ell]}$-equivariant. Then setting $i=k$ will give the claim.

We first construct an auxiliary retraction.  Consider the short exact sequence
$$ 0 \to \Poly_{\leq k,\leq 0}(\XX)  \stackrel{\iota_{x_0}}{\to} B \to B/\iota_{x_0}(\Poly_{\leq k,\leq 0}(\XX)) \to 0.$$
By \Cref{exact-crit}, we obtain a retract homomorphism from $B$ to $\Poly_{\leq k}(\XX)$, and from  \Cref{integ-finite}, we have a retract homomorphism from $\Poly_{\leq k}(\XX)$ to $\Poly_{\leq k,\leq 0}(\XX)$.  Composing the two, we may thus find a retract homomorphism $r \colon B \to \Poly_{\leq k,\leq 0}(\XX)$ to the above sequence.

We can now build the retract homomorphisms $r_{\leq i}$.  For $i=0$ this is easy: we just restrict $r$ to $B_{\leq 0}$, with the equivariance being obvious as $\Gamma_{[\ell]}$ acts trivially on both $B_{\leq 0}$ and $\Poly_{\leq k,\leq 0}(X)$. Now suppose that $0 < i \leq k$ and that the retract homomorphism $r_{\leq i-1}$ has already been constructed.  To construct $r_{\leq i}$, we will chase the diagram in \Cref{retract-fig}.
\begin{figure}[ht]
\centering
\begin{tikzcd}
0 \arrow[d, shift left = 1ex]\\
    Z^1[\Gamma_{[\ell]}; \Poly_{\leq k,\leq i-1}(\XX)] 
    \arrow[r, hookrightarrow, "\iota_{x_0}^{\oplus \Gamma_{[\ell]}}"', shift right = 1 ex] 
    \arrow[u] 
    \arrow[d, shift left = 1ex, "s"] 
    & Z^1[\Gamma_{[\ell]}; B_{\leq i-1}] 
    \arrow[l, "r_{\leq i-1}^{\oplus \Gamma_{[\ell]}}"', shift right = 1 ex] \\ 
    \Poly_{\leq k, \leq i}(\XX) 
    \arrow[u, "d_{\Gamma_{[\ell]}}"] \arrow[r,hookrightarrow, "\iota_{x_0}"', shift right = 1 ex] 
    \arrow[d, shift left = 1 ex, "r \iota_{x_0}"]  
    & 
    B_{\leq i}
    \arrow[u, "d_{\Gamma_{[\ell]}}"] 
    \arrow[d, hookrightarrow] 
    \arrow[l, dashed, "r_{\leq i}"', shift right = 1 ex]
  \\
    \Poly_{\leq k, \leq 0}(\XX)
   \arrow[u] 
    \arrow[r, hookrightarrow, "\iota_{x_0}"', shift right = 1 ex] \arrow[d, shift left = 1 ex]
    & 
    B 
    \arrow[l, "r"', shift right = 1 ex] 
  \\
    0 \arrow[u]
\end{tikzcd}
    \caption{The construction of the retract homomorphism $r_{\leq i}$ will involve ``chasing'' the indicated diagram.  We caution that this diagram is only partially commutative, although it will be important to note that the left column is exact.}
  \label{retract-fig}
\end{figure} 

If $Q \in \Poly_{\leq k, \leq i}(\XX)$, then $d_{\Gamma_{[\ell]}} Q$ lies in 
$Z^1[\Gamma_{[\ell]}; \Poly_{\leq k,\leq i-1}(\XX)]$.  Conversely, if $q \in Z^1[\Gamma_{[\ell]}; \Poly_{\leq k,\leq i-1}(\XX)]$, then by \Cref{integ-finite}, we have $q = d_{\Gamma_{[\ell]}} Q$ for some $Q \in \Poly_{\leq k}(\XX)$; as $q$ has degree $\leq i-1$ in $\Gamma_{[\ell]}$, $Q$ has degree $\leq i$.  From this we see that $d_{\Gamma_{[\ell]}}$ is a surjection from $\Poly_{\leq k, \leq i}(\XX)$ to $Z^1[\Gamma_{[\ell]}; \Poly_{\leq k,\leq i-1}(\XX)]$, making the upward sequence on the left column of \Cref{retract-fig} short and exact.  The map $r\iota_{x_0}$ is a retract homomorphism for this sequence by the definition of $r$, hence comes with an associated section homomorphism which we label $s$.

We now set $r_{\leq i}$ to be the homomorphism
$$ r_{\leq i} \coloneqq s r_{\leq i-1}^{\oplus \Gamma_{[\ell]}} d_{\Gamma_{[\ell]}} + r;$$
equivalently, for $ \tilde Q \in B_{\leq i}$, $r_{\leq i}(\tilde Q)$ is the unique element of $\Poly_{\leq k, \leq i}(\XX)$ such that
\begin{equation}\label{rb}
 d_{\Gamma_{[\ell]}} r_{\leq i} \tilde Q = r_{\leq i-1}^{\oplus \Gamma_{[\ell]}} d_{\Gamma_{[\ell]}} \tilde Q
\end{equation}
and
\begin{equation}\label{rot}
 r \iota_{x_0}  r_{\leq i} \tilde Q = r \tilde Q.
\end{equation}
If $\tilde Q = \iota_{x_0}(Q)$ for some $Q \in \Poly_{\leq k, \leq i}(\XX)$, then (since $r_{\leq i-1}$ is a $\Gamma_{[\ell]}$-equivariant retract homomorphism)
$$ d_{\Gamma_{[\ell]}} Q = r_{\leq i-1}^{\oplus \Gamma_{[\ell]}} \iota_{x_0}^{\oplus \Gamma_{[\ell]}} d_{\Gamma_{[\ell]}} Q = r_{\leq i-1}^{\oplus \Gamma_{[\ell]}} d_{\Gamma_{[\ell]}} \tilde Q$$
and
$$ r \iota_{x_0} Q = r \tilde Q$$
and hence by the uniqueness of $r_{\leq i} \tilde Q$ we have $r_{\leq i}\tilde Q = Q$; thus $r_{\leq i}$ is a retract.  It remains to show $\Gamma_{[\ell]}$-equivariance.  By \eqref{rb} and the retract nature of $r_{\leq i-1}$, 
$$ \iota_{x_0} \partial_{\gamma_{[\ell]}}r_{\leq i} \tilde Q = \partial_{\gamma_{[\ell]}}\tilde Q$$
for any $\gamma_{[\ell]} \in \Gamma_{[\ell]}$ and $\tilde Q \in  B_{\leq i}$; the same claim is true at $i=0$ since both sides vanish in that case.  If $\tilde Q \in B_{\leq i-1}$, we then conclude that
$$ \iota_{x_0} \partial_{\gamma_{[\ell]}} r_{\leq i} \tilde Q =  \iota_{x_0}\partial_{\gamma_{[\ell]}} r_{\leq i-1} \tilde Q $$
hence by injectivity of $\iota_{x_0}$ we see that $r_{\leq i} \tilde Q$ and $\tilde r_{\leq i-1}(Q)$ differ by an element of $\Poly_{\leq k,\leq 0}(\XX)$.  Using \eqref{rot} (which also holds at $i=0$ as $r$ is a retract) we conclude that $r_{\leq i}$ and $r_{\leq i-1}$ agree on $\Poly_{\leq k,\leq i-1}(\XX)$.  In particular, we have from \eqref{rb} that
$$ d_{\Gamma_{[\ell]}} r_{\leq i} \tilde Q = r_{\leq i-1}^{\oplus \Gamma_{[\ell]}} d_{\Gamma_{[\ell]}} \tilde Q = r_{\leq i}^{\oplus \Gamma_{[\ell]}} d_{\Gamma_{[\ell]}} \tilde Q$$
for all $\tilde Q \in \Poly_{\leq k,\leq i}(\XX)$.  Hence $r_{\leq i}$ commutes with $\partial_{\gamma_{[\ell]}}$ for all $\gamma_{[\ell]} \in \Gamma_{[\ell]}$, giving the required $\Gamma_{[\ell]}$-equivariance.
\end{proof}

Now we split a sequence of polynomial cocycles.


\begin{proposition}[Splitting a polynomial cocycle sequence]\label{cocycle-root} Let $k \geq 1$, and let $\XX$ be a $(k-1)$-pure ergodic $\Gamma$-system with large spectrum.  Then for almost all $x_0 \in X$, the short exact sequence 
$$ 0 \to \Poly^1_{\leq k-1}(\Gamma,\XX) \stackrel{\iota^{\oplus \Gamma}_{x_0}}{\to} \Poly^1_{\leq k-1}(\Gamma,\Gamma) \to \Poly^1_{\leq k-1}(\Gamma,\Gamma) / \iota^{\oplus \Gamma}_{x_0} \Poly^1_{\leq k-1}(\Gamma,\XX) \to 0$$ 
splits finitely in the category of $(k-1)$-filtered locally compact abelian groups.
\end{proposition}

\begin{proof}  By \Cref{puresplit}, it suffices to show that $\imath_{x_0}^{\oplus \Gamma}(\Poly^1_{\leq k-1}(\Gamma,\XX)) \leq \Poly_{\leq k-1}^1(\Gamma,\Gamma)$ is pure. Let $\mathcal{R}\subset \Z^n\times \{1,\ldots,k\}$  be a system of relations in $n$-variables and let $\tilde{p}=(\tilde{p}_1,\ldots,\tilde{p}_n)$ with $\tilde{p}_i\in \Poly^1_{\leq k-1}(\Gamma,\Gamma)$ be such that for all $(\vec{m};j)\in \mathcal{R}$ there exists some $q_{(\vec{m};j)}\in \Poly^1_{\leq k-1}(\Gamma,\XX)$ such that
\begin{equation}\label{eq-(i)}
    \imath_{x_0}^{\oplus \Gamma}q_{(\vec{m};j)} = \vec{m}\cdot \tilde{p} + \Poly_{\leq k-j-1}^1(\Gamma,\Gamma).
\end{equation}
    We need to find some $q=(q_1,\ldots,q_n)$ in $\Poly^1_{\le k-1}(\Gamma,\XX)$ such that 
    \begin{equation}\label{needtoshow}
    q_{(\vec{m};j)}-\vec{m}\cdot q \in \Poly_{\leq k-j-1}^1(\Gamma,\XX).
    \end{equation}
    Let $\tilde P = (\tilde{P}_1,\ldots,\tilde{P}_n) \in \Poly_{\leq k}(\Gamma)$ be the polynomials $\tilde{P}(\gamma) \coloneqq \tilde{p}(\gamma,0_\Gamma)$. We see that  $\tilde p = d_\Gamma \tilde P$, thus
\begin{equation}\label{nmult}
\imath_{x_0}^{\oplus \Gamma}q_{(\vec{m};j)} = \vec{m}\cdot d_\Gamma\tilde{P} + \Poly_{\leq k-j-1}^1(\Gamma,\Gamma).
\end{equation}
We will solve \eqref{needtoshow} ``locally'' at first, and then glue together the local solutions to create a global solution.\\
Recall the F\o lner sequence \eqref{exhaust} of finite subgroups $\Gamma_{[\ell]}$.  We will recursively construct, for each $\ell$, a solution $P_{[\ell]} = (P_{[\ell],1},\ldots, P_{[\ell],n}) \in \Poly_{\leq k}(\XX)$ to the equation
\begin{equation}\label{ngam}
 q_{(\vec{m};j)}(\gamma,\cdot)=\vec{m} \cdot \partial_\gamma  P_{[\ell]} \mod \Poly_{\leq k-j}(\XX)
\end{equation}
for all $\gamma\in \Gamma_{[\ell]}$, in such a manner that we have the compatibility condition
\begin{equation}\label{compat0}
 d_{\Gamma_{[\ell-1]}} P_{[\ell]} = d_{\Gamma_{[\ell-1]}} P_{[\ell-1]}
 \end{equation}
for all $\ell \geq 1$. Assuming this, we see that for all $\gamma\in \Gamma$, we can choose $\ell$ sufficiently large and set $q(\gamma,x) = \partial_\gamma P_{[\ell]}$. By \eqref{compat0} this definition is independent of the choice of $\ell$ for sufficiently large $\ell$ and so $q$ is well defined and a cocycle (because each two $\gamma_1,\gamma_2\in \Gamma$ belong to some $\Gamma_{[\ell]}$ for sufficiently large $\ell$.). From \eqref{ngam} we will then have that $q$ satisfies \eqref{needtoshow}, thus completing the proof.

It remains to construct the $P_{[\ell]}$.  For $\ell=0$ we can just take $P_{[0]}=0$.  Now suppose inductively for $\ell>0$ that $P_{[\ell-1]}$ has already been constructed. Let $B_{\leq k-1}$ denote the subgroup of $\Poly_{\leq k-1}(\Gamma)$ generated by $\partial_\gamma \tilde{P}$, for all $\gamma\in \Gamma_{[\ell]}$ and all of $\imath_{x_0}(\Poly_{\leq k-1}(\XX))$. Since $\Poly_{\leq k-1}(\Gamma)$ is of bounded exponent (\Cref{mtimes}), $\imath_{x_0}(\Poly_{\leq k-1}(\XX))\subseteq B_{\leq k-1}$ is of finite index, and we conclude from \Cref{retract-invariances}, that we can find a retract homomorphism $r \colon B_{\leq k-1} \to \Poly_{\leq k-1}(\XX)$ which is $\Gamma_{[\ell]}$-equivariant.  For almost all $x_0\in X$, we will chase the (partially commutative) diagram in \Cref{root-fig}.
\begin{figure}[ht]
\centering
\begin{tikzcd}
    Z^1[\Gamma_{[\ell]};\Poly_{\leq k-2}(\XX)] \arrow[r, hookrightarrow, "\iota_{x_0}^{\oplus \Gamma_{[\ell]}}"] & Z^1[\Gamma_{[\ell]};B_{\leq k-2}] \\ 
    \Poly_{\leq k-1}(\XX) \arrow[u, "d_{\Gamma_{[\ell]}}"] \arrow[r,hookrightarrow, "\iota_{x_0}"', shift right = 1 ex]  & B_{\leq k-1} \arrow[u, "d_{\Gamma_{[\ell]}}"] \arrow[l,rightarrow, "r"', shift right = 1 ex]\\
    \Poly_{\leq k-1}(\XX) \arrow[u, "\vec{m}"] \arrow[r, hookrightarrow, "\iota_{x_0}"', shift right = 1 ex] & B_{\leq k-1} \arrow[l, "r"', shift right = 1 ex] \arrow[u, "\vec{m}"]   
\end{tikzcd}
    \caption{The diagram that we will chase to construct $P_{[\ell]}$, where $n$ denotes the operation of multiplication by $n$.  We caution that this diagram is only partially commutative, and the sequences are not expected to be exact. The reader is invited to annotate this diagram as the proof progresses, for instance starting with $\tilde P$ in the bottom right group, $P_{[\ell-1]}$ in the bottom left, and $q$ in the top left.}
  \label{root-fig}
\end{figure} 
By $\Gamma_{[\ell]}$-equivariance of $r$, the map
$\gamma \mapsto r(\partial_\gamma \tilde P)$ 
is a cocycle on $\Gamma_{[\ell]}$. Thus, by \Cref{integ-finite} we can find a polynomial $\bar{P}\in \Poly_{\leq k}(\XX)$ such that 
\begin{equation}\label{newP}
r(\partial_\gamma \tilde{P}) = \partial_\gamma \bar{P}
\end{equation}
for all $\gamma\in \Gamma_{[\ell]}$. 
We consider the polynomial $P' \in \Poly_{\leq k}(\XX)$ defined by 
$$P' \coloneqq \bar{P}- P_{[\ell-1]}.$$
Then for any $\gamma \in \Gamma_{[\ell-1]}$, from the retract properties of $r$, together with \eqref{ngam} for $\ell-1$, we see that the following equalities hold modulo $\Poly_{\leq k-j-1}(\Gamma)$. 
\begin{align*}
\iota_{x_0}(\vec{m}\cdot\partial_\gamma P') &= \imath_{x_0}(\vec{m}\cdot \partial_{\gamma}\overline{P}) - \imath_{x_0}(\vec{m}\cdot \partial_\gamma P_{[\ell-1]})\\&=\vec{m} \cdot  \iota_{x_0}r(\partial_\gamma\tilde{P})- \iota_{x_0}^{\oplus\Gamma} q_{(\vec{m},j)}\\
&= 0,
\end{align*}
where the last equality follows from \eqref{nmult}. Since $\imath_{x_0}^{\oplus \Gamma}q_{\vec{m},j}$, $\partial_\gamma\tilde{P}$ are polynomials of degree $\leq k-1$, we may apply $r$ to both of these functions whenever $\gamma\in \Gamma_{[\ell]}$. Since the difference is in $\Poly_{\leq k-j-1}(\Gamma)$, and $r$ preserves the degrees, we see that
$$q_{(\vec{m},j)}(\gamma,\cdot) = \vec{m}\cdot r(\partial_\gamma \tilde{P})$$ modulo $\Poly_{\leq k-j-1}(\XX)$. Now, taking $\imath_{x_0}$ on both sides the claim follows.

By the injectivity of $\iota$, we conclude that $\vec{m} \cdot \partial_\gamma P' \in \Poly_{\leq k-j-1}(\XX)$ for all $\gamma\in \Gamma_{[\ell]}$. Applying  \Cref{integ-finite} for each polynomial in $\gamma\mapsto \partial_\gamma P'$, we can find a polynomial $P'' = (P''_1,\ldots,P''_n) \in \Poly_{\leq k}(\XX)$ such that 
\begin{equation}\label{p-p'}
\partial_\gamma P' = \partial_\gamma P''
\end{equation} and 
\begin{equation}\label{residue}
    \vec{m}\cdot P''\in \Poly_{\leq k-j}(\XX)
\end{equation} where the latter is true since $\vec{m}\cdot P''$ is exactly the function defined as in \eqref{antiderivative} for $\partial_\gamma \vec{m} \cdot P'$.\\
We now define 
\begin{equation}\label{pell-def}
P_{[\ell]} \coloneqq P_{[\ell-1]} + P' - P'' \in \Poly_{\leq k}(\XX).
\end{equation}
From \eqref{p-p'} we clearly have \eqref{compat0}.  For $\gamma \in \Gamma_{[\ell]}$, modulo $\Poly_{\leq k-j-1}(\XX)$ we have
\begin{align*}
   \iota_{x_0}( \vec{m}\cdot \partial_\gamma P_{[\ell]} ) &= \vec{m}\cdot \iota_{x_0} \partial_\gamma P_{[\ell-1]}  + \vec{m}\cdot \iota_{x_0} \partial_\gamma P'  - \iota_{x_0} \partial_\gamma(\vec{m}\cdot P'')  \\
    &= \vec{m}\cdot \iota_{x_0} \partial_\gamma P_{[\ell-1]}  + \vec{m}\cdot \iota_{x_0} \partial_\gamma\bar P - \vec{m}\cdot\iota_{x_0} \partial_\gamma P_{[\ell-1]} - \vec{m}\cdot \partial_\gamma P''\\
    &= \vec{m}\cdot \imath_{x_0}r(\partial_\gamma \tilde P)\\
    &= \iota_{x_0} (q_\gamma), 
\end{align*}
where the last equation again follows from \eqref{nmult}, as before. Thanks to the equivariance and retract properties of $r$ and \eqref{eq-(i)}.  By the injectivity of $\iota_{x_0}$, we conclude \eqref{ngam}, and this closes the induction. 
\end{proof}

\section{Retraction with global invariance}\label{invariant}

We introduce the following notion of relative purity. At the end of this paper (see the proof of \Cref{technical} in \Cref{proof}), we will prove that a relatively pure extension of a pure system is pure. 

\begin{definition}[Relative purity]
    Let $k\geq 1$ and let $\XX=\YY\times_\rho U$ be an abelian extension of ergodic $\Gamma$-systems. 
    We say that $\XX$ is \emph{$k$-pure relative to $\YY$} if $(d_\Gamma\Poly_{\leq k}(\XX))^U$ is a pure $k$-filtered locally compact abelian subgroup of $\Poly_{\leq k-1}^1(\Gamma,\XX)$. 
\end{definition}

The following technical result is needed in the key cyclic linearization result (\Cref{cycliclinearization}).

\begin{theorem}[Equivariant filtration preserving retractions]\label{EDPretraction}
    Let $k\geq 1$ and let $\XX=\YY\times_\rho U$ be an abelian extension of ergodic $\Gamma$-systems with $\rho$ exact and taking values in a cyclic group $U=\Z/m\Z$. Suppose that $\XX$ is both $k$-pure and $k$-pure relative to $\YY$. Let $n\in\N$, $\{p_1,\ldots,p_n\}\subset \Poly_{\leq k-1}^1(
    \Gamma,\XX)$, and $\{d_\Gamma Q_1,\ldots,d_\Gamma Q_n\}\subseteq d_\Gamma \Poly_{\leq k}(\XX)$. Let $A$ be the group generated by $d_\Gamma Q_1,\ldots,d_\Gamma Q_n$, their $U$-translations, and $(d_\Gamma\Poly_{\leq k}(\XX))^U$ and let $B$ be the group generated by $p_1,\ldots,p_n$, their $U$-translations and $A$. Then there is a $U$-equivariant filtration preserving retraction $r:B\rightarrow A$.   
\end{theorem}
\begin{proof}
By assumption, $(d_\Gamma\Poly_{\leq k}(\XX))^U$ is pure in $\Poly_{\leq k-1}^1(\Gamma,\XX)$. Since $B$ is generated by $(d_\Gamma\Poly_{\leq k}(\XX))^U$ and finitely many elements, we obtain from \Cref{puresplit} a filtration preserving (but not necessarily equivariant) retract $$r_0:B\rightarrow (d_\Gamma\Poly_{\leq k}(\XX))^U.$$

For each $0\leq i \leq k$, let $B_{\leq i}\leq B$ denote the subgroup of $B$ of polynomials of degree $\leq i$ with respect to the $U$-action, and similarly define $A_{\leq i} \leq A$.  By construction, the restriction of $r_0$ to $B_{\leq 0}$ is $U$-equivariant. 
Now, by induction on $i$, we construct a filtration preserving $U$-equivariant retraction $r_{\leq i} \colon   B_{\leq i}\rightarrow A_{\leq i}$; then setting $r=r_{\leq k}$ proves the claim. Assume we have already constructed $r_{\leq i-1}$. Let $q\in B_{\leq i}$. Since $\partial_u q \in B_{\leq i-1}$ and the $U$-equivariance of $r_{\leq i-1}$, we have that $u\mapsto r_{\leq i-1}(\partial_uq)$ is a $U$-cocycle. Furthermore, since $r_{\leq i-1}$ preserves the filtrations, we also have that
$$r_{\leq i-1}(\partial_u q) = d_\Gamma Q_u$$ where $d_\Gamma Q_u \in d_\Gamma\Poly_{\leq k-\weight(u)}(\XX)\cap A_{\leq i-1}$. Let $e$ be a generator of $\Z/m\Z$. Since $u\mapsto d_\Gamma Q_u$ is a cocycle, ergodicity implies that $\Sigma_e Q_e\coloneqq \sum_{i=0}^{m-1}Q_e\circ V_{ie} \in \mathbb{T}$ is a constant. By divisibility of $\mathbb{T}$, we can modify $Q_u$ by subtracting a constant $c_u$ such that now $u\mapsto d_\Gamma Q_u$ is a cocycle, while still having $d_\Gamma Q_u \in d_\Gamma \Poly_{\leq k-\weight(u)}(\XX)\cap A_{\leq i-1}$ and $r_{\leq i-1}(\partial_u q) = d_\Gamma Q_u$. By \Cref{poly-integ}, there is $Q\in \Poly_{\leq k}(\XX)$ such that $Q_u = \partial_u Q$ and $d_\Gamma Q \in A_{\leq i}$. In other words, $r_{\leq i-1}^{\oplus U} d_U$ takes values in $d_Ud_\Gamma A_{\leq i}$. 

Now consider the short exact sequence

$$0\rightarrow (d_\Gamma\Poly_{\leq k}(\XX))^U\rightarrow A_{\leq i}\overset{d_U}\rightarrow d_U A_{\leq i}\rightarrow 0.$$

Since $A_{\leq i}/(d_\Gamma\Poly_{\leq k}(\XX))^U$ is finitely generated, by \Cref{puresplit}, there exists a cross-section of filtered groups $s_{\leq i} \colon  d_U(A_{\leq i}) \rightarrow A_{\leq i}$ such that $r_0+s_{\leq i}\circ d_U = \mathrm{Id}$ on $A_{\leq i}$. We now set
$$r_{\leq i} \coloneqq s_{\leq i} r_{\leq i-1}^{\oplus U} d_U + r_0.$$
To complete the proof, we claim that $r_{\leq i}$ is a filtration preserving and $U$-equivariant retraction. Since $r_{\leq i-1},s_{\leq i}$, and $r_0$ are filtration preserving it follows from the definition that so is $r_{\leq i}$. The direct computation  Since
$$d_U r_{\leq i} = d_U s_{\leq i} r_{\leq i-1}^{\oplus U} d_U + d_U r_0 = r_{\leq i-1}^{\oplus U}d_U,$$ shows that $r_{\leq i}$ is $U$-equivariant. Finally, in order to verify that $r_{\leq i}$ is a retract we induct on $i$. If $t\in A_{\leq 0}$, then $r_{\leq i}(t)=r_0(t)=t$. Let $t\in A_{\leq i}$ and assume inductively that $r_{\leq i-1}^{\oplus U}(d_Ut) = d_Ut$. By the choice of the cross-section $s_{\leq i}$, $r_{\leq i}(t) = s_{\leq i}(d_U t) + r_0(t) = t$, as required.
\end{proof}

\subsection{The cyclic linearization lemma}

The following technical lemma plays a crucial role in this paper.

\begin{lemma}[Cyclic linearization]\label{cycliclinearization} Let $k\geq 0$ and let $\XX = \YY \times_{\rho} U$ be an ergodic $\Gamma$-system of order $\leq k$ with $\rho$ exact and $U=\Z/m\Z$ a cyclic group. Suppose that $\XX$ is both $k$-pure and $k$-pure relative to $\YY$. Let $-1\leq d<k$, let $n\geq 1$, let $\mathcal{R}$ be a finite system of relations in $n$-variables of type at most $d+1$, and let $\sigma = (\sigma_1,\ldots,\sigma_n) \in Z^1(\Gamma,\XX)$. If for all $u \in U$ and for all $(\vec{m},j)\in \mathcal{R}$,
\begin{equation} \label{generalCL}
\partial_u \sigma \in \Poly^1_{\leq d-\weight(u)}(\Gamma,\XX) + d_\Gamma \mathcal{M}(\XX)
\end{equation}
and
\begin{equation}\label{mCL}
\partial_u (\vec{m} \cdot \sigma) \in \Poly^1_{\leq d-j-\weight(u)}(\Gamma,\XX) + d_\Gamma\mathcal{M}(\XX).
\end{equation} 
Then there exists $\sigma'=(\sigma'_1,\ldots,\sigma'_n) \in Z^1(\Gamma,\XX)$ cohomologous to $\sigma$ such that for all $u \in U$ and for all $(\vec{m},j)\in \mathcal{R}$, 
$$ \partial_u \sigma' \in \Poly^1_{\leq d-\weight(u)}(\Gamma,\XX)$$ 
and
$$\partial_u (\vec{m} \cdot \sigma') \in \Poly^1_{\leq d-j-\weight(u)}(\Gamma,\XX).$$
\end{lemma}

\begin{proof}
    Let $e\in U$ be a generator. For $1\leq i\leq k+1$, define 
    \[
    m_i=\min\{n\in \N\colon ne \in U_{>i-1}\}.
    \]
    Note that $m_i$ divides $m_{i+1}$ for $i=1,\ldots,k$. Moreover, since $X$ is of order $\leq k$, $U_{>k}$ is trivial, and so $m_{k+1}=m$.
    From \eqref{generalCL}, we have
    \begin{equation}\label{rhoi}
    \partial_{m_ie}\sigma = p_i - d_\Gamma F_i
    \end{equation}
    where $p_i\in \Poly^1_{\leq d-i}(\Gamma,\XX)$ and $F_i\in \mathcal{M}(\XX)$. 
    First, we claim that we can assume without loss of generality that for all $(\vec{m},j)\in\mathcal{R}$, 
    \begin{equation}\label{wlog}
    \vec{m}\cdot p_i\in \Poly_{\leq d-j-i}^1(\Gamma,\XX).
    \end{equation} 
    Indeed, from \eqref{mCL} and since the $p_i$'s are polynomial, we have that $$\vec{m}\cdot p_i \in \Poly_{\leq d-j-i}^1(\Gamma,\XX) + d_\Gamma \Poly_{\leq d}(\XX).$$ Fix $1\leq i \leq k+1$, and let $B$ be the group generated by $d_\Gamma \Poly_{\leq d}(\XX)$ and the $n$ coordinate functions of $p_i$. By \Cref{poly-root}, there exists a filtration preserving retraction $r:B\rightarrow d_\Gamma \Poly_{\leq d}(\XX)$. Write $r(p_i) = d_\Gamma F'_i$ (where we apply $r$ coordinate wise), replacing $p_i$ with $p_i-r(p_i)$ and $F_i$ with $F_i-F'_i$ we have that \eqref{wlog} and \eqref{rhoi} hold simultaneously.

    Assuming \eqref{wlog} holds, we construct, by backwards induction on $1\leq \ell\leq k+1$, cocycles $\sigma^{(\ell)} = (\sigma^{(\ell)}_1,\ldots,\sigma^{(\ell)}_n)$ in $Z^1(\Gamma,\XX)$, polynomial cocycles $p_i^{(\ell)} = (p_{i,1}^{(\ell)},\ldots,p_{i,n}^{(\ell)})$ in $\Poly^1_{\leq d-i}(\Gamma,\XX)$, and measurable functions $F_i^{(\ell)} = (F_{i,1}^{(\ell)},\ldots,F_{i,n}^{(\ell)})$ in $\mathcal{M}(\XX)$ such that:
    \begin{itemize}
        \item[(1)] $\sigma^{(\ell)}$ are cohomologous to $\sigma$;
        \item[(2)] $\vec{m}\cdot p_i^{(\ell)}\in \Poly^1_{\leq d-i-j}(\Gamma,\XX)$ for all $(\vec{m},j)\in\mathcal{R}$ and $1\leq i\leq k+1$;
        \item[(3)] $\partial_{m_ie}\sigma^{(\ell)} = p_i^{(\ell)} -d_\Gamma F_i^{(\ell)}$ for all $1\leq i\leq k+1$;
        \item[(4)] $F_i^{(\ell)} =0$ for all $i\geq \ell.$
    \end{itemize}
    Assuming we accomplished this, setting $\ell=1$, we have $F_i^{(1)}=0$ for all $1\leq i\leq k+1$, and thus $\sigma'=\sigma^{(1)}$ satisfies the required result. 
    
    When $\ell=k+1$, since $U_{>k}$ is trivial, it follows from \eqref{generalCL} and ergodicity that $F_{k+1}$ is a constant. The choice $\sigma^{(k+1)}=\sigma$, $p_i^{(k+1)}=p_i$, and $F_i^{(k+1)}=F_i-F_{k+1}$ then satisfies $(1)-(4)$ for this $\ell$. Assume inductively that we have already constructed $\sigma^{(\ell+1)}$, $p_i^{(\ell+1)}$ and $F_i^{(\ell+1)}$. Setting $i=\ell$ in $(3)$ we get
    \begin{equation}\label{l+1}
        \partial_{m_\ell e}\sigma^{(\ell+1)} = p_\ell^{(\ell+1)} -d_\Gamma F_\ell^{(\ell+1)}.
    \end{equation}
    Let $n_{\ell}\coloneqq m_{\ell+1}/m_{\ell}$. We have the telescoping identity
    $$\partial_{m_{\ell+1}e}\sigma^{(\ell+1)} =\sum_{t=0}^{n_{\ell}-1} \partial_{m_{\ell}e}\sigma^{(\ell+1)}\circ V_{t\cdot m_{\ell} e}.$$ From $(3),(4)$ and the induction hypothesis, the left hand side is equal to $p_{\ell+1}^{(\ell+1)}$. Therefore by \eqref{l+1}, we obtain the equation
    \begin{equation}\label{telescoping}
        d_\Gamma \left(\sum_{t=0}^{n_{\ell}-1}  F_{\ell}^{(\ell+1)} \circ  V_{t\cdot m_{\ell} e}\right) = \sum_{t=0}^{n_{\ell}-1} p_{\ell}^{(\ell+1)}\circ V_{t\cdot m_{\ell}e} - p_{\ell+1}^{(\ell+1)}.
    \end{equation}
    Let $A$ denote the group generated by the coboundary $ d_\Gamma \left(\sum_{t=0}^{n_{\ell}-1}  F_{\ell}^{(\ell+1)} \circ  V_{t\cdot m_{\ell} e}\right)$, all of its translations by $U$, and $(d_\Gamma \Poly_{\leq k}(\XX))^U$. Let $B$ denote the group generated by this $A$, the polynomials $p_{\ell}^{(\ell+1)}$ and $p_{\ell+1}^{(\ell+1)}$ and all of their translations by $U$. Then by \Cref{EDPretraction}, we can find a $U$-equivariant filtration-preserving retraction $r:B\rightarrow A$. Write $r(p_{\ell}^{(\ell+1)}) \coloneqq d_\Gamma P_\ell$ and $r(p_{\ell+1}^{(\ell+1)}) =  d_\Gamma P_{\ell+1}$.  Applying $r$ to both sides of \eqref{telescoping}, the $U$-equivariance of $r$ and ergodicity implies that
    \begin{equation}\label{constant}
        \sum_{t=0}^{n_{\ell}-1}  (F_{\ell}^{(\ell+1)} - P_{\ell}) \circ  V_{t\cdot m_{\ell} e} + P_{\ell+1} 
    \end{equation}
    is constant. Consider the function \begin{equation}\label{lineareq1}F'_{\ell} \coloneqq F_{\ell}^{(\ell+1)} - P_{\ell}  - c
    \end{equation}
    where $c$ is an $n_{\ell}$-root of the constant in \eqref{constant}, we have
    \begin{equation}\label{linearmodulo}
         \sum_{t=0}^{n_{\ell}-1}  F'_{\ell} \circ  V_{t\cdot m_{\ell} e} = P_{\ell+1}.
    \end{equation}
    From the induction hypothesis, we see that for all $i\geq \ell+1$, we have $\partial_{m_ie}\sigma^{(\ell+1)} = p_i^{(\ell+1)}\in \Poly_{\leq d-i}^1(\Gamma,\XX)$. Since $r$ is filtration-preserving it preserves the degree of the polynomials and thus,
    \begin{equation}\label{degrees}
        \sum_{t=0}^{m_i/m_{\ell}-1} P_{\ell+1}\circ V_{t\cdot m_{\ell} e} \in \Poly_{\leq d-i+1}(\XX)
    \end{equation}
    In particular setting $i=k+1$ we get
$$\sum_{t=0}^{m_{k+1}/m_{\ell}-1}P_{\ell+1}\circ  V_{t\cdot m_\ell e}=0.$$
    Combining this with \eqref{linearmodulo}, we deduce that
    \begin{equation}
        \sum_{t=0}^{m_{k+1}/m_\ell-1}F'_\ell \circ V_{t\cdot m_\ell e} = \sum_{t=0}^{m_{k+1}/m_{\ell+1}-1}\left(\sum_{t=0}^{n_{\ell}-1} F'_\ell \circ V_{t\cdot m_\ell e} \right)\circ V_{t\cdot m_{\ell+1}e} =0.
    \end{equation}
    Equivalently, the map $z\mapsto F'_z\coloneqq\sum_{t=0}^{z-1}F'_\ell \circ V_{t\cdot m_{\ell} e}$ is a cocycle from the group $\left<m_{\ell} e\right>$ to $\mathcal{M}(\XX)$. Thus, by \Cref{HK-C8}, we can find $F'\in \mathcal{M}(X)$ such that \begin{equation}\label{lineareq2}
    \partial_{m_{\ell} e} F' = F'_\ell.
    \end{equation}  Define the cocycle $$\sigma^{(\ell)} \coloneqq \sigma^{(\ell+1)}-d_\Gamma F'.$$ For all $i\leq \ell$, set
 \begin{equation}\label{lineareq3}
    F_i^{(\ell)} \coloneqq F_i^{(\ell+1)} - \partial_{m_\ell e} F' - P_\ell -c
\end{equation}
and
  $$  p_i^{(\ell)} \coloneqq p_i^{(\ell+1)}+d_\Gamma P_{\ell},$$
and for all $i\geq \ell+1$, set $F_i^{(\ell)}=0$ and  $p_i^{(\ell)} = p_i^{(\ell+1)} - d_\Gamma (\sum_{t=0}^{m_i/m_{\ell+1}-1} P_{\ell+1}\circ V_{t\cdot m_{\ell+1}e})$. 

We verify the properties $(1)-(4)$. By transitivity of the cohomology equivalence, $\sigma^{(\ell)}$ is cohomologous to $\sigma$, giving $(1)$. Since $r$ is filtration-preserving, it preserves the degrees of polynomials so we have that $d_\Gamma P_{\ell}\in \Poly_{\leq d-\ell}^1(\Gamma,\XX)$ and $d_\Gamma (\vec{m}\cdot P_{\ell}) \in \Poly_{\leq d-\ell-j}^1(\Gamma,\XX)$, thus giving $(2)$ when $i\leq \ell$. When $i\geq \ell+1$  property $(2)$ follows from \eqref{degrees}. When $i\leq \ell$, property $(3)$ follows from a direct computation, and when $i\geq \ell+1$ we have
\begin{align*}
    \partial_{m_{i}e}F' &= \sum_{t=0}^{m_i/m_{\ell}-1}\partial_{m_\ell e}F'\circ V_{t\cdot m_\ell e} \\
    &= \sum_{t=0}^{m_i/m_{\ell+1}-1} \left(\sum_{t=0}^{n_{\ell}-1} F'_\ell\circ V_{t\cdot m_\ell e}\right)\circ V_{t\cdot m_{\ell+1}e} 
    \\ &= \sum_{t=0}^{m_i/m_{\ell+1}-1}P_{\ell+1}\circ V_{t\cdot m_{\ell+1}e}.
\end{align*}
Therefore, by construction $p_{i}^{(\ell)} + \partial_{m_ie}d_\Gamma F' = p_i^{(\ell+1)} = \partial_{m_ie}\sigma^{(\ell+1)}$. Finally, property $(4)$ follows from definition when $i\geq \ell+1$. When $i=\ell$, property $(4)$ follows by combining equations \eqref{lineareq1}, \eqref{lineareq2} and \eqref{lineareq3}.
\end{proof}

\section{The straightening property: The role of relative purity}\label{stronglyabramov:section}

The \emph{straightening property} was studied by Bergelson, Tao, and Ziegler in \cite{btz} to investigate the structure of the Host--Kra factors of $\mathbb{F}_p^\omega$-systems. 
\begin{definition}[The straightening property]
    Let $d\geq 1$ and let $\Gamma$ be a countable abelian group. An ergodic $\Gamma$-system $\XX$ is said to satisfy the \emph{straightening property of type $d$} if any cocycle $\rho \colon \Gamma\times \XX\rightarrow \mathbb{T}$ of type $\ell$ is cohomologous to a polynomial of degree $\leq \ell-1$ for all $\ell\leq d$.
\end{definition}

The main result of this section establishes that polynomial towers satisfying the relative purity property fulfill the straightening property. More precisely:
\begin{proposition}\label{stronglyabramov}
Let $\XX$ be an ergodic $\Gamma$-system, let $1\leq j \leq k$, and let $d\geq 0$. Suppose that $\XX$ is an exact polynomial tower of order $\leq k$ and height $j$ with large spectrum and such that each $\XX_i$ is $d$-pure and each extension  
\[
\begin{tikzcd}
    \XX_i \arrow[d, Rightarrow, "\sigma_{i-1}; U_i"] \\
    \XX_{i-1} 
\end{tikzcd}
\]
is relatively $d$-pure for $1<i\leq j$. Then $\XX$ satisfies the straightening property of type $d+1$. 
\end{proposition}

To prove this proposition, we follow the argument of Bergelson, Tao, and Ziegler \cite{btz} with various modifications, although we will be able to simplify some parts of their proof by using  \Cref{cocycle-integ}. The remainder of this section is devoted to the proof of \Cref{stronglyabramov}. 

For technical reasons arising from the fact that \Cref{8.11} is only applicable when the order of the underlying system is smaller than the order of the cocycle $\rho$ (the \emph{low order case}), Bergelson, Tao and Ziegler \cite{btz} study the high order case $d<k$ and the low order case $d+1\geq k$, separately. We need to do the same, and we shall start by reducing matters to the low order case.
\begin{proposition}
Let $k\geq 1$, it suffices to prove \Cref{stronglyabramov} in the case where $d+1\geq k$. 
\end{proposition}
\begin{proof}
    Let $\rho \colon \Gamma\times \XX\rightarrow \mathbb{T}$ be a cocycle of type $d+1$. By \Cref{7.9}, $\rho$ is cohomologous to a cocycle measurable with respect to $\ZZ^{\leq d+1}(\XX)$. It is left to show that when $k>d+1$, $\ZZ^{\leq d+1}(\XX)$ is an exact polynomial tower of height $j$, has large spectrum and that each $(\ZZ^{\leq d+1}(\XX))_i$ is $d$-pure and a relative $d$-pure extension of $\ZZ^{\leq d+1}(\XX)$. 
    To do so we induct on $j$. When $j=1$, there is nothing to prove because $k=1$ and $Z^1(\XX)=\XX$. Let $j\geq 2$ and assume inductively that the claim holds for all smaller values of $j$. Write $\XX=X_{j-1}\times_{\sigma_j} U_j$. Since $\ZZ^{\leq d+1}(\XX) =\ZZ^{\leq d+1}(\XX_{j-1})\times_{\sigma_j \mod U_{j,>d+1}} U_j/U_{j,>d+1},$ we see from the induction hypothesis that $\ZZ^{\leq d+1}(\XX)$ is an exact polynomial tower of height $j$. Finally, the large spectrum follows from the fact that it only depends on $Z^1(\XX)$ and all the purity results follow from the fact that polynomials of degree $\leq d+1$ are measurable with respect to $\ZZ^{\leq d+1}(\XX)$ (see \Cref{ppfacts}(ii)) 
\end{proof}
This proposition is particularly useful as throughout the proof we will start with a cocycle $\rho$ and by subtracting polynomials from $\rho$ we will show that it is cohomologous to cocycles that are measurable with respect to smaller factors. It is then important that these cohomologous cocycles inherit the same type as $\rho$ on the smaller factors. We obtain this by applying \Cref{8.11} which requires that the order of the underlying system is equal to the type of the cocycle $\rho$.
We proceed by induction on $d$.  The case $d=0$ follows immediately from \Cref{type}(vi). Let $d\geq 1$ (note that now by the previous proposition we can assume $d\geq k\geq 1$) and assume that the claim holds for all $d'<d$, that is, any $\mathbb{T}$-valued cocycle of type $s$ is cohomologous to a polynomial of degree $\leq s-1$ for all $s<d+1$. 

We proceed by another induction on $j$. We postpone the proof of the induction base $j=1$ (which also corresponds to $k=1$) for later. Assume that $j\geq 2$ and that the claim holds for all $j'<j$.
Write $\XX = \XX_{j-1}\times_\sigma U$, where $\XX_{j-1}$ is the sub-tower of height $j-1$, $\sigma \colon \Gamma\times \XX_{j-1}\rightarrow U$ is an exact polynomial cocycle of degree $\leq k-1$, and $U$ is a compact abelian group. Let $\rho \colon \Gamma\times\XX\rightarrow \T$ be a cocycle of type $d+1$. By \Cref{type}(iv), $\partial_u \rho$ is a cocycle of type $d$ (at most), thus by the induction hypothesis, $\partial_u \rho$ is cohomologous to a polynomial of degree $\leq d-1$.\\

\noindent \textit{Reduction to finite $U$.} Consider the group $$G\coloneqq\{(u,F)\in U\times \mathcal{M}(\XX) \colon \partial_u \rho -d_\Gamma F\in \Poly_{\leq d-1}^1(\Gamma,\XX)\}$$ with group law $$(u,F)\cdot (u',F') = (u+u', F'+F\circ V_{u'}).$$ We have just argued that the projection map $p:G\rightarrow U$ is surjective. The kernel of $p$ is isomorphic to $\Poly_{\leq d}(\XX)$. Therefore, we have the short exact sequence
\begin{equation}\label{shortseq}
0\rightarrow \Poly_{\leq d}(\XX)\rightarrow G\rightarrow U\rightarrow 0.
\end{equation}

\begin{lemma}[Right split for an open subgroup]\label{opensection}
    There exists an open subgroup $U'\leq U$ and a homomorphism $\varphi:U'\rightarrow G$ such that $p\circ \varphi:U'\rightarrow U$ is the natural embedding.
\end{lemma}

\begin{proof}
In the case where $U$ is $p$-torsion, this was established in \cite[Proposition 6.1]{btz}, but a similar arguments extends to general groups of bounded exponent; see also \cite[Proof of Proposition 5.2]{jst-tdsystems}.
\end{proof}

We want to use \Cref{opensection} to reduce matters to finite $U$. Let $U'$ be as in \Cref{opensection}, and write $\varphi(u) = (u,F_u)$. By \Cref{HK-C8}, we can find a measurable map $F \colon \XX\rightarrow \mathbb{T}$ such that $F_u = \partial_u F$. Since $(u,F_u)\in G$, we can find a polynomial cocycle $p_u$ of degree $\leq d-1$ such that 
$$\partial_u \rho = p_u + d_\Gamma F_u.$$
Replacing $\rho$ with the cohomologous cocycle $\rho+d_\Gamma F$, we may assume that
\begin{equation}\label{pu}
\partial_u \rho = p_u \text{ for all } u\in U'.
\end{equation}
In particular, $(p_u)_{u\in U'}\in \Poly_{\leq d-1}^1[U';Z^1(\Gamma,\XX)]$. 

Let $U_\bullet$ denote the type filtration on $U$ with respect to $\sigma$. By \Cref{typefiltration}, $U_{>i}$ fixes the $\sigma$-algebra associated with the factor $\ZZ^{\leq i}(\XX)$ for all $i\geq 0$. By \Cref{type}(iv), $\partial_u \rho$ is a cocycle of type $\leq d-\weight(u)$ for all $u\in U$. Let $U'_{\bullet}$ be the induced filtration on $U'$. By the induction hypothesis, since $p_u$ is of type $d-\weight(u)$ for all $u\in U'$, we can write 
$$p_u = q_u + d_\Gamma F'_u$$ for some polynomial cocycle $q_u \colon\Gamma\times \XX\rightarrow \mathbb{T}$ of degree $\leq d-\weight(u)-1$ (with the convention that $q_u=0$ if $\weight(u)\geq d$) and a measurable map $F'_u \colon\XX\rightarrow \mathbb{T}$. 

Since $U'_{>0}=U'$, $F'_u$ is a polynomial of degree $\leq d$ for all $u\in U'$. By \Cref{ppfacts}(i), there are at most countably many polynomials of degree $\leq d$ up to constants, repeating the Pettis' lemma argument from above, we can find an open subgroup $U''\leq U'$, such that $F'_u$ is a constant for every $u\in U''$. In other words, $(p_u)_{u\in U''} \in \Poly_{\leq d-\weight}^1[U'';Z^1(\Gamma,\XX)]$. Since the filtration is of finite degree, by passing to an open subgroup of $U''$, we may assume that the induced filtration on $U''$ is a sub-filtration of $U'$. Since $U$ is isomorphic to a direct product of finite groups, by shrinking $U''$ if necessary, we may assume that $U$ splits as a direct product of $U''$ and some finite subgroup $W$ (i.e., $U=U''\times W$). Writing $\XX = (\XX_{j-1}(\XX)\times_{\sigma_W} W)\times_{\sigma_{U''}}U''$ where $\sigma_W$ and $\sigma_{U''}$ are the projections of $\sigma$ to $W$ and $U''$ respectively, we can apply \Cref{cocycle-integ} with $(p_u)_{u\in U''}$ and deduce that there exists $q\in \Poly_{\leq d}^1(\Gamma,\XX)$ such that $\partial_u q = p_u$ for all $u\in U''$. By \eqref{pu}, $\rho - q$ is measurable with respect to the factor $\XX=\XX_{j-1}(\XX)\times_{\sigma_W} W$. By \Cref{exactroots-descent} and \Cref{8.11}, it suffices therefore to prove the proposition in the case where the group $U$ is finite.\\

\noindent \textit{The finite case.} Suppose that $U$ is finite. By the structure theorem of finite abelian groups, we can write $U = \prod_{j=1}^N \mathbb{Z}/m_j\mathbb{Z}$ for some integers $m_1,\ldots,m_N$. By \cite[Theorem 1.4]{jst-tdsystems}, all $m_1,\ldots,m_N$ divide $m$. We induct on $N$. If $N=0$, then $U$ is trivial, and the claim follows by the induction on $j$. Fix $N\geq 1$ and assume that the claim holds for all $N'<N$. Let $e$ denote the generator of $\mathbb{Z}/m_N\mathbb{Z}$. Each subgroup of $\mathbb{Z}/m_N\mathbb{Z}$ takes the form $a\mathbb{Z}/m_N\mathbb{Z}$ for some $a$ which divides $m_N$. For each $i\geq 0$, let $a_i$ be minimal such that $a_ie\in U_{>i}$. Since the $U_{>i}$ are decreasing, we must have that $a_1|a_2|\ldots|a_k=m_N$. By the induction hypothesis, we can find polynomials $p_i \colon\Gamma\times \XX\rightarrow \mathbb{T}$ of degree $\leq d-i-1$ (where $p_i=0$ if $i\geq d$) and measurable maps $F_i \colon\XX\rightarrow \mathbb{T}$ such that
\begin{equation}\label{CLi}
    \partial_{a_ie} \rho = p_i + d_\Gamma F_i.
\end{equation}
Applying \Cref{cycliclinearization} with an empty system $\mathcal{R}$ of relations and modifying $p_i$ if necessary (without increasing its degree), we may assume that there exists a coboundary $d_\Gamma F\in d_\Gamma \mathcal{M}(\XX)$ such that
\begin{equation}\label{integrated}
\partial_{a_ie} (\rho + d_\Gamma F) = p_i.
\end{equation} 
By the cocycle identity, $$\partial_{se} (\rho+d_\Gamma F) = p_s$$ where $p_s\coloneqq\sum_{t=0}^{s-1} p_0\circ V_{te}$ is a cocycle on the group $\left<e\right>$. Furthermore, $p_{a_ie} = p_i$ is a polynomial of degree $\leq d-i$. Therefore, we can apply \Cref{cocycle-integ} and find $Q\in \Poly_{\leq d}^1(\Gamma,\XX)$ such that $p_s = \partial_{se} Q$. In other words, $\rho-Q+d_\Gamma F$ is $e$-invariant. By \cite[Proposition 8.11]{btz} and \Cref{exactroots-descent}, we can apply the induction hypothesis on $N$, and deduce that $\rho-Q+d_\Gamma F$ (and therefore $\rho$) are cohomologous to a polynomial cocycle of degree $\leq d$.\\ 

\noindent \textit{The Kronecker case.}
It remains to establish the claim for $j=1$. In this case, $\XX_1$ is isomorphic to a $\Gamma$-rotational system on a compact abelian group $U$. Here the main difference from the previous argument is that the finite factors of $\XX_1$ might no have large spectrum, and hence, we might not be able to apply \Cref{cocycle-integ}. 

Recall that by the induction hypothesis in $d$, we have 
$$\partial_u \rho = p_u + d_\Gamma F_u,$$
where $p_u\in \Poly_{\leq d-1}(\XX_1)$ for all $u\in U$. Using a linearization argument for $u\mapsto F_u$ as right after the proof of \Cref{opensection}, we find an open subgroup $U'\leq U$ such that 
\[
\partial_u \rho = p_u \text{ for all } u\in U'.
\]
In particular, $(p_u)_{u\in U'}\in \Poly_{\leq d-1}^1[U';Z^1(\Gamma,\XX_1)]$, and therefore by separately applying \Cref{poly-integ} for each $\gamma\in \Gamma$, we find a (non-cocycle) polynomial $Q \colon \Gamma\times \XX\rightarrow \T$ of degree $\leq d$ such that $\partial_u Q = p_u$ for all $u\in U'$ and $\gamma\in\Gamma$. By construction, $\rho'\coloneqq \rho - Q$ is $U'$-invariant, but not necessarily a cocycle. We write $U = U'\times W$ for some finite $W$. Then for all $w\in W$, letting $p'_w = p_w - \partial_w Q$, we have
\begin{equation}\label{CLmodU'}
\partial_w \rho' = p'_w + d_\Gamma F_w.
\end{equation}
For all $u\in U'$, we deduce that $\partial_u p'_w = - d_\Gamma \partial_u F_w$. Since $\partial_u p'_w$ is at most of degree $\leq d-2$, we deduce that $\partial_u F_w$ is a polynomial of degree $\leq d-1$. By \Cref{poly-integ}, there is a polynomial $R_w$ of degree $\leq d$ such that $\partial_u F_w = \partial_u R_w$. Replacing $F_w$ with $F_w-R_w$ and $p'_w$ with $p'_w + d_\Gamma R_w$, we may assume that $F_w$ and $p'_w$ are $U'$-invariant without effecting \eqref{CLmodU'}. This achieves a reduction to the finite group case, albeit we have lost the cocycle property on the way. 

The finite compact abelian group $W$ is isomorphic to $\prod_{j=1}^N \mathbb{Z}/m_j\mathbb{Z}$. We induct on $N$. The case $N=0$ is trivially true. Fix $N$ and assume that the claim is true for all $N'<N$. Let $e=e_N$ denote the generator of the last factor of $W$ and let $W' = \prod_{j=1}^{N-1} \mathbb{Z}/m_j\mathbb{Z}$. By the cocycle property applied to \eqref{CLmodU'}, we have 
\[
d_\Gamma\left(\sum_{t=0}^{m_N-1} F_e\circ V_{t\cdot e}\right) =\sum_{t=0}^{m_N-1} p'_e \circ V_{t\cdot e}.
\]
Although $p'_e$ is not necessarily a cocycle in $\gamma$, by the identity $p'_e = p_e + \partial_e Q$ and since $t\cdot e\mapsto \partial_{te} Q$ is a cocycle on $\langle e \rangle$ (and therefore $\sum_{t=0}^{m_N-1} \partial_eQ\circ V_{t\cdot e} = 0$), we obtain   
\begin{equation}\label{finiteF}
d_\Gamma\left(\sum_{t=0}^{m_N-1} F_e\circ V_{t\cdot e}\right) =\sum_{t=0}^{m_N-1} p_e\circ V_{t\cdot e},
\end{equation}
where $p_e$ is a cocycle. 

Let $A$ denote the group generated by the coboundary on the left hand side of \eqref{finiteF} and $(d_\Gamma \Poly_{\leq d}(\XX))^{\left<e\right>}$, and let $B$ denote the group generated by $A$, $p_e$, and all of its $\left<e\right>$-translations. We want to apply \Cref{EDPretraction}. For this, we need to show first that $$(d_\Gamma \Poly_{\leq d}(W))^{\left<e\right>}\leq \Poly^1_{\leq d-1}(\Gamma,W')$$ is a pure filtered subgroup. Indeed, let $q\in \Poly^1_{\leq d-1}(\Gamma,W')$. Since $W'$ is a finite transitive $\Gamma$-system and $\partial_{\gamma'}q_\gamma =\partial_{\gamma}q_{\gamma'}$ for all $\gamma,\gamma'\in \Gamma$ (by the cocycle property), $q$ is a cocycle of type $1$, and thus cohomologous to a constant by \Cref{type}(vi). In other words, 
\begin{equation}\label{type1}
\Poly^1_{\leq d-1}(\Gamma,W') =d_\Gamma \Poly_{\leq d}(W') + \widehat{\Gamma}.
\end{equation}
By assumption, $\mathrm{W}$ is $d$-pure. By \Cref{exactroots-descent}, the $\Gamma$-system $\left<e\right>$ is $d$-pure, in particular $1$-pure. Thus, we find a retraction $r \colon \widehat{\Gamma}\to \widehat{\left<e\right>}$. Let $n\geq 1$ be arbitrary, and let $\mathcal{R}$ be a system of relations in $n$-variables of type at most $d$. 
Let $q=(q_1,\ldots,q_n)$ be in $\Poly^1_{\leq d-1}(\Gamma,W')$ such that for every $(\vec{m};j)\in \mathcal{R}$ there exists $d_\Gamma Q_{(\vec{m};j)} \in (d_\Gamma \Poly_{\leq d}(W))^{\left<e\right>}$ satisfying $$d_\Gamma Q_{(\vec{m};j)} =\vec{m}\cdot q + \Poly^1_{\leq d-j-1}(\Gamma,W').$$ Using \eqref{type1}, we find $P\in \Poly_{\leq d}(W')$ and $c\in \hat{\Gamma}^n$ such that $q=c+d_\Gamma P$, therefore $$d_\Gamma (Q_{(\vec{m};j)} -\vec{m}\cdot P) = \vec{m} \cdot c +  \Poly^1_{\leq d-j-1}(\Gamma,W').$$ Let $\xi\in \widehat{\left<e\right>}^n$ be such that $r(c_i) = \xi_i$ for all $i=1,\ldots,n$, and set $Q=P+\xi$. We conclude that $$d_\Gamma Q_{(\vec{m};j)} = \vec{m}\cdot Q + \Poly^1_{\leq d-j-1}(\Gamma,W'),$$ proving the claim. 

Now we can apply \Cref{EDPretraction} and obtain an $\left<e\right>$-equivariant retraction $\tilde{r} \colon  B\rightarrow A$ of filtered groups. Thus, we can write $\tilde{r}(p_e)=d_\Gamma P_e$ for some polynomial $P_e$, and from \eqref{finiteF}, the $\left<e\right>$-equivariance of $\tilde{r}$, and ergodicity, we have that $$\sum_{t=0}^{M_N-1}F_e\circ V_{t\cdot e}- \sum_{t=0}^{M_N-1}P_e\circ V_{te}$$ is a constant. Modifying $P_e$ by subtracting an $M_N$ root of that constant we may assume that $\sum_{t=0}^{M_N-1}(F_e-P_e)\circ V_{t\cdot e} = 0$. Thus, the map $F_{ze}\coloneqq \sum_{t=0}^{z-1} (F_e-P_e)\circ V_{te}$ is a cocycle in $z$ and so there exists some measurable map $F \colon W'\rightarrow \mathbb{T}$ such that $\partial_e F = F_e - P_e$. Replacing $p'_e$ with $p'_e-d_\Gamma P_e$ and $F_e$ by $F_e-P_e$, we can assume $\partial_e F = F_e$ for $F$ in \eqref{finiteF}. 

Plugging this in \eqref{CLmodU'}, we have 
$$\partial_e (\rho'+ d_\Gamma F) = p'_e.$$
Setting $p''_{se} \coloneqq \sum_{t=0}^{s-1} p_e\circ V_{te}$, the cocycle identity implies that 
$$\partial_{se}(\rho'+d_\Gamma F) = p''_{se}.$$

By \Cref{poly-integ}, applied once for every $\gamma\in \Gamma$, we can find a polynomial  $Q_e \colon \Gamma\times U\rightarrow \T$ such that $p''_{se} = \partial_{se} Q_e$ and in particular, $p'_e = \partial_e Q_e$. We deduce that $\rho''\coloneqq\rho' - Q_e + d_\Gamma F$ is $e$-invariant. By \Cref{8.11} and \Cref{exactroots-descent}, we can apply the induction hypothesis on $N$, and deduce \eqref{CLmodU'} for $\rho''$ and then the claim follows by the induction on $N$. This concludes the proof. 

We close this section by proving a useful corollary of \Cref{stronglyabramov}. We show that purity implies the exactness of a polynomial tower extension.
\begin{corollary}\label{exactisexact}
   Let $d,j,k\geq 1$. Let $\XX=\YY\times_\rho U$ be an abelian extension of ergodic $\Gamma$-systems with $\rho$ a cocycle of type $\leq d$. Suppose that $\YY$ admits an exact polynomial tower of order $\leq k$ and height $j$ such that each $\YY_i$ is $d$-pure and the extensions
\[
\begin{tikzcd}
    \YY_i \arrow[d, Rightarrow, "\rho; U_i"] \\
    \YY_{i-1} 
\end{tikzcd}
\]
   are relatively $d$-pure for all $1<i\leq j$. Then $\rho$ is cohomologous to an exact polynomial cocycle of degree $\leq d-1$. 
\end{corollary}
\begin{proof}
Let $H$ denote the group of all pairs $(\xi,F)\in \widehat{U}\times \mathcal{M}(X,\T)$ such that $\xi\circ\rho_\gamma + d_\gamma F\in \Poly_{\leq i-1}^1(\Gamma,\XX)$ for all $\gamma\in \Gamma$ whenever $\xi\in \widehat{U_{>i}}$ (cf.~\Cref{u-dual}). Thus by \Cref{stronglyabramov} applied to $\YY$, the projection $p:H\rightarrow \widehat{U}$ is onto. The kernel can be identified with $\Poly_{\leq d}(\XX)$ and we have a short exact sequence
\begin{equation}\label{Hses'}
0\rightarrow \Poly_{\leq d}(\XX)\rightarrow H\rightarrow \widehat{U}\rightarrow 0.
\end{equation}
To prove that this short exact sequence splits in the category of locally compact abelian groups, we rely on \Cref{splitsubgroup}. In \cite[Theorem 1.4]{jst-tdsystems} we have established that $U$, and therefore $\widehat{U}$ is of bounded exponent. Let $n\in \mathbb{N}$ and let $(\xi,F)\in H$ be such that  $n\cdot \xi = 0$ and $n\cdot F\in \Poly_{\leq d}(\XX)$. Since $(\xi,F)$ are in $H$ we can write $$\xi\circ\rho = p_\xi + d_\Gamma F$$ for some $p_\xi\in \Poly^1_{\leq d-1}(\Gamma,\XX)$. Since $n\cdot \xi=0$, we deduce that $n\cdot d_\Gamma  F = n\cdot p_\xi$. By \Cref{poly-root}, we see that there exists some $R$ in $\Poly_{\leq d}(\XX)$ such that $n\cdot d_\Gamma F = n\cdot d_\Gamma R$. Modifying $R$ by a constant if necessary we may assume that $n F = n R$. Since $R\in \Poly_{\leq d}(\XX)$, we see that the assumptions in \Cref{splitsubgroup} hold. Thus there exists a cross-section $\xi\mapsto (\xi,F_\xi)$ from $\widehat{U}$ to $H$. In particular, $\xi\mapsto F_\xi$ is a homomorphism, and so by Pontryagin duality we have $F_\xi = \xi\circ F$ for some measurable map $F \colon \XX\rightarrow U$. Let $\rho'\coloneqq \rho - d_\Gamma F$. Since $\rho$ and $\rho'$ are cohomologous, the type filtrations with respect to these two cocycles are equal. Furthermore, for every $\xi\in U_{>i}$ we have that $\xi\circ\rho'$ is a polynomial cocycle of degree $\leq i-1$. Equivalently, 
$\mathrm{Ann}(U_{>i}) \subseteq \{\xi\in\widehat U: \xi\circ\rho'\in \Poly^1_{\le i-1}(\Gamma,X)\}$. The other inclusion follows immediately by \Cref{ppfacts}(v).
\end{proof}

\section{Proof of main theorem} \label{proof}

We prove \Cref{technical} by induction on $j$ (and fix $k$ throughout). The induction base is the following assertion. 

\begin{lemma}\label{j=1}
      Every ergodic $\Gamma$-rotational system has an ergodic $\Gamma$-extension with the structure of a $k$-pure $\Gamma$-rotational system with large spectrum.
\end{lemma}

\begin{proof}
   The action of $\Gamma$ on $\mathrm{Z}$ is defined by a homomorphism $\alpha \colon \Gamma\rightarrow Z$ with dense image. By Pontryagin duality, there is an embedding $\widehat{\alpha} \colon  \widehat{Z}\rightarrow \widehat{\Gamma}$ which we use to identify $\widehat{Z}$ with a countable discrete subgroup of $\widehat{\Gamma}$. Conversely, given any countable discrete subgroup $\Sigma\leq \widehat{\Gamma}$, Pontryagin duality yields a homomorphism $\phi \colon  \Gamma\rightarrow \widehat{\Sigma}$ with dense image which induces an ergodic $\Gamma$-rotational system $\mathrm{Z}_{\Sigma}$ on the compact abelian group $\widehat{\Sigma}$ with the $\Gamma$-action given by $T^\gamma (\xi) = \phi(\gamma)+\xi$. 
   
   We want to construct a countable discrete pure subgroup $\Sigma\leq \widehat{\Gamma}$ that contains both $\widehat{Z}$ and $\tilde{\Gamma}$, where $\tilde{\Gamma}$ was defined in \eqref{gamma-basis-dense}. By Pontryagin duality, it will then follow that $\mathrm{Z}_{\widehat{\Sigma}}$ is an ergodic $\Gamma$-rotational system with large spectrum extending $\mathrm{Z}$. 
   
   We construct $\Sigma$ by an infinite recursion. We start with $\Sigma_0$: The subgroup generated by $\widehat{Z}$ and $\tilde{\Gamma}$. Suppose that we have constructed $\Sigma_n$. We describe how to construct $\Sigma_{n+1}$. Let $\xi\in \Sigma_n$. Check if there exist $m\in \mathbb{N}$ and $f(\xi,m)\in \hat{\Gamma}$ such that $m f(\xi,m) = \xi$ and there does not exist $g \in \Sigma_n$ such that $m g=\xi$. Note that such $f(\xi,m)$ does not need to exist and if it exists it may not be unique. For every $\xi\in \Sigma_n$ if there is such a root $f(\xi,m)$, we choose exactly one  and add it to the set $A$. Define $\Sigma_{n+1}$ to be the subgroup of $\hat{\Gamma}$ generated by $\Sigma_n\cup A$. Finally, define $\Sigma$ to be the union of all $\Sigma_n$. By construction, $\Sigma$ is a discrete countable pure subgroup of $\hat{\Gamma}$. Equivalently, the $\Gamma$-system $\ZZ_\Sigma$ is $1$-pure.

   Next, we will prove that the $\Gamma$-system $\mathrm{Z}_\Sigma$ is also $k$-pure for all $k\geq 2$. Let $\mathcal{R}$ be a finite set of relations in $n$-variables of type at most $k+1$ and let $P=(P_i)_{i=1,\ldots,n}\in \Poly_{\leq k}(\Gamma)$ be such that for almost all $x_0 \in Z_\Sigma$ and all $(\vec{m};i) \in \mathcal{R}$ there exists $Q_{(\vec{m};i)} \in \Poly_{\leq k}(\ZZ_\Sigma)$ such that $$\iota_{x_0}(Q_{(\vec{m};i)}) = \vec{m}\cdot P \mod \Poly_{\leq k-i}(\Gamma).$$

     By \Cref{ppfacts}(iv), for each $(\vec{m};i)\in \mathcal{R}$ there is an open neighborhood of the identity $V_{(\vec{m};i)}\subset Z_\Sigma$ such that $\partial_uQ_{(\vec{m};i)} = \xi_{(m;j)}(u)$ is constant for all $u\in V_{(\vec{m};i)}$. Since $Z_\Sigma$ is totally disconnected,  $V_{(\vec{m};i)}$ contains an open subgroup $V'_{(\vec{m};i)}$. The cocycle identity implies that the restriction of $\xi_{(m;j)}$ to $V'_{(\vec{m};i)}$ is a homomorphism. Since  $V'_{(\vec{m};i)}$ is also totally disconnected, the kernel of $\xi_{(\vec{m};i)}$ is an open subgroup.  
     Since $\mathcal{R}$ is finite, we can let $V$ denote the intersection of all the $\ker \xi_{(\vec{m};i)}$. Since $Z_\Sigma$ is isomorphic to the product of finite cyclic groups, shrinking $V$ if necessary we may assume that $V$ is a cylinder subgroup of $Z_\Sigma$. Equivalently, there is a closed subgroup $W$ of $Z_\Sigma$ such that $Z_\Sigma=V\times W$. The finite factor $W$ is then $1$-pure (say by \Cref{exactroots-descent}) and so the short exact sequence
     $$0\rightarrow \ker \alpha \rightarrow \Gamma \overset{\alpha}\rightarrow W\rightarrow 0$$ splits, where $\alpha \colon \Gamma\rightarrow W$ is the homomorphism inducing the action by rotations on $W$ (note that by ergodicity $\alpha$ must be onto). Let $s \colon W\rightarrow \Gamma$ be a cross-section and define $Q\coloneqq P\circ s$. Since $s$ is a homomorphism, we see that $Q=(Q_1,\ldots,Q_n)$ are polynomials of degree $\leq k$ and furthermore, since all $Q_{(\vec{m};i)}$ are measurable with respect to $W$, we have $Q_{(\vec{m};i)} =\vec{m}\cdot Q \mod \Poly_{\leq k-i}(\mathrm{Z}_\Sigma)$ after lifting $Q$ to $Z_\Sigma$. This proves that $Z_\Sigma$ is $k$-pure.
\end{proof}

To prove \Cref{technical}, it remains to establish the induction step. Let $j\geq 2$ and assume that the claim holds for all $j'<j$. By \Cref{large-spectrum-exist}, the $\Gamma$-system $\XX$ admits an extension with large spectrum. Since an extension of a system with large spectrum has large spectrum, we can assume that $\XX$ has large spectrum.  Write $\XX=\ZZ^{\le j-1}(\XX)\times_\rho U$ as in \Cref{abelext}. By the induction hypothesis, there exists an exact polynomial tower $\YY$ of order $\leq k$ and height $j-1$ extending $\ZZ^{\le j-1}(\XX)$ satisfying all the requirements in \Cref{technical}. Let $\pi \colon \YY\rightarrow \ZZ^{\le j-1}(\XX)$ be the factor map, and let $\rho' \colon \YY\rightarrow U'$ be a minimal cocycle cohomologous to $\rho\circ\pi$ with image $U'$ (cf. \Cref{zimmer}). By \Cref{exactisexact}, we may assume  that $\rho'$ is an exact polynomial cocycle of degree $\leq k-1$.
 
 \begin{lemma}\label{relativepure}
     The $\Gamma$-system $\YY\times_{\rho'} U'$ admits an ergodic extension of the form $\YY_\infty\coloneqq\YY\times_{\rho_\infty} U_\infty$, where $\rho_\infty$ is an exact polynomial cocycle of degree $\leq k-1$ and $\YY_\infty$ is a relatively $k$-pure extension of $\YY$. 
 \end{lemma}
 
 \begin{proof}
     We construct via an induction on $\ell\geq 0$ a sequence of compact abelian groups $U_{\ell}$ and exact polynomial cocycles $\rho_{\ell}\in \Poly_{\leq k-1}^1(\Gamma,\YY,U_{\ell})$ which induce an increasing sequence of exact polynomial towers of the form $\YY_{\ell} \coloneqq \YY\times_{\rho_{\ell}} U_{\ell}$. Set $\YY_0 \coloneqq\YY\times_{\rho'} U'$, and suppose that we have already constructed $U_{\ell}$, $\rho_{\ell}$, and thus $\YY_{\ell}$. 
     
     By \Cref{ppfacts}(i), the group $\Poly_{\leq k}(\YY_{\ell})/\Poly_{\leq 0}(\YY_{\ell})$ is at most countable. From each such coset chose one element and for this element chose a measurable representative which is defined everywhere, and let $\{Q_i\}$ denote the countable collection obtained in this way. Suppose $Q=(Q_{i_1},\ldots,Q_{i_n})$ is a finite subset of these representatives in $(\Poly_{\leq k}(\YY_\ell))^{U_{\ell}}$. Check if there is a finite set of relations $\mathcal{R}$ in $n$-variables of type at most $k+1$ such that for almost every $y_0\in \YY_{\ell}$ there exists $P=(P_1,\ldots,P_n)\in \Poly_{\leq k}(\Gamma)$ such that for all $(\vec{m};i)\in \mathcal{R}$, 
     \begin{equation}\label{abstract-sol}
         \imath_{y_0}(Q) = \vec{m} \cdot P \mod \Poly_{\leq k-i}(\Gamma).
     \end{equation}
      If such a situation occurs, by \Cref{cocycle-root}, there is a tuple of cocycles $(p_1,\ldots,p_n)\in \Poly_{\leq k-1}^1(\Gamma,\YY)$ such that $\vec{m}\cdot (p_1,\ldots,p_n) = d_\Gamma Q$. Now organize all these tuples (while ranging over all $Q=(Q_{i_1},\ldots,Q_{i_n})$ and finite sets of relations $\mathcal{R}$ in $n$-variables of type at most $k+1$ for which such a situation occurs) into $p=(p_i)_{i\in\mathbb{N}}$. By \Cref{zimmer}, there is a minimal type $k$ cocycle $\rho_{\ell+1} \colon  \Gamma\times \YY\rightarrow U_{\ell+1}$ with image in some compact abelian group $U_{\ell+1}$ cohomologous to $p$. By \Cref{exactisexact}, we may assume that $\rho_{\ell+1}$ is an exact polynomial cocycle of degree $\leq k-1$. We then let $\YY_{\ell+1} = \YY\times_{\rho_{\ell+1}}U_{\ell+1}$. 

      We claim that $\YY_{\ell+1}$ extends $\YY_{\ell}$. It clearly extends $\YY$. Let $\xi\in \widehat{U_{\ell}}$. By construction, viewing $\xi$ as a function $\YY_\ell$, there exists some polynomial $Q_i$ from the list above such that $\xi-Q_i$ is constant. Considering the trivial relation $\vec{m} = (1;0)$, we see that $d_\Gamma \xi$ is one of the cocycles in $p$. By construction, there is a measurable map $F \colon \YY\rightarrow \mathbb{T}$ such that 
    $d_\Gamma \xi = \chi\circ\rho_{\ell+1} + d_\Gamma F$, where $\chi\in \widehat{U_{\ell+1}}$. Lifting everything to a generic ergodic joining of $\YY_{\ell+1}$ and $\YY_\ell$ (which in retrospect will be just $\YY_{\ell+1}$), we see that 
    $d_\Gamma (\xi - \chi-F) = 0$ and therefore $\xi = \chi+F+c$ for some $c\in \mathbb{T}$. We deduce that $\xi$ is measurable with respect to $\YY_{\ell+1}$. Since $L^2(\YY_\ell)$ is generated by $L^2(\YY)$ and $\xi\in \hat{U}_\ell$, we get a unitary map $\pi_{\ell+1,\ell} \colon  L^2(\YY_{\ell})\rightarrow L^2(\YY_{\ell+1})$ 
    such that if $f\in L^2(\YY_\ell)$ is measurable with respect to $\YY$, then so is $\pi_{\ell+1,\ell}(f)$ and for every $\xi\in \widehat{U_\ell}$, $\pi_{\ell+1,\ell}(\xi) = \chi_\xi + F_\xi$ for some $\chi_\xi \in \widehat{U_{\ell+1}}$ and a measurable map $F_\xi \colon  \YY\rightarrow \mathbb{T}$. We deduce that 
    $$\chi_{\xi+\xi'} - \chi_\xi - \chi_{\xi'} = F_{\xi+\xi'} - F_\xi - F_{\xi'}$$ for all $\xi,\xi'\in \widehat{U_\ell}$. Observe that the left hand side only depends on $u\in U_{\ell+1}$, while the right hand side only depends on $y\in \YY$, and therefore both sides are constants. Moreover, the left hand side can only be a constant if it is trivial so we conclude that $\xi\mapsto \chi_\xi$ and $\xi\mapsto F_\xi$ are homomorphisms. By Pontryagin duality, there exists a map $\sigma_{\ell+1,\ell}(y,u) = \chi_{\ell+1,\ell}(u) + F_{\ell+1,\ell}(y)$ where $\chi_{\ell+1,\ell} \colon U_{\ell+1}\rightarrow U_\ell$ is a homomorphism and $F_{\ell+1,\ell} \colon  \YY\rightarrow U_\ell$ is a measurable map such that $\pi_{\ell+1,\ell}(\xi) = \xi\circ \sigma_{\ell+1,\ell}$. In other words,  $\pi_{\ell+1,\ell}$ can be realized as a factor map 
    \begin{align*}
    \pi_{\ell+1,\ell}& \colon \YY_{\ell+1}\rightarrow \YY_\ell\\
        \pi_{\ell+1,\ell}&(y,u) = (y,\chi_{\ell+1,\ell}(u)+F_{\ell+1,\ell}(y)).
    \end{align*}
    Furthermore, the maps $\chi_{\ell+1,\ell}$ and the groups $U_\ell$ form an inverse limit system. Let $U_\infty$ denote the inverse limit of $(U_{\ell})_{\ell\geq 0}$ and let $\rho_{\infty}$ denote the inverse limit of the $\rho_\ell$. We obtain a system $\YY_{\infty} \coloneqq \YY\times_{\rho_{\infty}} U_{\infty}$ which is the inverse limit of the $\YY_\ell$.

    As an inverse limit of ergodic exact polynomial towers, $\YY_{\infty}$ is an ergodic exact polynomial tower of order $\leq k$ and height $j$. It is left to show that $\YY_\infty$ is a relatively $k$-pure extension of $\YY$. This follows from our construction. Indeed, let $\mathcal{R}$ be a finite set of relations in $n$-variables of type at most $k$ and let $q=(q_1,\ldots,q_n)\in \Poly^1_{\leq k-1}(\Gamma,\YY_\infty)$ be cocycles such that $$\vec{m}\cdot q = d_\Gamma Q_{(\vec{m};i)} \mod \Poly^1_{\leq k-i-1}(\Gamma,\YY_{\infty})$$ for some $Q_{(\vec{m};i)}\in (\Poly_{\leq k}(\YY_\infty))^{U_\infty}$ for all $(\vec{m};i)\in \mathcal{R}$. 

    Since there are at most finitely many polynomials involved, they are all measurable with respect to $\YY_{\ell}$ for some sufficiently large $\ell$. Then each of the polynomials $Q_{(\vec{m};i)}$ differs by a constant from a polynomial in the list $\{Q_i\}$ defined above, so we can ensure that all such polynomials are from that list by subtracting a constant if necessary. For almost every $y_0\in \YY_\ell$, we can define polynomials $P=(P_1,\ldots,P_n)\in \Poly_{\leq k}(\Gamma)$ by $P(\gamma) \coloneqq q(\gamma,y_0)$ since $d_{\gamma'}P(\gamma)=\partial_{\gamma'} q(\gamma,y_0)+q(\gamma',y_0)$. 
    
    By construction (see \eqref{abstract-sol}), there are cocycles $q'=(q'_1,\ldots,q'_n)\in \Poly_{\leq k-1}^1(\Gamma,\YY)$ such that $\vec{m}\cdot q' = \vec{m}\cdot q \mod \Poly_{\leq k-i-1}^1(\Gamma,\YY_\infty)$ and $q'$ are coboundaries in $\YY_{\ell+1}$. Hence, we can find polynomials $Q'=(Q'_1,\ldots,Q'_n)\in (\Poly_{\leq k}(\YY_\infty))^{U_\infty}$ such that $\vec{m}\cdot d_\Gamma Q' = d_\Gamma Q_{(\vec{m};i)} \mod \Poly^1_{\leq k-i-1}(\YY_\infty)$. This proves relative $k$-purity of $\YY_\infty$.       
 \end{proof}

To complete the proof of \Cref{technical} it remains to show that $\YY_{\infty}$ is $k$-pure. First, we reduce this task to the case where $\YY_{\infty}$ is an extension of $\YY$ by a finite group:   
    \begin{lemma}\label{finite-red}
        If for any decomposition $U_\infty = V\times W$ where $V\leq U_{\infty}$ is open and $W\leq U_{\infty}$ is finite, the system $\YY_{W} \coloneqq \YY\times_{\rho_{W}} W$, where $\rho_W = \rho_{\infty} \mod V$, is $k$-pure, then also $\YY_\infty$ is $k$-pure. 
    \end{lemma}
    \begin{proof}
        Let $\mathcal{R}$ be a finite set of relations in $n$-variables of type at most $k+1$, let $P=(P_1,\ldots,P_n)\in \Poly_{\leq k}(\Gamma)$, and suppose that for almost every $y_0\in \YY_\infty$ and every $(\vec{m};i)\in\mathcal{R}$ there exists $Q_{(\vec{m};i)}\in \Poly_{\leq k}(\YY_\infty)$ such that $$\imath_{y_0}(Q_{(\vec{m};i)}) = \vec{m}\cdot P + \Poly_{\leq k-i}(\Gamma).$$
        Since $\mathcal{R}$ is finite, by the same argument as in the proof of \Cref{j=1}, there is a cylinder subgroup $V\leq U_{\infty}$ such that $\partial_u Q_{(\vec{m};i)}=0$ for all $u\in V$ and all $(\vec{m};i)\in \mathcal{R}$. Write $U_{\infty} = V\times W$. By construction, all $Q_{(\vec{m};i)}$ are measurable with respect to $\YY_{W}$. By $k$-purity of $\YY_W$, there exists $Q=(Q_1,\ldots,Q_n)\in \Poly_{\leq k}(\YY_{W})$ such that $\vec{m}\cdot (\imath_{x_0}(Q)-P)\in \Poly_{\leq k-i}(\Gamma)$. Lifting $Q$ from $\YY_W$ to $\YY_\infty$ (using the factor map) yields the claim.
    \end{proof}

In the remainder of this section, we fix a finite group $W$ as in \Cref{finite-red}, and will establish that $\YY_{W}$ is $k$-pure. We will accomplish this by showing that $\YY_W$ is $d$-pure for all $1\leq d\leq k$ by an induction on $d$. The base case $d=0$ is trivially true. Assume that $\YY_{W}$ is $(d-1)$-pure. 

Let $\mathcal{R}$ be a finite set of relations in $n$-variables of type at most $d+1$, let $P=(P_1,\ldots,P_n)\in \Poly_{\leq d}(\Gamma)$, and suppose that for almost every $y_0\in \YY_W$ and all $(\vec{m};i)\in \mathcal{R}$ there exists $Q_{(\vec{m};i)}\in \Poly_{\leq d}(\YY_W)$ such that
\begin{equation}\label{known}
\imath_{y_0}(Q_{(\vec{m};i)}) = \vec{m}\cdot P +\Poly_{\leq d-i}(\Gamma).
\end{equation}
We need to find $Q=(Q_1,\ldots,Q_n)\in \Poly_{\leq d}(\YY_W)$ such that 
\begin{equation}\label{needed}
    \vec{m}\cdot (P-\imath_{y_0}(Q))\in \Poly_{\leq d-i}(\Gamma).
\end{equation}
Or equivalently, by injectivity of $\imath_{y_0}$,  
\begin{equation}\label{needed2}
\vec{m}\cdot Q - Q_{(\vec{m};i)}\in \Poly_{\leq d-i}(\YY_{W}).
\end{equation}
First, we take advantage of the induction hypothesis to solve the following lower order equation.
\begin{lemma}\label{derivativesolution}
    There exists a polynomial cocycle $r=(r_1,...,r_n)\in \Poly_{\leq d-1}^1(\Gamma,\YY_W)$ such that
    \begin{equation}\label{cocycleversion}
    d_\Gamma Q_{(\vec{m};i)} =\vec{m}\cdot r \mod  \Poly^1_{\leq d-i-1}(\Gamma,\YY_W)
\end{equation}
for all $(\vec{m};i)\in\mathcal{R}$.  
\end{lemma}
\begin{proof}
   Taking the $d_\Gamma$-derivative in \eqref{known}, we get
\begin{equation}
    \vec{m}\cdot (d_\Gamma P - \imath_{y_0}(d_\Gamma Q_{(\vec{m};i)})) \in \Poly_{\leq d-i-1}^1(\Gamma,\Gamma)
\end{equation} 
Let $B$ denote the subgroup of $\Poly_{\leq d-1}^1(\Gamma,\Gamma)$ generated by $\imath^{\oplus\Gamma}_{y_0}(\Poly^1_{\leq d-1}(\Gamma,\YY_W))$ and $d_\Gamma P_t$ for all $t=1,...,n$. By $(d-1)$-purity and \Cref{cocycle-root}, there is a retraction of filtered groups $q\colon B\rightarrow \Poly^1_{\leq d-1}(\Gamma,\YY_{W})$. Let $r=(r_1,...,r_n)\in \Poly^1_{\leq d-1}(\Gamma,\YY_W)$ denote the image of $(d_\Gamma P_1,\ldots,d_\Gamma P_n)$ under $q$. Since $q$ preserves the filtration, we obtain the claim. 
\end{proof}
We can enforce another reduction. By \Cref{ppfacts}(iii), all polynomials of degree $\leq d$ are $W_{>d}$-invariant, where $W_{>d}$ is from the type-filtration on $W$. 
Hence if we find a solution to \eqref{needed} on $\YY_{W/W_{>d}}$, we can lift it to a solution on $\YY_W$, and thus by quotienting out $W_{>d}$, we can assume that the filtration on $W$ is of degree $\leq d$. From now on, we will replace $\YY_{W}$ with $\YY_{W/W_{>d}}$ and $\YY_\infty$ with $\YY\times_{\rho_\infty \mod U_{>d}} U_\infty/U_{\infty,>d}$. 

We want to take advantage of the construction of $\YY_{\infty}$ as an inverse limit of the $\YY_{\ell} = \YY\times_{\rho_{\ell}} U_{\ell}$, since on $\YY_{\ell}$ we can find solutions to all equations involving only polynomials that are linear with respect to $U_{\ell}$. This gives us the following further reduction.
\begin{proposition}\label{linearizeQmj}
It suffices to show that $Q_{(\vec{m};i)}$ must take the form
\begin{equation}\label{wanted}
    Q_{(\vec{m};i)} = \xi_{(\vec{m};i)} + Q'_{(\vec{m};i)}+\vec{m}\cdot R \mod \Poly_{\leq d-i}(\YY_W).
\end{equation}
for some character $\xi_{(\vec{m};i)}\in \widehat{W}$, $W$-invariant polynomial $Q'_{(\vec{m};i)}\in \Poly_{\leq d}(\YY_W)$ (that is, measurable with respect to $\YY$), and $R=(R_1,...,R_n)\in \Poly_{\leq d}(\YY_W)$ for all $(\vec{m};i)\in\mathcal{R}$. 
\end{proposition}
Assuming \Cref{linearizeQmj} for now, let us complete the proof that $\YY_W$ is $k$-pure. 
\begin{proof}[Completing the proof of \Cref{finite-red}]
Recall that we are assuming \eqref{known} and need to find a solution  $Q=(Q_1,...,Q_n)\in \Poly_{\leq d}(\YY_W)$ to \eqref{needed2}. Let $\xi_{(\vec{m};i)}$ be as in \Cref{linearizeQmj}, then we can rewrite \eqref{known} as 
\begin{equation}\label{linearizedeq}
    \imath_{y_0}(\xi_{(\vec{m};i)}+Q'_{(\vec{m};i)}) = \vec{m}\cdot P' + \Poly_{\leq d-i}(\Gamma)
\end{equation}
where $P' = P - \imath_{y_0}(R)$. We first find a solution $Q$ in $\Poly_{\leq d}(\YY_\infty)$. Let $\pi_W \colon \YY_\infty\rightarrow \YY_W$ be the factor map. Then $(\xi_{(\vec{m};i)}+Q'_{(\vec{m};i)})\circ \pi_W \colon \YY_\infty \rightarrow \mathbb{T}$ is a polynomial that is linear on $U_\infty$. By Pontryagin duality and finiteness of $\mathcal{R}$, there is some sufficiently large $\ell$  such that all the characters $\xi_{(\vec{m};i)}$ factor through $U_{\ell}$. By construction of $\YY_{\ell+1}$, we can solve \eqref{linearizedeq} on $\YY_{\ell+1}$, and lift the solution to $\YY_\infty$. In other words, we find $Q=(Q_1,...,Q_n)\in \Poly_{\leq d}(\YY_\infty)$ such that 
$$(\xi_{(\vec{m};i)}+Q'_{(\vec{m};i)})\circ\pi_W-\vec{m}\cdot Q\in \Poly_{\leq d-i}(\YY_\infty).$$
Our next goal is to reduce this solution back to $\YY_W$. Recall that from the construction of $W$ we have $U_\infty=W\times V$ for some open subgroup $V\leq U$. Since $(\xi_{(\vec{m};i)}+Q'_{(\vec{m};i)})\circ\pi_W$ is $V$-invariant and by \Cref{ppfacts}, for all $v\in V$, we have that $$\vec{m}\cdot \partial_v Q \in \Poly_{\leq d-i-\weight(v)}(\YY_\infty).$$ Moreover, since $Q$ is a polynomial of degree $\leq d$, for all $v\in V$, we have that $$\partial_v Q \in \Poly_{\leq d-\weight(v)}(\YY_\infty).$$ By \Cref{poly-integ} (where we view $\YY_\infty$ as an abelian extension of $\YY_W$ by $V$), there is a polynomial $Q'=(Q'_1,\ldots,Q'_n)$ such that $\partial_v (Q-Q')=0$ for all $v\in V$ and $\vec{m}\cdot Q'\in \Poly_{\leq d-i}(\YY_\infty)$. Hence $Q-Q'$ is a solution to \eqref{linearizedeq} in $\YY_W$, and thus $Q-Q'+R$ is a solution to \eqref{needed2}. 
\end{proof}

It remains to prove \Cref{linearizeQmj}.
\begin{proof}[Proof of \Cref{linearizeQmj}]
By \cite[Theorem 1.4]{jst-tdsystems}, $W$ is a finite $m$-torsion group for some $m\geq 1$, and thus isomorphic to $\prod_{i=0}^N \mathbb{Z}/m_i\mathbb{Z}$ for some integers $m_1,\ldots,m_N\in \mathbb{N}$ dividing $m$. 
For $0\leq \ell \leq N$, let $W^{(\ell)} \coloneqq \prod_{i=1}^{\ell} \mathbb{Z}/m_i\mathbb{Z}$. 
By induction, we prove that for all $0\leq \ell \leq N$, there are $\xi_{\ell,(\vec{m};i)}\in \widehat{W^{(\ell)}}$, $Q_{\ell, (\vec{m};i)}\in \Poly_{\leq d}(\YY_{W/W^{(\ell)}})$, and $R_\ell = (R_{\ell,1},\ldots,R_{\ell,n})\in \Poly_{\leq d}(\YY_{W})$ such that 
\begin{equation}\label{eql}
    Q_{(\vec{m};i)} = \xi_{\ell,(\vec{m};i)} + Q_{\ell,(\vec{m};i)}\circ \pi_\ell + \vec{m}\cdot R_\ell \mod \Poly_{\leq d-i}(\YY_{W})
\end{equation}
where $\pi_\ell \colon \YY_{W}\rightarrow \YY_{W/W^{(\ell)}}$ is the factor map and we view $\xi_{\ell,(\vec{m};i)}$ as characters on $W$ by assigning the value $0$ on the coordinates complement to $W^{(\ell)}$. 
Then the $\ell=N$ case will prove the claim in \eqref{wanted}.

If $\ell=0$, take $\xi_{\ell,(\vec{m};i)}=0$, $ Q_{\ell,(\vec{m};i)}=Q_{(\vec{m};i)}$, and $R_\ell=0$. Fix $\ell\geq 1$ and assume that we have already constructed $\xi_{\ell-1,(\vec{m};i)}$, $Q_{\ell-1,(\vec{m};i)}$, and $R_{\ell-1}$.

At this point we take advantage of the cocycles $r\in \Poly_{\leq d-1}^1(\Gamma,\YY_W)$ from \Cref{derivativesolution}. Observe that the solution $r$ to \eqref{cocycleversion} is not unique since we can replace $r$ with $r-q$ whenever $q$ is a polynomial cocycle of degree $\leq d-1$ satisfying $\vec{m}\cdot q \in \Poly_{\leq d-i-1}^1(\Gamma,\YY_W)$. We obtain the following reduction.
\begin{lemma}\label{Wl-1inv}
 Let $r_{\ell-1}\coloneqq r - d_\Gamma R_{\ell-1}$. 
 There exists a solution $r$ to \eqref{cocycleversion}, such that $r_{\ell-1}$ is $W^{(\ell-1)}$-invariant.
\end{lemma}
\begin{proof}
    Since $R_{\ell-1}\in \Poly_{\leq d}(\YY_W)$, we see that $r_{\ell-1}$ is a polynomial cocycle of degree $\leq d-1$.  
    By \Cref{typefiltration} and \Cref{ppfacts}(iii), we have that
$$\partial_w r_{\ell-1}\in \Poly_{\leq d-1-\weight(w)}^1(\Gamma,\YY_W)$$
for all $w\in W^{(\ell-1)}$.
    By a similar argument as in the proof of \Cref{exactroots-descent} (using \Cref{u-dual} and \Cref{8.11}), it follows from \Cref{cocycle-integ} that there is a polynomial cocycle $q_{\ell-1}$ of degree $\leq d-1$ such that $\partial_w q_{\ell-1} = \partial_w r_{\ell-1}$ for all $w\in W^{(\ell-1)}$. 
       
    Since by \eqref{cocycleversion} we have 
    $$
        \vec{m}\cdot r_{\ell-1} = d_\Gamma \left(\xi_{\ell-1,(\vec{m};i)} +Q_{\ell-1,(\vec{m};i)}\circ\pi_{\ell-1}\right) \mod \Poly^1_{\leq d-i-1}(\Gamma,\YY_{W}),
    $$
    it follows from the $W^{(\ell-1)}$-invariance of $d_\Gamma \left(\xi_{\ell-1,(\vec{m};i)} +Q_{\ell-1,(\vec{m};i)}\circ\pi_{\ell-1}\right)$ that  
$$\partial_w \vec{m}\cdot r_{\ell-1} \in \Poly^1_{\leq d-i-1-\weight(w)}(\Gamma,\YY_{W})$$ for all $w\in W^{(\ell-1)}$, and thus 
\[
\vec{m}\cdot q_{\ell-1}\in \Poly^1_{\leq d-i-1}(\Gamma,\YY_W).
\]
We can replace $r$ with $r-q_{\ell-1}$ and obtain the desired result.
\end{proof}
Using \Cref{Wl-1inv}, we can assume that $r_{\ell-1}$ is $W^{(\ell-1)}$-invariant. 
However, we still have by \eqref{cocycleversion} that   
\begin{equation}\label{nroot}
    \vec{m}\cdot r_{\ell-1}  = d_\Gamma(\xi_{\ell-1,(\vec{m};i)} + Q_{\ell-1,(\vec{m};i)}\circ\pi_{\ell-1}) \mod \Poly_{\leq d-i-1}^1(\Gamma,\YY_{W/W^{(\ell-1)}}).
\end{equation}

In order to complete the proof of \Cref{linearizeQmj}, we need to pause and prove a small lemma.
\begin{lemma}
    There exists $f=(f_1,\ldots,f_n) \in \mathcal{M}(\YY_{W/W^{(\ell-1)}})$ (not necessarily polynomials) such that 
    \begin{equation}\label{f}
        Q_{\ell-1,(\vec{m};i)} = \vec{m}\cdot f \mod \Poly_{\leq d-i}(\YY_{W/W^{(\ell-1)}})
    \end{equation}
    for all $(\vec{m},i)\in\mathcal{R}$.
\end{lemma}
\begin{proof}
We begin by deriving a similar equation for $Q_{(\vec{m};i)}$ in place of $Q_{\ell-1,(\vec{m};i)}$. By the induction hypothesis, there exists a filtration-preserving retraction $b \colon \Poly_{\leq d-1}(\Gamma)\rightarrow \Poly_{\leq d-1}(\YY_W)$ for the homomorphism $\iota_{y}\colon \Poly_{\leq d-1}(\YY_W)\to \Poly_{\leq d-1}(\Gamma)$ for almost all $y\in Y_W$. Since $\Poly_{\leq d-1}(\YY_{W})\subseteq \mathcal{M}(\YY_W)$ and $\mathcal{M}(\YY_W)$ is divisible (due to the divisibility of $\T$), we can use the Zorn's lemma to extend $b$ to a retraction $b\colon \mathcal{M}(\Gamma)\rightarrow \mathcal{M}(\YY_W)$ of (non-filtered) abelian groups for the homomorphism $\iota_{y}\colon \mathcal{M}(\YY_W)\to \mathcal{M}(\Gamma)$ for almost all $y\in Y_W$. Applying $b$ to every element in \eqref{known}, we can find a function $f\in \mathcal{M}(\YY_W)^n$ such that
$$Q_{(\vec{m};i)} = \vec{m}\cdot f \mod \Poly_{\leq d-i}(\YY_W).$$

We will now modify $f$ into a solution to \eqref{f}. From the inductive assumption (case $\ell-1$ in \eqref{eql}), we see that we may replace $f$ with $f-R_{\ell-1}$ such that without loss of generality we have
\begin{equation}\label{xi+Q}\xi_{\ell-1,(\vec{m},i)} + Q_{\ell-1,(\vec{m},i)}\circ\pi_{\ell-1} = \vec{m} \cdot f \mod \Poly_{\leq d-i}(\YY_W).
\end{equation}

Now taking the derivative with respect to $d_{W^{(\ell-1)}}$, and since $Q_{\ell-1,(\vec{m},i)}\circ\pi_{\ell-1}$ is $W^{(\ell-1)}$-invariant, we see that
\begin{equation}\label{xidiv}d_{W^{(\ell-1)}} \xi_{\ell-1,(\vec{m},i)} =\vec{m}\cdot d_{W^{(\ell-1)}} f \mod  \Poly_{\leq d-i-\weight}(\YY_W).\end{equation}
Theorem \ref{poly-integ} gives rise to a section 
$$s: d_{W^{(\ell-1)}}\Poly_{\leq d}(\YY_W)\rightarrow \Poly_{\leq d}(\YY_W)$$ mapping $d_{W^{(\ell-1)}} \xi_{\ell-1,(\vec{m},i)}$ to $\xi_{\ell-1,(\vec{m},i)} +c$ for some constant $c\in \mathbb{T}$.\footnote{Indeed, taking $r(y,u)=u-u_0$ and $q(y,u)=u_0-u$ as in the proof of Theorem \ref{poly-integ} gives $\partial_r \xi_{\ell-1,(\vec{m},i)}\circ q = \xi(u) - \xi(u_0)$.}

Consider the short exact sequence
$$0\rightarrow \mathcal{M}(\YY_{W/W^{(\ell-1)}}) \rightarrow \mathcal{M}(\YY_W)\rightarrow d_{W^{(\ell-1)}}\mathcal{M}(\YY_W)\rightarrow 0$$ of abelian groups (without topology or filtration). Since $\mathcal{M}(\YY_{W/W^{(\ell-1)}})$ is divisible, by Zorn's lemma, the section $s$ can be extended to a section
$$s:d_{W^{(\ell-1)}}\mathcal{M}(\YY_W)\rightarrow \mathcal{M}(\YY_W)$$ of abelian groups.
Applying $s$ to $d_{W^{(\ell-1)}}f$ produces a function $f' \colon \YY_W\rightarrow \mathbb{T}$ satisfying that $f-f'$ is $W^{(\ell-1)}$-invariant, yet from \eqref{xidiv} we have $$ \vec{m}\cdot f'= s(\vec{m}\cdot d_{W^{(\ell-1)}} f) = \xi_{\ell-1,(\vec{m};i)} +c\mod \Poly_{\leq d-i}(\YY_W). $$
Using the divisibility of $\mathbb{T}$ we can find some $c'\in \mathbb{T}^n$ such that $\vec{m}\cdot c'=c$. Letting $f''=f-f'-c'$, we see that $f''$ is measurable with respect to the factor $\YY_{W/W^{(\ell-1)}}$, and from \eqref{xi+Q} we deduce that
$$Q_{\ell-1,(\vec{m},i)}\circ\pi_{\ell-1} = \vec{m}\cdot f'' \mod \Poly_{\leq d-i}(\YY_{W/W^{(\ell-1)}}),$$
as required.

    
\end{proof}
We are set to complete the proof of \Cref{linearizeQmj}. The idea now is to work with the measurable solution $f$ and show that it corresponds to a polynomial solution. Let  
\begin{equation}\label{r'l-1}
    r'_{\ell-1}\coloneqq r_{\ell-1} - d_\Gamma f,
\end{equation} where $r_{\ell-1}$ is viewed as a function on $\YY_{W/W^{(\ell-1)}}$ by \Cref{Wl-1inv}. Let $e\coloneqq e_\ell\in W$ be the standard generator of the $\ell^{\text{th}}$-component of $W$. Let $t\in \left<e\right>$. By \Cref{type}(iv), $\partial_t r'_{\ell-1}$ is of type $d-\weight(t)$. Therefore, by  \eqref{r'l-1}, $\partial_t r'_{\ell-1}$ is cohomologous to a polynomial cocycle $p_t$ of degree $\leq d-1-\weight(t)$. Furthermore, since $\xi_{\ell-1,(\vec{m};i)}$ is $e$-invariant, it follows from \eqref{nroot}, \eqref{f}, and \eqref{r'l-1} that $\vec{m}\cdot \partial_{t} r'_{\ell-1}\in \Poly_{\leq d-i-1-\weight(t)}^1(\Gamma,\YY_{W/W^{(\ell-1)}})$. We conclude that $\vec{m}\cdot p_t$ is cohomologous to a polynomial cocycle $q_t'$ of degree $\leq d-i-1-\weight(t)$:  
\begin{equation}\label{mpi}
\vec{m}\cdot p_t = q_t' + d_\Gamma F_t' 
\end{equation}
for some $q_t'\in \Poly_{\leq d-i-1-\weight(t)}^1(\Gamma,\YY_{W/W^{(\ell-1)}}).$
Let $B$ denote the subgroup of $\Poly^1_{\leq d-2}(\Gamma,\YY_{W/W^{(\ell-1)}})$ generated by $d_\Gamma \Poly_{\leq d-1}(\YY_{W/W^{(\ell-1)}})$ and the polynomials in $p_t$ (note that since $\weight(t)\geq 1$, $p_t$ is of degree $\leq d-2$ for all $t\geq 0$). By $(d-1)$-purity and  \Cref{poly-root}, we can find a retraction of filtered groups $v:B\rightarrow d_\Gamma \Poly_{\leq d-1}(\YY_{W/W^{(\ell-1)}})$. Since $p_t$ is cohomologous to $p_t-v(p_t)$, we may assume that $v(p_t)=0$.  Applying $v$ to \eqref{mpi} yields $v(q_t')=d_\Gamma F_t'$ (note that it follows from \eqref{mpi} that $F_t'$ is polynomial of degree at most $d-w(t)$), and therefore $F_t'\in \Poly_{\leq d-i-\weight(t)}(\YY_{W/W^{\ell-1}})$. In particular, we see from \eqref{mpi} that 
\begin{equation}\label{mp}
\vec{m}\cdot p_t\in \Poly^1_{\leq d-i-1-\weight(t)}(\Gamma,\YY_{W/W^{(\ell-1)}}),
\end{equation}
and since we only subtracted a coboundary from $p_t$, we still have
\begin{equation}\label{CLr'}
\partial_{t}r_{\ell-1}'= p_t + d_\Gamma F_t
\end{equation}
for some measurable map $F_t \colon \YY_{W/W^{(\ell-1)}}\rightarrow \T$.
From \eqref{r'l-1} we deduce that
\begin{equation}\label{eQl-1}
    \partial_{t} f - F_t \in \Poly_{\leq d-\weight(t)}(\YY_{W/W^{(l-1)}}).
\end{equation}
We are in the setting of \Cref{cycliclinearization}: 
$$\partial_t r'_{\ell-1} \in \Poly^1_{\leq d-1-\weight(t)}(\Gamma,\YY_{W/W^{(\ell-1)}}) + d_\Gamma \mathcal{M}(\YY_{W/W^{(\ell-1)}}),$$
and by \eqref{mp} and the exact same reasoning we have
$$ \vec{m}\cdot \partial_t r'_{\ell-1} \in \Poly^1_{\leq d-i-1-\weight(t)}(\Gamma,\YY_{W/W^{(\ell-1)}})+d_\Gamma \mathcal{M}(\YY_{W/W^{(\ell-1)}}).$$
Applying \Cref{cycliclinearization}, we find a measurable map $F \colon \YY_{W/W^{(\ell-1)}}\rightarrow \mathbb{T}$ such that \begin{equation}\label{linear}\partial_t(r'_{\ell-1}-d_\Gamma F)\in \Poly^1_{\leq d-1-\weight(t)}(\Gamma,\YY_{W/W^{(\ell-1)}})
\end{equation}
and 
\begin{equation}\label{dF}\partial_t(\vec{m}\cdot r'_{\ell-1}-d_\Gamma \vec{m}\cdot F)\in \Poly^1_{\leq d-i-1-\weight(t)}(\Gamma,\YY_{W/W^{(\ell-1)}}).
\end{equation}

Combining these equations with \eqref{CLr'} and \eqref{mp} gives that  $$\partial_{t} F -F_t \in \Poly_{\leq d-\weight(t)}(\XX)$$ and from \eqref{eQl-1} it follows that 
\begin{equation}\label{R1}
    \partial_t(f-F)\in \Poly_{\leq d-\weight(t)}(\YY_{W/W^{(\ell-1)}})
\end{equation}
Next we wish to multiply this equation by $\vec{m}$, but we first need an observation. By \eqref{nroot}, \eqref{f}, and \eqref{r'l-1},  
$$\vec{m}\cdot r'_{\ell-1} = \vec{m}\cdot r_{\ell-1} -d_\Gamma Q_{\ell-1,(\vec{m};i)}\circ\pi_{\ell-1} = d_\Gamma \xi_{\ell-1,(\vec{m};i)} \mod \Poly^1_{\leq d-i-1}(\Gamma,\YY_{W/W^{(\ell-1)}}).$$
Since $d_\Gamma \xi_{\ell-1,(\vec{m};i)}$ is $e$-invariant, by Proposition \ref{ppfacts}(ii),
$$\partial_t \vec{m}\cdot r'_{\ell-1}\in \Poly^1_{\leq d-i-1-\weight(t)}(\Gamma,\YY_{W/W^{(\ell-1)}}).$$ Combining this with \eqref{dF} we deduce that
$$
    \partial_t  \vec{m} \cdot d_\Gamma F \in \Poly^1_{\leq d-i-1-\weight(t)}(\YY_{W/W^{(\ell-1)}}).
$$
Notice that $d-i-1-\weight(t)$ can be strictly smaller than $-1$ if $t\in W_{>d-i-2}\cap\left<e\right>$, in which case our conventions say that $\partial_t  \vec{m} \cdot d_\Gamma F=0$, and ergodicity implies that $\partial_t  \vec{m} \cdot F=\xi(t)$ is a constant. 
By the cocycle identity, the map $t\mapsto \xi(t)$ is a character on $W_{>d-i-1}\cap\left<e\right>$. We arbitrarily extend $\xi$ to $\left<e\right>$, and obtain that  
\begin{equation}\label{dF2}
      \partial_t  \vec{m} \cdot F  + \xi(t)\in \Poly_{\leq d-i-\weight(t)}(\YY_{W/W^{(\ell-1)}}).
\end{equation}
Thus, multiplying \eqref{R1} by $\vec{m}$, we see that
\begin{equation}\label{mR1}
    \vec{m}\cdot \partial_t(f-F) = \partial_t Q_{\ell-1,(\vec{m};i)}\circ \pi_{\ell-1} - \partial_t \xi \mod  \Poly_{\leq d-i-\weight(t)}(\YY_{W/W^{(\ell-1)}}).
\end{equation}
Applying \Cref{poly-integ} to \eqref{R1}, we find a polynomial $R \colon \YY_{W/W^{(\ell-1)}}\rightarrow \mathbb{T}$ of degree $\leq d$ such that $\partial_{t} R = \partial_{t} (f-F)$. Thus, by \eqref{mR1}, 
$$\partial_t (Q_{\ell-1,(\vec{m};i)}\circ \pi_{\ell-1}) - \partial_t \xi  - \partial_t \vec{m}\cdot R \in \Poly_{\leq d-i-\weight(t)}(\YY_{W/W^{(\ell-1)}}).$$
By another application of \Cref{poly-integ}, there is a polynomial $Q'\in \Poly_{\leq d-i}(\YY_{W/W^{(\ell-1)}})$ such that 
$$
    Q_{\ell,(m;i)} \coloneqq Q_{\ell-1,(\vec{m};i)}\circ \pi_{\ell-1} - \xi  - \vec{m} \cdot R - Q'
$$
is $e$-invariant. Setting $R_{\ell} = R_{\ell-1} + R$, $\xi_{\ell,(\vec{m};i)} = \xi_{\ell-1,(\vec{m};i)}+\xi $ (viewing $R$ and $\xi$ on $\YY_{W}$ by lifting by the factor map $\YY_{W}\rightarrow Y_{W/W^{(\ell-1)}}$), we see that \eqref{eql} holds. This completes the proof.
\end{proof}

\section{Exact polynomials towers are translational systems}\label{sec:tran}

In this section, we prove \Cref{thm:tran}. 

Our first result is that polynomials on exact polynomial towers can be taken to be continuous. 
\begin{lemma}[Continuous representatives for polynomials]\label{polycont}
Let $0\leq j\leq k$, and let $\XX$ be an exact polynomial tower of order $\leq k$ and height $j$. Then every polynomial on $\XX$ is equal $\mu$-almost everywhere to a continuous polynomial.
\end{lemma}
\begin{proof}
We induct on $j$. When $j=0$, $\XX$ is trivial, and every polynomial is equal to a constant, and therefore continuous. Now, let $j\geq 1$ and assume that the claim holds for all exact polynomial towers of height at most $j-1$. Write
\[
\XX=\XX_{j-1}\times_\rho U
\]
where $\XX_{j-1}$ is an exact polynomial tower of order $\leq k$ and height $j-1$, $U$ is a compact abelian group, and $\rho$ is an exact cocycle.

Let $Q\in \Poly(\XX)$ be arbitrary. By \Cref{ppfacts}(iv), there exists an open neighbourhood $U'\subseteq U$ such that for every $u\in U'$ the derivative $\partial_u Q$ is $\mu$-almost everywhere constant; thus for each $u\in U'$ there exist some constant $\xi(u)\in \T$ such that
\[
\partial_u Q=\xi(u)\qquad \mu\text{-a.e.}
\]
Since $\Gamma$ is of bounded exponent, $U$ is totally disconnected (cf. \cite[Theorem 1.4]{jst-tdsystems}), hence $U'$ contains an open subgroup $V\leq U$.

By Fubini's theorem, we may modify $Q$ on a null set, such that the identity $\partial_v Q=\xi(v)$ holds for all $v\in V$ simultaneously. Using the cocycle identity for the $V$-action and the fact that each $\partial_v Q$ is a constant, we obtain
that $\xi \colon  V\to \T$ is a group homomorphism. By automatic continuity (see \cite{Rosendal}), $\xi$ is continuous; in particular $\ker\xi$ is an open subgroup of $V$. Replacing $V$ by $\ker\xi$, we may assume
\[
\partial_v Q=0\qquad \forall v\in V\ \text{a.e.}
\]
Thus $Q$ is $V$-invariant, hence factors through the quotient $U/V$: there exists a measurable
\[
Q'\colon X_{j-1}\times (U/V)\to \T
\]
such that
\[
Q(y,t)=Q'(y,t\bmod V)\qquad \mu\text{-a.e.}
\]
Since $V$ is open in the compact group $U$, the quotient $U/V$ is finite and therefore discrete.

Fix $u_0\in U$ and write $\bar u_0\coloneqq u_0\bmod V\in U/V$. It suffices to show that the section
\[
y\longmapsto Q'(y,\bar u_0)=Q(y,u_0)
\]
is equal a.e.\ to a continuous polynomial on $\XX_{j-1}$.

To see polynomiality of the section, we use \Cref{poly-calc}. Consider the maps
\[
r(y,u)\coloneqq u-u_0,\qquad q(y,u)\coloneqq u_0-u.
\]
Because $\rho$ is an exact polynomial cocycle, for every $i\ge 0$ we have
\[
\rho \bmod U_{>i}\ \in\ \Poly^1_{\le i-1}(\Gamma,\XX_{j-1},U/U_{>i}),
\]
hence the functions $r$ and $q$ satisfy
\[
r \bmod U_{>i}\in \Poly_{\le i}(\XX,U/U_{>i}),\qquad q \bmod U_{>i}\in \Poly_{\le i}(\XX,U/U_{>i}),
\]
since $\partial_\gamma r=\rho_\gamma$ and $\partial_\gamma q=-\rho_\gamma$.

Apply \Cref{poly-calc} with $F\coloneqq Q$, $s=1$, $\ell_1=0$, $r_1\coloneqq r$, and $q$ as above. We obtain that
\[
g(y,u)\coloneqq \partial_{u-u_0}Q\bigl(V_{u_0-u}(y,u)\bigr)
\]
is a polynomial on $\XX$ of degree $\le \deg(Q)$. But by the definition of $g$,
\[
g(y,u)=Q(y,u)-Q(y,u_0).
\]
Hence the function $(y,u)\mapsto Q(y,u_0)$ is a polynomial on $\XX$ (being the difference of two polynomials). Since it is $U$-invariant, it descends to a polynomial on the base $\XX_{j-1}$.

By the induction hypothesis, $y\mapsto Q(y,u_0)$ is equal a.e.\ to a continuous polynomial on $\XX_{j-1}$. As $\bar u_0\in U/V$ was arbitrary and $U/V$ is finite discrete, these continuous representatives glue to a continuous representative of $Q'$ on $\XX_{j-1}\times(U/V)$, and hence yield a continuous representative of $Q$ on $\XX$.
\end{proof}

For the remainder of this section, fix an exact polynomial tower $\XX=(\XX_i)_{i=0}^j$ of order $\leq k$ and hight $j$ with structure groups $U_1,\ldots,U_j$ and exact cocycles $\rho_i\colon \Gamma\times \XX_{i-1}\to U_i$ for $i=1,\ldots,j$.  
We write points of $\XX_i$ as $(x_{i-1},u_i)$ with $x_{i-1}\in X_{i-1}$ and
$u_i\in U_i$.  Identifying $\XX$ with the iterated skew product model, we may
write a typical point of $X=X_j$ as
\[
x=(u_1,\dots,u_j)\in U_1\times\cdots\times U_j. 
\]
For each structure group $U_i$, let $(U_{i,>\ell})_{\ell\ge0}$ denote its type
filtration.  We equip the product 
\[
\mathcal{U}\coloneqq \prod_{i=1}^j U_i
\]
with the product filtration
$\mathcal{U}_{>\ell}\coloneqq \prod_{i=1}^j U_{i,>\ell}$.

\begin{definition}[Exact polynomials along a tower]\label{def:orexactpolyheight}
Let $\XX$ be an exact polynomial tower of order $\le k$ and height $j$.

\begin{itemize}
\item[(i)] For $i=0$, we set $\Poly^0(\XX)\coloneqq U_1$ (viewing $U_1$ as the group of
continuous functions $X_0=\mathrm{pt}\to U_1$).

\item[(ii)] For $1\le i\le j-1$, an \emph{exact polynomial of height $i$} is a continuous
map
\[
P\colon X_i\to U_{i+1}
\]
such that, for every $\ell\ge0$,
\[
P \bmod U_{i+1,>\ell}\ \in\ \Poly_{\le \ell}(X_i,\,U_{i+1}/U_{i+1,>\ell}).
\]
We denote the group of such maps by $\Poly^i(\XX)$.

\item[(iii)] Each $\Poly^i(\XX)$ is equipped with the filtration
$(\Poly^i_{-d}(\XX))_{d\ge0}$ defined by
\[
P\in \Poly^i_{-d}(\XX)\quad \Longleftrightarrow\quad
P\bmod U_{i+1,>\ell}\ \in\ \Poly_{\le \ell-d}(X_i,\,U_{i+1}/U_{i+1,>\ell})
\ \ \forall \ell\ge0.
\]
\end{itemize}
\end{definition}

For each $d\ge0$, define
\begin{equation}\label{eq:Gd-new}
G_d(\XX)\ \coloneqq\ \Poly^0_{-d}(\XX)\rtimes \Poly^1_{-d}(\XX)\rtimes\cdots\rtimes \Poly^{j-1}_{-d}(\XX),
\end{equation}
and set $G(\XX)\coloneqq G_0(\XX)$, i.e.
\begin{equation}\label{eq:G-new}
G(\XX)\ \coloneqq\ \Poly^0(\XX)\rtimes \Poly^1(\XX)\rtimes\cdots\rtimes \Poly^{j-1}(\XX).
\end{equation}
At this stage $\rtimes$ is only a Cartesian product notation. We will show in \Cref{prop:G-group} that $G(X)$ has the structure of a group.

Every element $p=(p_0,\dots,p_{j-1})\in G(\XX)$ defines a transformation of $X$
by the rule
\begin{equation}\label{eq:action-new}
\begin{split}
V_{(p_0,\dots,p_{j-1})}&(u_1,\dots,u_j)
\coloneqq \\
&\bigl(u_1+p_0,\ u_2+p_1(u_1),\ \dots,\ u_j+p_{j-1}(u_1,\dots,u_{j-1})\bigr).
\end{split}
\end{equation}
In this sense, we understand the derivative $\partial_p F$ for $p\in G(\XX)$ and $F\in \mathcal{M}(\XX)$. 
We identify $\mathcal{U}=\prod_{i=1}^j U_i$ with the subset of $G(\XX)$ consisting of
constant functions (and once we show that $G(\XX)$ is a group, $\mathcal{U}$ is also a subgroup).

\begin{lemma}[Multi-level polynomial degree calculation]\label{lem:poly-calc-multi}
Let $\XX$ be an exact polynomial tower of order $\le k$ and height $j$.
Let $d,\ell_1,\dots,\ell_s\ge0$, let $F\in \mathcal{M}(\XX)$ satisfy
\begin{equation}\label{eq:dUF-new}
d_{\mathcal{U}}F\ \in\ \Poly^1_{\le d-\weight}\bigl[\mathcal{U};\mathcal{M}(\XX)\bigr],
\end{equation}
and let $r^1\in G_{\ell_1}(\XX),\dots,r^s\in G_{\ell_s}(\XX)$ and $q\in G(\XX)$.
Then the function
\[
g(x)\ \coloneqq\ \partial_{r^1}\cdots \partial_{r^s}F\bigl(V_q x\bigr)
\]
lies in $\Poly_{\le d-\sum_{t=1}^s \ell_t}(\XX)$.
\end{lemma}

\begin{proof}
The argument is the same as in \Cref{poly-calc}, with the only change being that
derivatives are taken along the multi-level vertical group $\mathcal{U}$ and the induced
polynomial transformations $G(\XX)$.

Set $m\coloneqq d-\sum_{t=1}^s \ell_t$.  If $m\le -1$, then each $r^t\in G_{\ell_t}(\XX)$ is
trivial modulo $\mathcal{U}_{>\ell_t-1}$ at every level, hence takes values in
$\mathcal{U}_{>\ell_t-1}$.  Condition \eqref{eq:dUF-new} implies $F$ is invariant under such
translations to the required order, forcing $g\equiv 0$, as desired.

Assume now $m\ge0$ and induct on $m$.  It suffices to show $\partial_\gamma g\in \Poly_{\le m-1}(\XX)$
for all $\gamma\in\Gamma$.  Fix $\gamma$ and $x\in X$.  One expands $T^\gamma g(x)$ by a discrete
chain rule exactly as in the proof of \Cref{poly-calc}: writing
$T^\gamma q(x)=q(x)+\partial_\gamma q(x)$ and successively using
\[
\partial_{T^\gamma r^t(x)}=\partial_{\partial_\gamma r^t(x)}\,V_{r^t(x)}+\partial_{r^t(x)}
\qquad (t=1,\dots,s),
\]
one expresses $\partial_\gamma g$ as a sum of terms analogous to \eqref{term-1}, \eqref{term-2}, \eqref{term-3}
in \Cref{poly-calc}.  Each such term has the same form as $g$, but with either:
(i) $F$ replaced by $\partial_\gamma F$ (lowering $d$ by $1$), or
(ii) an extra derivative by $r^{s+1}\coloneqq \partial_\gamma q$ (which contributes $\ell_{s+1}=1$), or
(iii) one of the $r^t$ replaced by $\partial_\gamma r^t$ (raising its $\ell_t$ by $1$),
and with harmless translations of the remaining $r^{t'}$ and of $q$.
In all cases the value of $m$ drops by $1$, so the induction hypothesis yields that
each term lies in $\Poly_{\le m-1}(\XX)$, completing the proof.
\end{proof}

Now, we verify that $G(X)$ has the structure of a filtered group. 

\begin{proposition}[Group structure and filtration]\label{prop:G-group}
On the set $G(\XX)$, the operation 
\begin{equation}\label{eq-grouplaw on G}
p\cdot q
=
\bigl(q_0+p_0,\ q_1+p_1\circ V_{q_0},\ \dots,\ q_{j-1}+p_{j-1}\circ V_{(q_0,\dots,q_{j-2})}\bigr).
\end{equation}
defines a group law, and
$\bigl(G_d(\XX)\bigr)_{d\ge0}$ is a degree-$k$ filtration on $G(\XX)$.
\end{proposition}

\begin{proof}
To see that the operation above is well-defined, observe that each coordinate on the right-hand side of \eqref{eq-grouplaw on G} is obtained from the $p_i,q_i$ by addition and composition
with some $V_{(\cdots)}$, and \Cref{lem:poly-calc-multi} ensures it remains an exact polynomial
of the correct height and filtration degree.  Hence $p\cdot q\in G(\XX)$.

One can recursively verify that 
\[
p^{-1}=\bigl(-p_0,\ -p_1\circ V_{-p_0},\ \dots,\ -p_{j-1}\circ V_{(p_0,\dots,p_{j-2})}^{-1}\bigr),
\]
and again \Cref{lem:poly-calc-multi} implies $p^{-1}\in G(\XX)$.

To verify the filtration, computing the commutator
$[p,q]$ coordinatewise, one obtains expressions built from terms of the form
$\partial_{V_{(\cdots)}}(\cdot)$ applied to the various coordinates of $p$ and $q$.
If $p\in G_{\ell_1}(\XX)$ and $q\in G_{\ell_2}(\XX)$, then \Cref{lem:poly-calc-multi}
shows $[p,q]\in G_{\ell_1+\ell_2}(\XX)$.

Finally, since each type filtration on each $U_i$ has degree $\le k$, we have $U_{i,>k}=0$
and hence $G_{k+1}(\XX)=\{e\}$, so the filtration has degree $\le k$.
\end{proof}

\begin{theorem}[Structure theorem for polynomials on exact polynomial towers]\label{thm:poly-str-reindexed}
Let $\XX$ be an exact polynomial tower of order $\le k$ and height $j$.
For every $d\ge0$, the following are equivalent:
\begin{itemize}
\item[(i)] $P\in \Poly_{\le d}(\XX)$.
\item[(ii)] $d_{\mathcal{U}}P\in \Poly^1_{\le d-\weight}\bigl[\mathcal{U};\mathcal{M}(\XX)\bigr]$.
\item[(iii)] $d_{G(\XX)}P\in \Poly^1_{\le d-\weight}\bigl[G(\XX);\mathcal{M}(\XX)\bigr]$.
\end{itemize}
\end{theorem}

\begin{proof}
Since $\mathcal{U}\subseteq G(\XX)$, (iii)$\Rightarrow$(ii) is immediate.

(ii)$\Rightarrow$(i): assume $d_{\mathcal{U}}P\in \Poly^1_{\le d-\weight}\bigl[\mathcal{U};\mathcal{M}(\XX)\bigr]$.
Apply \Cref{lem:poly-calc-multi} with $F=P$, $q(x)=-x$ and with $r^1,\dots,r^{d+1}$
chosen to be the coordinate projections in $\mathcal{U}$. We see that $$g(u_1,...,u_j) \coloneqq  \partial_{u_1}...\partial_{u_j} P(0) = P(u_1,...,u_j)$$ is a polynomial of degree $\leq d$, thus $P\in  \Poly_{\le d}(\XX)$.

(i)$\Rightarrow$(iii): let $P\in\Poly_{\le d}(\XX)$.  Fix $\ell\ge0$ and $r\in G_\ell(\XX)$.
We must show $\partial_r P\in \Poly_{\le d-\ell}(\XX)$.
It suffices (by induction on $d$) to prove that $P$ is $G_{d+1}(\XX)$-invariant.
If $r\in G_{d+1}(\XX)$, then by definition each component of $r$ is trivial modulo
$U_{i,>d+1}$ at the relevant level, hence $r$ takes values in $\mathcal{U}_{>d+1}$.
Therefore it is enough to prove that $P$ is invariant under $\mathcal{U}_{>d+1}$.

We prove invariance under $\mathcal{U}_{>d+1}=\prod_{i=1}^j U_{i,>d+1}$ by downward induction
on the level $i=j,j-1,\dots,1$.  The top level $i=j$ follows from \Cref{ppfacts}(ii)--(iii):
elements of $U_{j,>d+1}$ have weight $>d+1$ and therefore annihilate $\Poly_{\le d}(\XX)$.

Assume $i<j$ and we already know invariance under $U_{i+1,>d+1},\dots,U_{j,>d+1}$.
Consider the factor obtained from $\XX$ by quotienting out these higher subgroups.
Exactness of the cocycles $\rho_{i+1},\dots,\rho_j$ implies that the induced extension at level $i$
remains an exact polynomial extension, and the action of $U_{i,>d+1}$ on this factor is free
(\Cref{ppfacts}(ii)).  Applying \Cref{ppfacts}(iii) on that factor yields $\partial_u P=0$
for all $u\in U_{i,>d+1}$, closing the downward induction.  Hence $P$ is $\mathcal{U}_{>d+1}$-invariant,
and thus $G_{d+1}(\XX)$-invariant, proving (iii).
\end{proof}

We are ready to prove \Cref{thm:tran}. 

\begin{proof}
Let $G(\XX)$ be as in \eqref{eq:G-new}.  
By construction $\mathcal{U}\subseteq G(\XX)$ acts by coordinatewise translations on
$X\cong U_1\times\cdots\times U_j$, hence is transitive.  Therefore $G(\XX)$ is also transitive.

Fix a basepoint $x_0\in X$ and define the stabilizer
\[
\Lambda\coloneqq \{g\in G(\XX): V_g x_0=x_0\}.
\]
The orbit map $\pi \colon G(\XX)\to X$, $\pi(g)=V_g x_0$, is surjective by transitivity and satisfies
$\pi(g)=\pi(g')$ iff $g^{-1}g'\in\Lambda$.  Hence $\pi$ factors through a bijection
\[
\iota: G(\XX)/\Lambda \to X,\qquad \iota(g\Lambda)=V_g x_0.
\]
Because $\mathcal{U}\subseteq G(\XX)$ acts transitively, for every $g\in G(\XX)$ there exists
$u\in\mathcal{U}$ with $V_u x_0=V_g x_0$, hence $u^{-1}g\in\Lambda$ and so $g\in \mathcal{U}\Lambda$.
Thus $G(\XX)=\mathcal{U}\Lambda$, and the quotient map $\mathcal{U}\to G(\XX)/\Lambda$ is surjective.
Since $\mathcal{U}$ is compact, it follows that $G(\XX)/\Lambda$ is compact; in particular $\Lambda$ is
co-compact.

Since by \Cref{polycont} the action $G(\XX)\times X\to X$, $(g,x)\mapsto V_g x$ is continuous, the orbit map $\pi$ is continuous, and $\Lambda=\pi^{-1}(\{x_0\})$
is closed.  Therefore $G(\XX)/\Lambda$ is compact Hausdorff, and $\iota$ is a homeomorphism.

By construction, the measure $\mu$ on $X$ is the iterated product of Haar measures on the structure groups,
and each $V_g$ acts by translations in the fibers at each level; hence $V_g$ preserves $\mu$ for all $g\in G(\XX)$. Thus $\mu_{G(\XX)/\Lambda}\coloneqq \iota^{-1}_*\mu$ defines a $G(\XX)$-invariant regular Borel probability measure on
$G(\XX)/\Lambda$. 

Finally, for each $\gamma\in\Gamma$, define $\phi(\gamma)\in G(\XX)$ by
\[
\phi(\gamma)\coloneqq \bigl(\rho_1(\gamma),\rho_2(\gamma),\dots,\rho_j(\gamma)\bigr).
\]
The cocycle identities for the $\rho_i$ imply that $\phi \colon  \Gamma\to G(\XX)$ is a homomorphism.
\end{proof}

\section{Inverse theorem for the Gowers norms for groups of bounded exponent}\label{proofgowers}

Our goal is to prove \Cref{inversegowers}.  When proving a similar result in \cite[Theorem 1.12]{jst-tdsystems}, our strategy could be briefly summarized as follows:
\begin{itemize}
    \item First assume that the statement failed, so that one could locate a sequence of increasingly bad counterexamples to the theorem, in which a sequence $f_n$ of functions have large Gowers norm but fail to correlate well with polynomials.
    \item By taking a suitable ultraproduct, and invoking the correspondence principle in \cite[Proposition 5.1]{jt21-1}, one can then associate an ergodic (factor of a) ``Loeb translational system'' to this sequence, involving a randomized sequence of shifts, and a limiting function $f$ that has a large Gowers--Host--Kra seminorm.
    \item By applying a structure theorem to this Loeb translational system, show that this limiting function $f$ correlates with an (ergodic-theoretic) polynomial.
    \item By carefully reversing the correspondence principle, and using some stability properties of polynomials, show that many of the original functions $f_n$ then correlate with a polynomial, giving a contradiction.
\end{itemize}

We will adopt a similar strategy here, but a new difficulty arises: our ergodic structure theorem requires one to perform a number of abelian extensions of the original system, thus theoretically losing the Loeb-type structure which is crucial for reversing the correspondence principle at the final step.  Fortunately, we will show (see \Cref{Gabelianextension} below) that such extensions can remain factors of the original Loeb system, so long as a certain non-degeneracy property of the action is satisfied.  Some simple counting arguments will show that this non-degeneracy property will hold for almost all of the random shifts used in the correspondence principle, thus allowing one to complete the arguments.

We turn to the details.  As mentioned above, a key tool will be the correspondence principle established in \cite[Proposition 5.1]{jt21-1}; in this section we will use the notation for ultrafilters and Loeb measures from that paper. We now give an alternative statement of that principle, relating the combinatorial Gowers norm on the ultraproduct of a finite abelian group as a Gowers--Host--Kra seminorm on a randomly generated action on that group.

\begin{definition}[Random action by translations]
    Let $G=(G,+)$ be an abelian group equipped with some $\sigma$-algebra $\mathcal{B}_G$ and probability measure $\mu_G$, let $(\Omega,P)$ be a probability space, and let $\bm{g} \colon  \Omega\rightarrow \Hom(\Z^\N, G)$ be a random homomorphism from $\Z^\N$ to $G$.
    \begin{itemize} 
    \item[(1)] {The \emph{random action by translations} induced by $\bm{g}$ is a $\Gamma_\omega$-action defined by $T_{\omega}^\gamma x = x + \bm{g}(\omega)(\gamma)$, for all $\omega\in \Omega$, where $\Gamma_{\omega} \coloneqq \Z^\N / \operatorname{ker}(\bm{g}(\omega))$, and by abuse of notation we view $\bm{g}(\omega)$ as a homomorphism from $\Gamma_\omega$ to $G$..}
    \item[(2)] {Let $F\subseteq L^\infty(G)$. The \emph{random $\sigma$-algebra generated by $F$ and $\bm{g}$}, $\mathcal{B}_\omega(F,\bm{g}) \leq \mathcal{B}_G$, is the minimal $\sigma$-algebra generated by $T_{\omega}^\gamma F$ for all $\gamma\in \Gamma_{\omega}$.}
    \item[(3)] {Similarly, the \emph{random factor generated by $F$ and $\bm{g}$} is the factor $X_{\omega} = (G,\mathcal{B}_{\omega}(F,\bm{g}), \mu_X, T_{\omega})$ associated with this $\sigma$-algebra and is equipped with the induced $\Gamma$-action and the induced measure (i.e. $\mu_X$ is the restriction of $\mu_G$ to $\mathcal{B}_{\omega}(F,\bm{g})$.}
    \end{itemize}
\end{definition}

\begin{proposition}[Correspondence Principle]\label{correspondence}
Let $\{G_n:n\in\mathbb{N}\}$ be a countable family of finite abelian groups, let $\alpha$ be a non-principal ultrafilter, let $G\coloneqq\prod_{n\rightarrow \alpha}G_n$ and $\Omega \coloneqq \prod_{n\rightarrow \alpha} (\Hom(\Z^\N, G_n))$ be the indicated ultraproducts, equipped both with the Loeb measure construction, let $F$ be an at most countable subset of $L^\infty(G)$, and let $\bm{g} \colon \Omega\rightarrow \Hom(\Z^\N,G)$ be the map 
$$\bm{g}(\lim_{n\rightarrow\alpha} \omega_n) (v) \coloneqq (\lim_{n\rightarrow \alpha} \omega_n(v))$$
for all $\omega_n \in \Hom(\Z^\N, G_n)$ and $v \in \Z^\N$.
Then for Loeb almost every $\omega\in \Omega$ we have $$\|f\|_{U^k(X_{\omega})} = \|f\|_{U^k(G)}$$ for any $f\in \mathcal{F}$, where $\XX_{\omega}$ is the random factor of $G$ generated by $F$ and $\bm{g}$ and $U^k(G)$ is the non-standard Gowers norm (cf. \cite[(4.21)]{jt21-1}).
\end{proposition}

\begin{proof} See  \cite[Proposition 5.1]{jt21-1}. 
\end{proof}

For $\omega \in \Omega$, the first isomorphism theorem gives the short exact sequence
\begin{equation}\label{gom-split}
0 \to \Gamma_\omega \stackrel{\bm{g}(\omega)}{\to} G \to G / \Gamma_\omega \to 0.
\end{equation}
In general, this sequence need not split.  Fortunately, this turns out to be the case when the $G_n$ are uniformly bounded-exponent, provided we make an additional non-degeneracy hypothesis.

\begin{lemma}[Existence of splitting]\label{Gsplit}
In the settings of \Cref{correspondence}, suppose that there is some $m\geq 1$ such that $G_n$ is an $m$-exponent abelian group for all $n\in \mathbb{N}$. Suppose furthermore that we have the non-degeneracy property that for every $m'<m$ which divides $m$, $\lim_{n\rightarrow \alpha} [G_n:m'G_n]$ is unbounded, where $[G_n:m'G_n]$ is the index of $m'G_n$ in $G_n$. Then for Loeb almost every $\omega\in \Omega$, the sequence \eqref{gom-split} admits a Loeb measurable retraction $r_{\omega}\colon G\rightarrow \Gamma_\omega$. In particular, $G = \Gamma_{\omega}\times G/\Gamma_{\omega}$ as measure spaces.
\end{lemma}
\begin{proof}
By \L os's theorem, $G$ is an $m$-exponent group and therefore so is $\Gamma_{\omega}$, thanks to \eqref{gom-split}. We will now take advantage of the non-degeneracy assumption to prove that $\mathrm{ker} \bm{g}(\omega) = m \Z^\N$ for almost all $\omega$. The inclusion $m \Z^\N \leq \operatorname{ker} \bm{g}(\omega)$ is clear, let $\omega = \lim_{n\rightarrow \alpha} \omega_n$. We claim that for every $j\in \mathbb{N}$, and every $m'<m$ dividing $m$, one has
\begin{equation}\label{probability}
P\left(\left\{\omega\in \Omega : m'\bm{g}(\omega)(e_j) \in \left<\bm{g}(\omega)(e_1),\ldots,\bm{g}(\omega)(e_{j-1})\right>\right\}\right) = 0
\end{equation}
where $e_1,e_2,\dots$ is the standard basis for $\Z^\N$, where we use $P$ to denote the Loeb measure on $\Omega$ and $m'<m$. By definition the left hand side of \eqref{probability} is
\begin{align*}
\mathrm{st}\lim_{n\rightarrow \alpha}\mu_n&\left(\left\{\omega\in \Hom(\Z^\N, G_n) : m' \bm{g}_{n}(\omega)(e_j) \in \left<\bm{g}_{n}(\omega)(e_1),\ldots,\bm{g}_{n}(\omega)(e_{j-1})\right>\right\}\right)\\
&\leq \mathrm{st}\lim_{n\rightarrow \alpha} \frac{m^{j-1}}{[G_n:m'G_n]} \\
&= 0.
\end{align*}
where $\mu$ is the product measure on $\Hom(\Z^\N, G_n)$ and the inequality follows from the fact that $\left<\bm{g}_{n}(\omega)(e_1),\ldots,\bm{g}_{n}(\omega)(e_{j-1})\right>$ is of size at most $m^{j-1}$. We deduce that for every $j$, we have that almost surely $\left<\bm{g}(\omega)(e_1),\ldots,\bm{g}(\omega)(e_j)\right> \cong (\mathbb{Z}/m\mathbb{Z})^j$. This implies that the image of $\bm{g}(\omega)$ is isomorphic to $\Z^\N / m\Z^\N$ for $P$-a.e. $\omega\in \Omega$. From \eqref{gom-split}  we deduce that $\Gamma_{\omega} \cong \Z^\N / m\Z^\N$.

Let $k \geq 0$ be a natural number.  By the above analysis, we see that for almost every $\omega \in \Omega$, and for all $n$ in an $\alpha$-large set $A_{\omega,k}$, the homomorphism $\bm{g}(\omega)$ maps $\Z^k/m\Z^k$ (viewed as the subgroup of $\Z^\N/m\Z^\N$ generated by $e_1,\dots,e_k$) injectively into $G_n$, giving the short exact sequence
$$ 0 \to \Z^k/m\Z^k \stackrel{\bm{g}(\omega)}{\to} G_n \to G_n / (\Z^k/m\Z^k) \to 0.$$
We claim that this sequence splits.  By \Cref{splitsubgroup}, it suffices to show that for every natural number $d$ and any $g_n \in \bm{g}(\omega)(\Z^k/m\Z^k) \cap d G_n$, that $g_n = \bm{g}(\omega)(d v)$ for some $v \in \Z^k/m\Z^k$. If we let $m' \coloneqq m / (d,m)$ be the first natural number such that $dm'$ is a multiple of $m$, then $m' g_n = 0$, hence if we write $g_n = \bm{g}(\omega)(w)$ for some $w \in \Z^k/m\Z^k$ then $m' w = 0$. This implies that $w = d v$ for some $v \in \Z^k/m\Z^k$, and the claim follows.

By shrinking the $A_{\omega,k}$ as necessary, we may assume that the $A_{\omega,k}$ are decreasing in $k$ with empty intersection.  For $n$ in $A_{\omega,k} \backslash A_{\omega,k+1}$, let $r_{\omega,n} \colon G_n \to \Gamma$ denote the retraction homomorphism from $G_n$ to $\Z^k/m\Z^k$ (which is a subgroup of $\Gamma$.  Taking ultralimits, we obtain a Loeb-measurable homomorphisms $r_{\omega} \colon G \to \Gamma$, which one verifies to be a retraction for \eqref{gom-split}, as required.
\end{proof}


We can use this splitting to show that ergodic abelian extensions of factors of $G$ can also be viewed as factors of $G$.

\begin{lemma}\label{Gabelianextension}
Let the notation and hypotheses be as in \Cref{Gsplit}. Then for almost every $\omega \in \Omega$, the following statement holds: if $\XX_{\omega}$ is a factor of $(G,T_\omega)$, and $\rho \colon \Gamma_\omega\times \XX_{\omega}\rightarrow U$ be an ergodic cocycle taking values in some compact abelian group $U$, then $\XX_{\omega}\times_\rho U$ is a factor of $G$.
\end{lemma}
\begin{proof}
Let $\pi^\XX_{\omega}\colon  G\rightarrow \XX_{\omega}$ be the factor map and let $r_{\omega}\colon G \to \Gamma_\omega$ be the retraction homomorphism from \Cref{Gsplit}. Define the map $\pi\colon  G\rightarrow \XX_{\omega}\times_{\rho} U$ by $\pi(g) \coloneqq (\pi^\XX(g),u+P(g))$ where $P(g) \coloneqq \rho(r_{\omega}(g),g-\bm{g}(\omega) r_{\omega} g)$.
We claim that $\pi$ is a factor map. First, a direct computation gives that 
$$\partial_\gamma P(g) = \rho(r_{\omega}(g)+\gamma, g-\bm{g}(\omega) r_{\omega} g) - \rho(r_{\omega}(g),g-\bm{g}(\omega) r_{\omega} g) = \rho(\gamma, \pi(g)).$$
Then,
\begin{align*}
\pi (T_{\omega}^{\gamma} g) &= (T_\XX^{\gamma}(\pi^\XX(g)), u+ \partial_\gamma P(g) +P(g))\\&= (T_\XX^{\gamma} (\pi^{\XX}(g)), u + P(g) + \rho(\gamma,\pi(g))) = T_\rho^\gamma \pi(g),
\end{align*}
giving the required intertwining relation.  Now, we claim that the product measure on $\XX_\omega\times_\rho U$ is the push-forward of the Loeb measure on $G$. By Fourier analysis it suffices to show that if $1\not = \xi\in\widehat{U}$ and $f\in L^2(\XX)$, then $\int_G \xi\circ P(g)\cdot f(\pi(g)) d\mu(g)=0$. Indeed, for every $\gamma\in \Gamma$ we have
\begin{align*}\int_G \xi\circ P(g)\cdot f(\pi(g)) d\mu(g) &= \int_G \xi \circ P(T^\gamma_\omega g)\cdot f(T^\gamma \pi(g))d\mu(g) \\&= \int_G \xi \circ \rho(\gamma,\pi(g))\cdot T^\gamma f(\pi(g))\cdot \xi\circ P(g) d\mu(g).
\end{align*}
Thus, if $\Gamma^{(n)}$ is a F\o lner sequence for $\Gamma_\omega$, one has
\begin{equation}\label{limit}
 \int_G \xi\circ P(g)\cdot f(\pi(g)) d\mu(g) = \lim_{n \to \infty} \int_G \E_{\gamma \in \Gamma_n} \xi \circ \rho(\gamma,\pi(g))\cdot T^\gamma f(\pi(g))\cdot \xi\circ P(g) d\mu(g).
 \end{equation}
The function $(x,u) \mapsto \xi(u) f(x)$ has mean zero on the ergodic system $\XX_\omega \times_\rho U$, hence by the ergodic theorem we see that
$$ (x,u) \mapsto \E_{\gamma \in \Gamma_n} \xi(u) \cdot \xi \circ \rho(\gamma,x)\cdot T^\gamma f(x)$$
converges to zero in mean on $\XX_\omega \times_\rho U$, hence
$$ x \mapsto \E_{\gamma \in \Gamma_n} \xi \circ \rho(\gamma,x)\cdot T^\gamma f(x)$$
converges to zero in mean on $\XX_\omega$.  Pulling back to $G$ and using Cauchy--Schwarz, we obtain the vanishing of \eqref{limit} as required.
\end{proof}


We are finally ready to prove \Cref{inversegowers}. We follow the argument in \cite[Section 8]{jst-tdsystems}. The case $k=1$ of \Cref{inversegowers} is well known. We shall assume from now on that $k\geq 2$.  Assume for contradiction that the claim fails for some $m,k\geq 1$ and $\delta>0$. Then for every $n\geq 1$, we can find an $m$-exponent finite abelian group $G_n$ and a $1$-bounded function $f_n\colon G_n\rightarrow \mathbb{C}$ with $\|f\|_{U^{k+1}(G_n)}>\delta$, yet there is no polynomial $P\in \Poly_{\leq k}(G_n)$ such that
\begin{equation}\label{nocorrelation}
    \left|\mathbb{E}_{x\in G_n} f_n(x)e(-P(x))\right|\geq \frac{1}{n}.
\end{equation}
Our next goal is to reduce matters to the non-degenerate case by proving the theorem by induction on $m$. If $m=1$, then all the groups are trivial and all functions are constants and the claim follows. Let $m > 1$ and assume that \Cref{inversegowers} was already established for all smaller values of $m$. Let $m'<m$ which divides $m$. If $[G_n:m'G_n]$ is bounded then by the structure theorem for finite abelian groups we have $G_n = \bigoplus_{t|m}(\mathbb{Z}/t\mathbb{Z})^{a_t}$. Choose such a representation for which $a_m$ is maximal. By assumption, $a_m$ is bounded. We therefore have that $G_n = G'_n\times H_n$ where $H_n =(\mathbb{Z}/t\mathbb{Z})^{a_m} $. Since $a_m$ is maximal, $G'_n$ is a $\tilde{m}$-exponent group for some $\tilde{m}<m$. By Fourier analysis
$f_n(g',h))= \sum_{\xi\in \widehat{H_n}} f_\xi^{(n)}(g')\cdot e(\xi(h))$. We have $\|f_{\xi}^{(n)}\cdot e(\xi(h))\|_{U^{k+1}(G_n)} = \|f_{\xi}^{(n)}\|_{U^{k+1}(G'_n)}$ and so by the triangle inequlity for the Gowers norms, there exists some $\xi\in \widehat{H_n}$ such that $\|f_{\xi}^{(n)}\|_{U^{k+1}(G'_n)}>\delta/(\max_n |H_n|)$. By the induction hypothesis, there exists $\varepsilon>0$, and a polynomial $P_n\in \Poly_{\leq k}(G'_n)$ such that  $|\mathbb{E}_{x\in G_n} f_{\xi}^{(n)}(x) e(-P(x))|>\varepsilon$. We deduce that $f_{\xi}^{(n)}\cdot e(\xi)$ correlates with $P\cdot e(-\xi)$, and since the latter is orthogonal to all $f_{\tau}^{(n)}\cdot e(\tau)$ for all $\tau\not = \xi$, we deduce that 
$$|\mathbb{E}_{(g',h)\in G_n} f^{(n)}(g,h) e(-P(g')+\xi(h))|>\varepsilon$$ which contradicts \eqref{nocorrelation}. 

We may therefore assume that $[G_n:m'G_n]$ is an unbounded sequence for all $m'<m$ dividing $m$. Choose a non-principal ultrafilter $\alpha$ with the property that $\lim_{n\rightarrow \alpha} [G_n:m'G_n]$ is unbounded for all $m'<m$ dividing $m$.
 
Let $G \coloneqq \prod_{n\rightarrow \alpha} G_n$ and let $f\coloneqq\lim_{n\rightarrow \alpha} f_n$. Then we can endow $G$ with the Loeb measure construction and we have $\|f\|_{U^{k+1}(G)}\geq \delta$, where $\|\cdot\|_{U^{k+1}(G)}$ is the non-standard Gowers norm. By \Cref{correspondence} we can find a random action $\omega\mapsto T_{\omega}$ and random factors $\omega\mapsto \XX_{\omega}$ such that $f$ is measurable with respect to $\XX_{\omega}$ for all $\omega\in \Omega$, and for Loeb almost every $\omega\in \Omega$ we have $\|f\|_{U^{k+1}(\XX_{\omega})} = \|f\|_{U^{k+1}(G)}\geq \delta$. Choose some $\omega_0$ satisfying this property and the conclusion of \Cref{Gabelianextension} and let $\XX=\XX_{\omega_0}$.  By \Cref{equiv}, we have an ergodic extension $\YY$ of $\XX$ that is of the form
$\YY = U_1\times_{\rho_1} U_2\times\ldots\times_{\rho_{j-1}} U_j$ for some compact abelian groups $U_1,\ldots,U_j$ and polynomial cocycles $\rho_1,\ldots,\rho_{j-1}$ of degree $\leq k-1$.  An inspection of the proof of \Cref{equiv} reveals that $\YY$ is obtained from $\XX$ by performing a finite sequence of abelian extensions.  Applying \Cref{Gabelianextension} repeatedly, we conclude that $\YY$ is a factor of $G$.

Let $\pi_\XX \colon G \to \XX$, $\pi_\YY \colon G \rightarrow \YY$, $\pi \colon \YY \rightarrow \XX$ be the factor maps.
By \Cref{factorsnorms}, we can find a measurable map $F \colon Z^k(\XX)\rightarrow \C$ such that 
    $$\int_G f(x) \overline{F(\pi_{\XX}(x))} d\mu_G(x) \not = 0$$
where $\pi_{\XX} \colon G\rightarrow \XX$ is the factor map.  By Fourier analysis, $F\circ\pi_{\YY}$ is a linear combination of characters $\xi$ of $U_1\times\ldots\times U_j$. Therefore, we can find some $\xi_1\in \widehat{U_1},\ldots,\xi_j\in \widehat{U_j}$ such that $$\int_G f(x) e\left(-\sum_{i=1}^j \xi_i\circ \pi_{\YY}(x))\right) d\mu_G(x) \not = 0.$$

We think of $\YY$ as a compact abelian group where each one of the structure groups $U_1,\ldots,U_j$ is equipped with the polynomial filtration. This gives rise to a nilspace $\YY=(\YY,C^n(\YY))$. We claim that the map $\pi_{\YY}$ is an \emph{almost polynomial map} in the sense that $\pi(x_\omega)_{\omega\in \{0,1\}^n} \in C^n(\YY)$ for all standard $n$ and $\mu_{\mathrm{HK}^n(G)}$-almost every $(x_\omega)_{\omega\in \{0,1\}^n} \in \mathrm{HK}^n(G)$. where $\mu_{\mathrm{HK}^n(G)}$ denotes the Loeb measure on $\mathrm{HK}^n(G)$. Since $\mathrm{HK}^n(\YY)$ is second countable, it suffices (as in the proof of \cite[Lemma 7.2]{jt21-1} to show that

$$\int_{\mathrm{HK}^n(G)}\prod_{\omega\in \{0,1\}^k} 1_{\pi^{-1}}(U_\omega)(x_\omega) d\mu_{\mathrm{HK}^n(G)}((x_\omega)_{\omega\in \{0,1\}^n})=0$$
whenever $U_\omega$ are open subsets of $\YY$ such that $\prod_{\omega\in \{0,1\}^n} U_\omega$ is disjoint from $C^n(\YY)$. Repeating the proof of \cite[Lemma 7.2]{jt21-1}, the integral above can be
re-expressed as a Gowers–Host–Kra inner product
\begin{equation}\label{innerproduct}
    \left<(1_{U_\omega})_{\omega\in \{0,1\}^n}\right>_{U^n(\YY)}
\end{equation}

For every $1\leq i \leq j$ we have that $\sigma_i \mod (U_i)_t$ is a polynomial of degree $\leq t-2$ for all $t\geq 0$, where $(U_i)_{\bullet}$ is the degree filtration. In particular $d^{t-1}\sigma_i$ takes values in $(U_i)_t$. We deduce that $\left(T^{\sum_{i=1}^n \omega_i h_i} y\right)_{\omega\in \{0,1\}^n} \in C^n(\YY)$ for almost every $y\in \YY$ and all $h_1,\ldots,h_n\in \Gamma$. Therefore, since $\prod_{\omega\in\{0,1\}^n} U_\omega$ avoids $\mathrm{HK}^n(\YY)$ we have
$$\prod_{\omega\in \{0,1\}^n} 1_{U_\omega}(T^{\sum_{i=1}^n \omega_i h_i y}) = 0$$
Taking multiple ergodic averages along
F\o lner sequences we conclude that \eqref{innerproduct} vanishes as claimed. Thus $\pi$ is an almost polynomial. Since $\xi\coloneqq \xi_1+\ldots+\xi_j$ is a polynomial of degree $\leq k$, we see that $\xi\circ\pi$ is also an almost polynomial where $\mathbb{T}$ is equipped with the degree $\leq k$ filtration $\mathcal{D}^k(\mathbb{T})$. By \cite[Lemma 7.3]{jt21-1} we can find an internal nilspace morphism $g:G\rightarrow {}^*\YY$ such that $\xi\circ \pi = \mathrm{st}(g)$. Writing $g=\lim_{n\rightarrow \alpha} g_n$ where $g_n:G_n\rightarrow \YY$ we conclude that
$$\mathrm{st}\lim_{n\rightarrow \alpha} \left|\mathbb{E}_{x\in G_n} f_n(x)\cdot e(-g_n(x))\right| \not = 0.$$
By definition, the map $g_n:G_n\rightarrow \mathbb{T}$ maps $\mathrm{HK}^k(G_n)$ to $\mathrm{HK}^k(\T)$ and is therefore a polynomial of degree $\leq k$. Therefore, we obtain a contradiction for \eqref{nocorrelation} for an $\alpha$-large set of $n\in \mathbb{N}$. This completes the proof. 

\appendix
\section{The Host--Kra factors}\label{host-kra-theory}
In \cite{host2005nonconventional} Host and Kra introduced cubic systems associated with ergodic dynamical systems (see also \cite{hk-book}). We generalize their definition here for arbitrary $\Gamma$-actions.
\begin{definition}[Cubic systems]
Let $\Gamma$ be a countable abelian group and $k\geq 0$. Let $\XX = (X,\X,\mu,T)$ be a $\Gamma$-system. The systems $\XX^{[k]} = (X^{[k]},\X^{[k]},\mu^{[k]},T^{[k]})$ are defined recursively as follows. When $k=0$ we define $\XX^{[0]}=\XX$, assuming that $\XX^{[k]}$ is defined let $X^{[k+1]} = X^{[k]}\times X^{[k]}$, $\X^{[k+1]} = \X^{[k]}\otimes \X^{[k]}$, $T^{[k+1]} = T^{[k]}\times T^{[k]}$ and for every $f,g:X^{[k]}\rightarrow \mathbb{C}$ define $$\int_{X^{[k+1]}} f\otimes g d\mu^{[k+1]} \coloneqq \int_{X^{[k]}}E(f|\mathcal{I}^{[k]})\cdot E(g|\mathcal{I}^{[k]})d\mu^{[k]}$$ where $\mathcal{I}^{[k]}$ is the $\sigma$-algebra of the $T^{[k]}$-invariant functions.
\end{definition}
\begin{proposition}\label{factorsnorms}
    Let $k\geq 1$, let $\Gamma$ be a countable abelian group and let $\XX = (X,\X,\mu,T)$ be an ergodic $\Gamma$-system. There exists a unique (up to isomorphism) factor $\ZZ^{\leq k}(\XX)$ of $\XX$ with the property that for $f\in L^\infty(\XX)$, $$\int_{X^{[k]}}\bigotimes_{\omega\in\{0,1\}^k} \mathcal{C}^{\mathrm{sgn}(\omega)} f d\mu^{[k]}=0\iff E(f|\ZZ^{\le k-1}(\XX))=0$$ where $\mathcal{C}$ is the complex conjugation and $\mathrm{sgn}(\omega)\coloneqq\sum_{i=1}^k \omega_i$.
\end{proposition}
\begin{proof} 
The case when $\Gamma=\mathbb{Z}$ was established by Host and Kra cf. \cite[\S 4]{host2005nonconventional}, \cite[Theorem 7, Chapter 9]{hk-book}. The general case follows the same argument, see e.g. \cite[Appendix A]{btz}.
\end{proof}
\begin{proposition}[Functoriality properties of $\ZZ^{\leq k}$]\label{prop-functoriality}
(cf. \cite[\S 4]{host2005nonconventional}, \cite[Propositions 11, 17, 21, and Theorem 20; Chapter 9]{hk-book}, \cite[Lemma A.22]{btz}) 
Let $k\geq 1$. 
\begin{itemize}
\item[(i)]  A factor  of an ergodic $\Gamma$-system  of order $\leq k$ is of order $\leq k$.
\item[(ii)]  Let $\pi \colon \YY \to \XX$ be a factor map between two ergodic $\Gamma$-systems. Let $\pi_k^\YY \colon \YY\rightarrow \ZZ^{\leq k}(\YY)$, $\pi_k^\XX \colon \XX\rightarrow \ZZ^{\leq k}(\XX)$ be the factor maps onto the $k$-th Host--Kra factors respectively. Then there exists a factor map $\pi_{k} \colon  \ZZ^{\leq k}(\YY)\rightarrow \ZZ^{\leq k}(\XX)$ such that $\pi^\XX_k\circ \pi = \pi_k \circ \pi^\YY_k$. 
\item[(iii)] An inverse limit of ergodic $\Gamma$-systems of order $\leq k$ is an ergodic $\Gamma$-system of order $\leq k$. 
\item[(iv)] If $\XX$ is an inverse limit of ergodic $\Gamma$-systems $\XX_i,i\in I$, then $\ZZ^{\leq k}(\XX)$ is an inverse limit of $\ZZ^{\leq k}(\XX_i),i\in I$.  
\item[(v)] If an ergodic $\Gamma$-system $\XX$ is of order $\leq k$, then it is of order $\leq k'$ for any $k' \geq k$.
\end{itemize}
\end{proposition}
\subsection{Cocycles and extensions
}
Using the cubic systems introduced above we can define a notion of type for cocycles.
\begin{definition}[cocycles of type $k$] \label{type:def} (cf. \cite[Definition 4.1]{btz})
Let $\Gamma$ be a countable abelian group, let $\XX$ be a $\Gamma$-system, let $U=(U,+)$ be a compact metrizable abelian group and let $k\geq 0$.
\begin{itemize}
\item For a measurable $f \colon X\rightarrow U$, we define $\Delta^{[k]}f \colon \XX^{[k]}\rightarrow U$ by
$$\Delta^{[k]}f((x_\omega)_{\omega\in \{0,1\}^k})\coloneqq\sum_{\omega\in \{0,1\}^k} (-1)^{\mathrm{sgn}(\omega)}f(x_\omega)$$ where $\mathrm{sgn}(\omega)\coloneqq\sum_{i=1}^k\omega_i.$
\item A cocycle $\rho \colon \Gamma\times \XX\rightarrow U$ is said to be \emph{of type $\leq k$} if $\Delta^{[k]}\rho$ is a coboundary on $\XX^{[k]}$. Equivalently, there exists a measurable map $F \colon \XX^{[k]}\rightarrow U$ such that $\Delta^{[k]}\rho(\gamma,(x_\omega)_{\omega\in \{0,1\}^k}) = F((T^\gamma x_\omega)_{\omega\in \{0,1\}^k}) - F((x_\omega)_{\omega\in \{0,1\}^{[k]}}$ for all $\gamma\in \Gamma$.
\end{itemize}
We adopt the convention that only the zero cocycle has type $\leq k$ for $k$ negative; in particular, $Z^1_{\leq -1}(\Gamma,\XX,U) = 0$. 
\end{definition}
We have the following properties about type of cocycles.  
\begin{proposition}\label{type}
Let $\XX$ be an ergodic $\Gamma$-system, let $U$ be a metrizable compact abelian group, let $\rho \colon \Gamma\times \XX\rightarrow U$ be cocycle, and let $k\geq 1$.
\begin{itemize}
    \item[(i)] (Moore--Schmidt theorem) $\rho$ is a coboundary if and only if $\xi\circ \rho$ is a coboundary as a cocycle on $\XX$ with values in $\T$ for all Fourier characters $\xi$ in the Pontryagin dual $\hat U$ of $U$. 
    \item[(ii)] $\rho$ is of type $k$ if and only if for every $\xi\in\hat U$, $\xi\circ\rho$ is of type $k$.
    \item[(iii)] If $\XX$ is of order $\leq k$, then $\XX\times_\rho U$ is of order $\leq k$ if and only if $\rho$ is of type $k$.
    \item[(iv)] If $\rho$ is of type $m\geq 0$ and $S\in \Aut(\XX)$ is an automorphism\footnote{An automorphism of a $\Gamma$-system $\XX$ is a measure-preserving isomorphism  $S$ of $(X,\mu)$ commuting (up to almost everywhere equivalence) with the $T$-action.} of the $\Gamma$-system $\XX$ that fixes the $\sigma$-algebra of the Host--Kra factor $\ZZ^{\leq k}(\XX)$, then $\partial_S \rho$ is a cocycle of type $\max(m-k-1,0)$.
    \item[(v)] Suppose $\YY=(Y,\mu,T)$ is an ergodic $\Gamma$-extension of $\XX$ with factor map $\pi \colon Y\rightarrow X$. If $\rho$ is of type $k$, then $\rho\circ\pi$ is of type $k$ as well. 
    \item[(vi)] Suppose that $U=\T$ is the torus and $\rho$ is a cocycle of type $1$. Then $\rho$ is a cohomologous to polynomial cocycle of degree $\leq 0$, i.e., a homomorphism $\Gamma\to \T$. 
    \item[(vii)] If $\XX$ is not ergodic, then $\rho$ is a coboundary if and only if $\rho$ is a coboundary on every ergodic component of $\XX$. 
\end{itemize}
\end{proposition}
\begin{proof}
    For $(i)$, see \cite{moore1980coboundaries}. The claim $(ii)$ follows immediately from $(i)$. The proof of $(iii),(v)$, and $(vii)$ can be found in \cite[Corollary 7.7]{host2005nonconventional}, \cite[Corollary 7.8]{host2005nonconventional}, and  \cite[Lemma 9.1]{host2005nonconventional} for $\Gamma=\Z$ respectively (see also \cite[Propositions 5, 8 and Corollary 9, Chapter 18]{hk-book} and \cite[Lemma 11, 
Chapter 5]{hk-book}) but the arguments extend without difficulty to arbitrary discrete countable abelian groups $\Gamma$. The proof of $(iv)$ was established in \cite[Lemma 5.3]{btz} for automorphisms of a specific type, but the same proof holds for arbitrary automorphisms. Claim $(vi)$ was established in this generality in \cite[Proposition A.10]{jst-tdsystems} (results of this type have appeared in the literature before \cite{moore1980coboundaries}, \cite[Lemma 10.3]{fw}, \cite[Chapter 5, Lemma 13]{hk-book}, \cite[Proposition 2.4(vi)]{jt21-1}). 
\end{proof}
Zimmer \cite{zimmer} studied when a cocycle extension is ergodic. This leads to the definition of image and minimality of cocycles.
\begin{definition}\label{image:def}
    Let $\Gamma$ be a countable abelian group, let $\XX$ be a $\Gamma$-system, and let $\rho \colon \Gamma\times \XX\rightarrow U$ be a cocycle taking values in a compact abelian group $U$. The \emph{image} of $\rho$ is the smallest closed subgroup $U_\rho\leq U$ containing $\rho(\gamma,x)$ for all $\gamma\in \Gamma$ and for almost every $x\in \XX$. The cocycle $\rho$ with image $U_\rho$ is called \emph{minimal} if there is no cocycle $\rho'$ cohomologous to $\rho$ with image $U_{\rho'}\lneqq U_{\rho}$. 
\end{definition}
The following propositions are due to Zimmer \cite[Corollary 3.8]{zimmer}.
\begin{proposition}\label{zimmer}
    Let $\Gamma$ be a countable abelian group, let $\XX$ be an ergodic $\Gamma$-system, and let $\rho \colon \Gamma\times \XX\rightarrow U$ be a cocycle taking values in a compact abelian group $U$. Then 
    \begin{itemize}
    \item[(1)] $\rho$ is cohomologous to a minimal cocycle. 
        \item[(2)] The abelian extension $\XX\times_\rho U$ is ergodic if and only if $\rho$ is minimal with image $U$. 
    \end{itemize}
\end{proposition}
We have the following weak structure result for the Host--Kra factors.
\begin{proposition}\label{abelext}
Let $\Gamma$ be a countable abelian group and let $\XX$ be an ergodic $\Gamma$-system. Then for every $k\geq 1$, the Host--Kra factor $\ZZ^{\leq k}(\XX)$ of order $\leq k$ is (isomorphic to) an abelian group skew-product extension $\ZZ^{\le k-1}(\XX)\times_\rho U$ of the Host--Kra factor $\ZZ^{\le k-1}(\XX)$ of order $\leq k-1$ by a compact metrizable abelian group $U$ and a cocycle $\rho$ of type $k$. 
\end{proposition}

\begin{proof} See \cite[Proposition 6.3]{host2005nonconventional}, \cite[Proposition 3, Chapter 18]{hk-book}, or \cite[Proposition 3.4]{btz}.  The arguments in \cite{host2005nonconventional}, \cite{hk-book} are formulated for $\Z$-systems, but (as observed in \cite{btz}) extend without difficulty to more general $\Gamma$-systems.
\end{proof}
We need the following descent result for cocycles from \cite[Proposition 8.11]{btz}. 
\begin{proposition}[Exact descent]\label{8.11}
    Let $\Gamma$ be a countable abelian group and let $k\geq 1$. Let $\XX$ be an ergodic $\Gamma$-system of order $\leq k$. Let $\mathrm{Y}$ be a factor of $\XX$ with factor map $\pi \colon \XX\rightarrow \mathrm{Y}$. Suppose that $\rho \colon \Gamma\times \mathrm{Y}\rightarrow \mathbb{T}$ is a cocycle. If $\rho\circ \pi$ is of type $k$, then $\rho$ is of type $k$. 
\end{proposition}
The following result  justifies the definition of a type filtration. It was established by Host and Kra for $\mathbb{Z}$-systems but the same proof extends to the action of all countable abelian groups (cf. \cite[Lemma 5.2 and Proposition 7.6]{host2005nonconventional}). 
\begin{proposition}\label{typefiltration}
Let $\Gamma$ be a countable abelian group, let $\XX$ be an ergodic $\Gamma$-system. Let $\mathrm{Y}=\XX\times_\rho U$ be an ergodic abelian extension of $\XX$. Let $i\geq 0$. Then every $u\in U$ induces a measure-preserving map $p_i u \colon\ZZ^{\leq i}(\XX)\rightarrow \ZZ^{\leq i}(\XX)$ such that letting $U_i=\{u\in U : p_i u = \id\}$ the following properties hold:
\begin{itemize}
    \item [(i)] $\ZZ^{\leq i}(\mathrm{Y})$ is isomorphic to an extension of $\ZZ^{\leq i}(\XX)$ by the compact abelian group $U/U_{>i}$ and a cocycle $\rho' \colon \Gamma \times \ZZ^{\leq i}(\XX)\rightarrow U/U_i$ such that $\rho'\circ \pi_i$ is cohomologous to $\rho \mod U_{>i}$, where $\pi_i \colon\XX\rightarrow \ZZ^{\leq i}(\XX)$ is the factor map.
    \item [(ii)] The cocycle $\rho \mod U_{>i}$ is of type $i$.
     \item [(iii)] $U_\bullet = (U_{>i})_{i=0}^\infty$ is the type filtration on $U$. 
\end{itemize}
\end{proposition}
We also have a non-ergodic version of claim $(i)$ from the previous proposition.
\begin{proposition}\label{7.9}
   Let $\Gamma$ be a countable abelian group, let $\XX$ be an ergodic $\Gamma$-system, and let $U=(U,U_\bullet)$ be a filtered group. 
   Let $\rho \colon \Gamma\times X\rightarrow U$ be a (not necessarily ergodic) exact cocycle, that is, $\rho \mod U_{>i}$ is of type $i$ for all $i\geq 0$.
   Then $\rho \mod U_{>i}$ is cohomologous to a cocycle measurable with respect to $\ZZ^i(\XX)$ for all $i\geq 0$.
\end{proposition}
\begin{proof}
    The proof of \cite[Corollary 7.9]{host2005nonconventional} directly extends from $\Z$-systems to $\Gamma$-systems for arbitrary countable abelian groups $\Gamma$.
\end{proof}


\section{Polynomials in groups of bounded exponent }

The following properties of polynomials are well known (see, e.g., \cite[Proposition A.12]{jst-tdsystems} and the references mentioned therein).

\begin{proposition}[Properties of polynomials]\label{ppfacts}
 Let $\XX$ be an ergodic $\Gamma$-system.
	\begin{itemize}
	\item[(i)] Let $k\geq 1$, and suppose that $P,Q\in \Poly_{\leq k}(\XX)$ and $P-Q$ is non-constant. Then $$\|e(P)-e(Q)\|_{L^2(\XX)} \geq \sqrt{2}/2^{k-2}$$ where $e(y)=e^{2\pi i y}$.  In particular, there are only countably many elements of $\Poly_{\leq k}(\XX)$ up to constants.
    \item[(ii)]  For any $m \geq 0$, a polynomial in $\Poly_{\leq m}(\XX)$ is measurable in $\ZZ^m(\XX)$.
	\item[(iii)] Let $m\geq 0, k\geq 1$ and let $P\in \Poly_{\leq m}(X)$. If $t\in \Aut(\XX)$ fixes the $\sigma$-algebra $\ZZ^{\leq k}(\XX)$, then $\partial_{t} P$ is a polynomial of degree $\leq m-k-1$. 
	\item[(iv)] Let $m\geq 0$, let $P\in \Poly_{\leq m}(X)$, and let $K\leq \Aut(\XX)$ be a compact subgroup. Then there is an open neighborhood of the identity $V\leq K$ such that $\partial_{u} P$ is a constant for every $u\in V$. 
	\item[(v)] Let $m\geq 0$ and let $f \in \mathcal{M}(X,\T)$. Then $f$ is a polynomial of degree at most $m-1$ if and only if $\Delta^{[m]}f(x)\equiv 0$ for $\mu^{[m]}$-almost every $x\in \XX^{[m]}$. 
    \item[(vi)]  If $\XX$ is an inverse limit of $(\XX_\alpha)_{\alpha \in A}$ is a directed set of ergodic $\Gamma$-systems (with compatible factor maps) and $k \geq 1$, then $\Poly_{\leq k}(\XX)$ is the union of the $\Poly_{\leq k}(\XX_\alpha)$ (where we embed the latter groups in the former in the obvious fashion).
	\end{itemize}
\end{proposition}

\begin{lemma}[On multiplication by $m$]\label{mtimes}
Let $k,m\geq 1$ be integers and let $\Gamma$ be an $m$-exponent group. If $P\in \Poly_{\leq k}(\Gamma)$, then $m\cdot P\in \Poly_{\leq k-1}(\Gamma)$. In particular, $\Poly_{\leq k}(\Gamma)/\Poly_{\leq 0}(\Gamma)$ is an $m^k$-exponent group.
\end{lemma}

\begin{proof}  For $k=1$, we can identify $\Poly_{\leq 1}(\Gamma)/\T$ with the Pontryagin dual $\hat \Gamma$ of $\Gamma$, which will be of exponent $m$ by Pr\"ufer's first theorem.  Now suppose $k>1$ and $\Poly_{\leq k-1}(\Gamma)/\T$ has already been shown to have exponent $m^{k-1}$.  If $P \in \Poly_{\leq k}(\Gamma)$, we conclude that $m^{k-1} \partial_\gamma P$ is constant for every $\gamma \in \Gamma$, hence $m^{k-1} P \in \Poly_{\leq 1}(\Gamma)$, hence $m^k P$ is constant, giving the claim.

To then obtain the same conclusion for $\Poly_{\leq k}(\XX)/\T$, use a sampling map $\iota_{x_0}$ defined in \eqref{sampling} for $x_0 \in \XX$ chosen outside of a null set (to make it an injective $\Gamma$-equivariant homomorphism).
\end{proof}

\section{The category of filtered locally compact abelian groups}\label{filteredcategory}

The category of filtered abelian groups plays an important role in our analysis.

\begin{definition}
Let $k \ge 0$.
\begin{itemize}
    \item[(i)] A \emph{$k$-filtered locally compact abelian group} is a pair $\AA = (A, A_\bullet)$ where $A$ is a locally compact abelian group and $A_\bullet = (A_i)_{i \ge 0}$ is a filtration of degree $\leq k$, that is,
    \[
    A = A_0 \ge A_1 \ge \cdots \ge A_{k+1} = \{0_A\},
    \]
    consisting of closed subgroups.  In particular, we can identify locally compact abelian groups with $1$-filtered locally compact abelian groups in the obvious fashion.

    \item[(ii)] We say that a $k$-filtered group $\AA$ is a \emph{subgroup} of a $k$-filtered group $\BB$ if $A_i$ is a closed subgroup of $B_i$ for all $i \ge 0$ and
    \[
    A_i = A \cap B_i \quad \text{for all } i.
    \]
    In this case we define the \emph{quotient} $\CC = \BB / \AA$ as the abelian group $B / A$ equipped with the filtration $(B_i / A_i)_{i \ge 0}$ and the quotient topology.

    \item[(iii)] The \emph{direct sum} $\AA \oplus \BB$ of two filtered groups $\AA = (A, A_\bullet)$ and $\BB = (B, B_\bullet)$ is the group $A \oplus B$ equipped with the filtration $(A_i \oplus B_i)_{i \ge 0}$ and the product topology.

    \item[(iv)] A map $\phi \colon \AA \to \BB$ between two $k$-filtered groups is called a \emph{morphism} if $\phi$ is a continuous homomorphism and $\phi(A_i) \subseteq B_i$ for all $i \ge 0$.  We write $\phi = (\phi_i)_{i \ge 0}$, where $\phi_i \colon A_i \to B_i$ is the restriction.  The morphism $\phi$ is \emph{injective} (resp. \emph{surjective}) if each $\phi_i$ is injective (resp. surjective).  We say that $\AA$ and $\BB$ are \emph{isomorphic} if there exists a bijective morphism $\phi \colon \AA \to \BB$.

    \item[(v)] A \emph{short exact sequence} of $k$-filtered locally compact abelian groups is a sequence
    \begin{equation}\label{ABC}
        0 \longrightarrow \AA \overset{\imath}{\longrightarrow} \BB \overset{\pi}{\longrightarrow} \CC \longrightarrow 0
    \end{equation}
    where $\imath$ is an injective morphism with closed image, $\pi$ is a surjective open morphism, and for each $i \ge 0$ the image of $\imath_i$ equals the kernel of $\pi_i$.  The sequence \eqref{ABC} \emph{splits} if there exists an isomorphism $\varphi \colon \BB \to \AA \oplus \CC$ such that $\varphi \circ \imath(a) = (a, 0_C)$ and $\pi \circ \varphi^{-1}(a, c) = c$ for all $a \in A$, $c \in C$.  

    We say that $\imath$ \emph{admits a retraction} if there exists a surjective morphism $r \colon \BB \to \AA$ with $r \circ \imath = \mathrm{Id}_\AA$, and that $\pi$ \emph{admits a cross-section} if there exists an injective morphism $s \colon \CC \to \BB$ with $\pi \circ s = \mathrm{Id}_\CC$.  The filtration is compatible with exactness at every level, i.e.
    \[
    0 \longrightarrow \AA_i \overset{\imath_i}{\longrightarrow} \BB_i \overset{\pi_i}{\longrightarrow} \CC_i \longrightarrow 0
    \]
    is a short exact sequence in the category of locally compact abelian groups.
\end{itemize}
\end{definition}

\begin{example}[Polynomial filtration]\label{polynomialfiltration}
Let $\XX$ be an ergodic $\Gamma$-system and $k \ge 1$.  The group $A = \Poly_{\le k}(\XX)$ with the filtration $A_i = \Poly_{\le k-i}(\XX)$ is a $k$-filtered group, called the \emph{polynomial filtration}.  If $\YY$ is an extension of $\XX$ with factor map $\pi \colon \YY \to \XX$ and $B = \Poly_{\le k}(\YY)$ carries its polynomial filtration, then $\pi$ induces an injective morphism with closed range
\[
\pi^* \colon \AA \to \BB, \quad P \mapsto P \circ \pi.
\]
\end{example}

\begin{lemma}[Equivalent conditions for splitting]\label{equivalenceses}
Let
\[
0 \to \AA \overset{\imath}{\longrightarrow} \BB \overset{\pi}{\longrightarrow} \CC \to 0
\]
be a short exact sequence of $k$-filtered locally compact abelian groups.  The following are equivalent:
\begin{itemize}
    \item[(i)] The sequence \eqref{ABC} splits.
    \item[(ii)] $\pi$ admits a cross-section.
    \item[(iii)] $\imath$ admits a retraction.
\end{itemize}
\end{lemma}

\begin{proof}
(i)$\Rightarrow$(ii),(iii): let $\varphi \colon \BB \to \AA \oplus \CC$ be an isomorphism such that $\varphi \circ \imath(a) = (a, 0)$ and $\pi \circ \varphi^{-1}(a, c) = c$.  Then $\pi$ admits the cross-section $s(c) = \varphi^{-1}(0, c)$ and $\imath$ admits the retraction $r(b)$ given by the first coordinate of $\varphi(b)$.

(ii)$\Rightarrow$(iii): if $s \colon \CC \to \BB$ is a cross-section for $\pi$, define $r(b) = b - s(\pi(b))$.  Then $\pi(r(b)) = 0$, so $r(b) \in A$ and $r \circ \imath = \mathrm{Id}_\AA$.

(iii)$\Rightarrow$(i): if $r \colon \BB \to \AA$ is a retraction, the map $\varphi \colon \BB \to \AA \oplus \CC$ given by $\varphi(b) = (r(b), \pi(b))$ is an isomorphism of $k$-filtered locally compact abelian groups.
\end{proof}

The existence of retractions and cross-sections preserving the filtration is particularly important in our analysis.  To understand when such splittings occur, we recall that an object $\AA$ in a category is called \emph{injective}, and an object $\CC$ is called \emph{projective}, if all exact sequences of the form~\eqref{ABC} split independently of the other variables.  For example, in the category of locally compact abelian groups, $\mathbb{T}$ and $\mathbb{R}$ are injective, and by Pontryagin duality $\mathbb{Z}$ and $\mathbb{R}$ are projective.  In the category of discrete abelian groups, divisible groups are injective.  Many of the groups we study are neither injective nor projective, yet still yield splittings.  For instance, if $A,B,C$ are discrete abelian groups of bounded exponent (hence $1$-filtered), the short exact sequence~\eqref{ABC} splits if and only if $\imath(A)$ is \emph{pure} in $B$: for every $n \in \mathbb{N}$ and $a \in A$, whenever $n x = \imath(a)$ has a solution $x \in B$, it also has a solution $x \in \imath(A)$.

\begin{lemma}[Splitting criterion]\label{splitsubgroup}
Let
\begin{equation}\label{ses}
0 \to A \stackrel{\iota}{\longrightarrow} B \stackrel{\pi}{\longrightarrow} C \to 0
\end{equation}
be a short exact sequence of (1-filtered) abelian groups with $C$ discrete of bounded exponent, and suppose that $\iota(A) \cap nB = n\,\iota(A)$ for all $n \in \mathbb{N}$.  Then~\eqref{ses} splits.
\end{lemma}

\begin{proof}
By Pr\"ufer's first theorem we may write $C$ as a direct sum of cyclic groups $\Z/m_i\Z$.  It suffices to find a section for each summand.  For each generator $c_i$ of such a cyclic group, pick $b_i' \in \pi^{-1}(c_i)$; then $m_i b_i' \in \iota(A)$, say $m_i b_i' = \iota(a_i)$.  By purity, there exists $a_i' \in A$ with $m_i a_i' = a_i$.  Setting $b_i = b_i' - \iota(a_i')$ yields $m_i b_i = 0$ and $\pi(b_i) = c_i$, giving the desired section.
\end{proof}

\subsection{Systems of relations and purity}

To extend Lemma~\ref{splitsubgroup} to the filtered setting, one must consider linear relations that hold only modulo lower levels of the filtration.

\begin{definition}[Systems of relations]
Let $\alpha$ be a cardinal\footnote{Since all locally compact abelian groups relevant here are second countable, it suffices to take $\alpha$ countable.}, and define $\Z^\alpha \coloneqq  \bigoplus_{i \in \alpha} \Z$.  For $k \ge 0$, a \emph{system of relations in $\alpha$-variables} (of type at most $k+1$) is a subset $\mathcal{R} \subseteq \Z^\alpha \times \{1, \ldots, k+1\}$.  Given a $k$-filtered locally compact abelian group $\AA = (A, A_\bullet)$, a family $a = (a_i)_{i \in \alpha}$ in $A$ \emph{satisfies} $\mathcal{R}$ if $\vec m \cdot a \in A_j$ whenever $(\vec m; j) \in \mathcal{R}$, where $\vec m \cdot a \coloneqq  \sum_{i \in \alpha} m_i a_i$ (a finite sum).  For any family $a$, let $\mathcal{R}_a$ denote the set of all relations of type at most $k+1$ satisfied by $a$.
\end{definition}

\begin{definition}[Purity]\label{pure:def}
Let $\alpha$ be a cardinal, and let $\AA = (A, A_\bullet)$ be a subgroup of a $k$-filtered locally compact abelian group $\BB = (B, B_\bullet)$.  We say that $\AA$ is \emph{pure up to length $\alpha$} in $\BB$ if for every system $\mathcal{R}$ of relations of type at most $k+1$ in $\alpha$-variables the following holds: for every family $b = (b_i)_{i \in \alpha}$ in $B$ satisfying
\[
\vec m \cdot b \bmod B_j \in \mathrm{im}(A/A_j \to B/B_j)
\]
for all $(\vec m; j) \in \mathcal{R}$, there exists a family $a = (a_i)_{i \in \alpha}$ in $A$ such that $\vec m \cdot (b - a) \in B_j$ for all $(\vec m; j) \in \mathcal{R}$.  We say that $\AA$ is \emph{pure} in $\BB$ if it is pure up to length $\alpha$ for every finite~$\alpha$.
\end{definition}

\begin{remark}
    The condition $\vec m \cdot b \bmod B_j \in \mathrm{im}(A/A_j \to B/B_j)$ means equivalently that there exists $a \in A$ with $\vec m \cdot b - a \in B_j$.
\end{remark}

\begin{theorem}[Splitting of pure subgroups]\label{puresplit}
Let $\AA = (A, A_\bullet)$ be a closed subgroup of a $k$-filtered locally compact abelian group $\BB = (B, B_\bullet)$ that is pure up to length $\alpha$, and let $\CC = (C, C_\bullet)$ be the discrete quotient admitting a generating set of cardinality at most~$\alpha$.  Then the short exact sequence~\eqref{ABC} splits.
\end{theorem}

\begin{proof}
As in Lemma~\ref{splitsubgroup}, it suffices to construct a morphism $s \colon \CC \to \BB$ with $\pi \circ s = \mathrm{Id}_\CC$.  Since $C$ carries the discrete topology, $s$ is automatically continuous.

Let $e = (e_i)_{i \in \alpha}$ be a generating family for $C$, and let $\mathcal{R}_e$ be the system of relations satisfied by $e$.  Choose $b = (b_i)_{i \in \alpha}$ in $B$ such that $\pi(b_i) = e_i$.  For each $(\vec m; j) \in \mathcal{R}_e$ we have $\vec m \cdot e \in C_j$, hence $\vec m \cdot b \bmod B_j \in \mathrm{im}(A/A_j \to B/B_j)$.  By purity, there exists $a = (a_i)_{i \in \alpha}$ in $A$ with $\vec m \cdot (b - a) \in B_j$ for all $(\vec m; j) \in \mathcal{R}_e$.

Set $b_i' \coloneqq  b_i - a_i$ and define
\[
s(\vec{\ell} \cdot e) \coloneqq  \vec{\ell} \cdot b', \qquad \vec{\ell} \in \Z^\alpha.
\]
This is well-defined: if $\vec{\ell} \cdot e = 0$ in $C$, then $(\vec{\ell}; k+1) \in \mathcal{R}_e$, hence $\vec{\ell} \cdot b' \in B_{k+1} = \{0\}$.  If $\vec{\ell} \cdot e \in C_j$, then $(\vec{\ell}; j) \in \mathcal{R}_e$ and $\vec{\ell} \cdot b' \in B_j$.  Finally, $\pi(b_i') = e_i$ implies $\pi \circ s = \mathrm{Id}_\CC$.
\end{proof}

\begin{definition}[Finite splitting]\label{finite-split}
Let $k \ge 1$, and let $\AA, \BB, \CC$ be discrete $k$-filtered abelian groups forming a short exact sequence~\eqref{ABC}.  We say that the sequence \emph{finitely splits} if for every subgroup $\imath(\AA) \le \BB' \le \BB$ with $\imath(\AA)$ of finite index in $\BB'$, the induced sequence
\[
0 \longrightarrow \AA \overset{\imath}{\longrightarrow} \BB' \overset{\pi|_{\BB'}}{\longrightarrow} \pi(\BB') \longrightarrow 0
\]
of $k$-filtered abelian groups splits.
\end{definition}

\bibliographystyle{abbrv}
\bibliography{bibliography}

\end{document}